\documentclass{amsart}
\usepackage{marktext}
\usepackage{gimac}

\usepackage{algpseudocode, algorithm}
\usepackage{mdframed}
\usepackage{algorithmicx}
\usepackage{algpseudocode}

\newmdenv[
  linewidth=0.5pt,
  skipabove=10pt,
  skipbelow=10pt,
  backgroundcolor=gray!3,
  linecolor=gray!60,
  innertopmargin=8pt,
  innerbottommargin=8pt,
  roundcorner=4pt,
  nobreak=false  
]{algbox}

\newcommand{\tr}{\mathrm{Tr}}
\newcommand{\Bias}{\mathrm{\tilde Bias}}

\DeclareMathOperator{\dist}{dist}

\DeclarePairedDelimiter{\ip}{\langle}{\rangle}

\newtheorem{assumption}[theorem]{Assumption}
\begin{document}

\title[Rapid convergence of Langevin AIS]
{Rapid convergence of Langevin AIS for multimodal distributions on~$\R^d$}

\author[Son]{Seungjae Son}
\address{%
  Department of Mathematical Sciences, Carnegie Mellon University, Pittsburgh, PA 15213.
}
\email{seungjas@andrew.cmu.edu}

\begin{abstract}
  We study annealed importance sampling (AIS) driven by the overdamped Langevin
  diffusion, targeting a multimodal Gibbs measure~$\pi_\epsilon \propto e^{-U/\epsilon}$
  on~$\R^d$ at low temperature~$\epsilon$. For a potential~$U$ with nondegenerate
  critical points, finitely many wells of comparable mass, and at least quadratic
  growth at infinity, we show that AIS with a harmonically spaced temperature schedule
  produces a weighted empirical measure that approximates~$\pi_\epsilon$ to accuracy~$\delta$
  at a total cost of order~$\epsilon^{-2}\delta^{-2}$, up to logarithmic factors, and the
  exponents do not depend on the dimension. This is an exponential
  improvement over plain Langevin Monte Carlo, whose relaxation time obeys the
  Arrhenius law. The assumptions cover, in particular, the potentials of Gaussian
  mixtures, as long as their critical points are nondegenerate and their wells carry
  comparable mass. The analysis extends that of~\cite{AgarwalIyerEA26} from the
  torus~$\T^d$ to the noncompact state space~$\R^d$.
  The two new ingredients are a weighted~$L^\infty$ bound on the eigenfunctions of the
  Langevin generator, obtained by combining Agmon estimates at infinity with a local
  maximum principle near the critical points, and a Weyl-type eigenvalue count derived
  from the Cwikel--Lieb--Rozenblum and Lieb--Thirring inequalities. For potentials with
  more than two wells the eigenfunction bound requires an additional argument, which we
  give by constructing a perturbation of~$U$ that renders every saddle an effective
  local maximum.
\end{abstract}


\maketitle
\section{Introduction}

In this section, we introduce the Langevin AIS algorithm and state our main result informally (Sections~\ref{ss:main-results} and~\ref{ss:informal}). We then place it in the context of the existing literature in Section~\ref{ss:related-work} and outline the plan of the paper in Section~\ref{ss:plan}.

\subsection{Main results}\label{ss:main-results}
Let~$\mathcal X$ be a configuration space and~$U \colon \mathcal X \to \R$ be an energy function.
Given~$\epsilon > 0$, the Gibbs measure with temperature~$\epsilon$ is defined by
\begin{equation}\label{e:piDef}
  \pi_\epsilon \defeq \frac{\tilde \pi_\epsilon}{Z_\epsilon}
  \quad\text{where}\quad
  \tilde \pi_\epsilon = e^{-U / \epsilon}
  \,,
  \quad\text{and}\quad
  Z_\epsilon = \int_{\mathcal X} \tilde \pi_\epsilon \, dx
  \,,
\end{equation}
where the last integral is carried out with respect to some reference measure on~$\mathcal X$.
We typically only consider the case where~$\mathcal X$ is the~$d$-dimensional Euclidean space~$\R^d$, or the~$d$-dimensional torus~$\T^d$.
In these cases the reference measure in~\eqref{e:piDef} is the Lebesgue measure.

Consider the overdamped Langevin equation
\begin{equation}\label{e:langevin}
  d X^\epsilon_t = -\grad U(X^\epsilon_t) \, dt + \sqrt{2\epsilon} \, dB_t
  \,,
\end{equation}
where~$B$ is a standard~$d$-dimensional Brownian motion.
It is well known (see for instance~\cite{Pavliotis14}) that the stationary distribution of the Langevin equation is~$\pi_\epsilon$, and the \emph{Langevin Monte Carlo (LMC)} algorithm obtains samples from~$\pi_\epsilon$ by simulating~\eqref{e:langevin} for a long time.

The disadvantage of LMC is that when~$U$ is not convex, the rate at which solutions to~\eqref{e:langevin} converge to the stationary distribution is exponentially small in~$1/\epsilon$.
This is the Arrhenius law~\cite{Arrhenius89}. Quantitatively, the spectral gap of the generator of~\eqref{e:langevin} is of order~$\exp(-\hat\gamma/\epsilon)$, where~$\hat\gamma$ is the energy barrier separating the wells of~$U$ (see~\eqref{e:gammaHatDef} below, and~\cite{Kolokoltsov00, MenzSchlichting14}).
We show that this slow rate of convergence can be overcome if one uses \emph{annealed importance sampling} (AIS)~\cite{Neal01}, or an autonormalized version of AIS.

\subsubsection{Langevin Annealed Importance Sampling}

For convenience, we first state the Langevin AIS algorithm.
\begin{algorithm}[H]
  \caption{Langevin Annealed Importance Sampling (Langevin AIS)}\label{a:aisLMC}
  \begin{algorithmic}[1]
    \item[\textbf{Tunable parameters:}]
    \Statex
      \begin{enumerate}[(1)]
	\item Target temperature~$\epsilon > 0$.
	\item Annealing schedule $K \in \N$ and~$\epsilon_1 > \cdots > \epsilon_{K+1} = \epsilon$
	\item LMC simulation time~$T$.
      \end{enumerate}
    \State Start with~$X_{0} \in \R^d$ arbitrary, and weight~$\tilde w_0 = 1$.
    \For{$k=1, \dots, K$}
      \State Simulate~\eqref{e:langevin} with temperature~$\epsilon_k$, starting from~$X_{k-1}$, and let~$X_k$ be the state after time~$T$.
      \State $\tilde w_{k} \gets \tilde w_{k-1} \frac{\tilde \pi_{\epsilon_{k+1}}(X_k)}{\tilde \pi_{\epsilon_k}(X_{k})}$.
	
    \EndFor
    \State \Return~$X_{K}$ and the \emph{unnormalized} weight~$\tilde w_K$.
  \end{algorithmic}
\end{algorithm}
\subsection{Informal statement of main result}\label{ss:informal}

We now state our main result informally. Its precise form is Theorem~\ref{t:main-theorem} in Section~\ref{ss:main-result}, and its extension to potentials with any finite number of wells is Theorem~\ref{t:main-theorem-multiwell} in Section~\ref{s:multi-well}.

\begin{theorem}[Informal]
    \label{t:main-theorem-intro}
    Let~$U\colon \R^d\to \R$ be a regular double-well potential with wells of equal depth (but not necessarily the same shape), whose growth at infinity is quadratic or steeper. Let~$\nu \in (0,1/8)$ and let~$\bar\epsilon \in (0, 1]$ be an initial temperature. Then there exist constants~$C_w(\nu), C_{T, 1}(\nu)$, and~$C_{T,2}(\nu, \bar\epsilon)$ such that the following holds. For any~$\underline\epsilon \in (0, \bar\epsilon/2)$ and~$\delta \in (0,1]$, we choose~$K, T, N$ such that
    \begin{align}
        \label{e:multiple-particles-Kchoice}
        K &= \floor*{ \frac{1}{\nu}\paren*{ \frac{1}{\underline\epsilon} - \frac{1}{\bar\epsilon}  } }\,,\\
        \label{e:multiple-particles-Tchoice}
        T &\geq   \max\Bigg\{C_{T, 1}\paren*{\frac{1}{\underline\epsilon} + \log K},\\
        &\qquad\qquad C_{T,2}\paren*{1 + \log\paren*{\frac{3}{\delta}} + 2\nu\max_{1\leq i\leq N} U(X_0^i)  }, \frac{\log 2}{3\nu}\Bigg\}\,,\\
        \label{e:multiple-particles-Nchoice}
        N &\geq \frac{9C_w}{\delta^2}\,,
    \end{align}
    and define the sequence of temperatures~$(\epsilon_k)_{k=1}^{K+1}$ such that
    \begin{equation}
        \bar\epsilon = \epsilon_1 > \epsilon_2 > \cdots > \epsilon_{K+1} = \underline\epsilon\,,
    \end{equation}
    and their inverse temperatures~$\set{1/\epsilon_k}_{k=1}^{K+1}$ are linearly spaced. Perform~$N$ independent runs of Algorithm~\ref{a:aisLMC} with these parameters and let~$(X_K^1, \tilde w_K^1)$, \ldots, $(X_K^N, \tilde w_K^N)$ be the resulting points and unnormalized weights. Define the empirical measure~$\mu_N$ by
    \begin{equation}\label{e:muN-wKi-def}
        \mu_N \defeq \sum_{i=1}^N w_K^i \delta_{X_K^i} \quad \text{where}\quad w_K^i = \frac{\tilde w_K^i}{\sum_{j=1}^N \tilde w_K^j}\,.
    \end{equation}
    Then for any bounded test function~$f$, we have 
    \begin{equation}\label{e:main-theorem-small-error}
       \norm*{ \ip*{f, \mu_N - \pi_{\epsilon_{K+1}}} }_{L^2(\P)} \leq \delta \norm{f}_\infty\,.
    \end{equation}
    Consequently, for any~$s > d/2$, there exists an explicit dimensional constant~$C_s$ (independent of~$\nu, \bar\epsilon, \underline\epsilon, \delta$) such that
    \begin{equation}\label{e:langevin-ais-convergence-sobolev}
        \norm*{\norm*{\mu_N-\pi_{\epsilon_{K+1}}}_{ H^{-s}(\R^d)}}_{L^2(\P)} \leq C_s \delta\,.
    \end{equation}
    Here, the notation~$\ip{f, \mu}$ used in~\eqref{e:main-theorem-small-error} is 
    \begin{equation}
        \ip{f, \mu} = \int_{\R^d} f\, d\mu\,,
    \end{equation}
    and~$\norm{\cdot}_{H^{-s}}$ denotes the inhomogeneous Sobolev norm with index~$-s$.
\end{theorem}

\subsection{Relation to previous work}\label{ss:related-work}

Sampling from a Gibbs measure at low temperature is a central problem in statistical
physics, Bayesian statistics and theoretical computer science, and the obstruction
addressed here, metastability, is equally fundamental. When~$U$ has several wells and~$\epsilon$
is small, the Langevin diffusion~\eqref{e:langevin} spends an exponentially long time, of order
$\exp(\hat\gamma/\epsilon)$, in a neighborhood of one well before crossing to another.
Sharp forms of this statement go back to~\cite{Eyring35,Kramers40}. See~\cite{Kolokoltsov00}
for a semiclassical treatment, \cite{BovierEckhoffEA04,BovierGayrardEA05} for a potential-theoretic one,
and~\cite{MenzSchlichting14} for the Eyring--Kramers
formula for the optimal constants in the Poincar\'e and logarithmic Sobolev
inequalities.
Because of this metastability, the mixing time of the Langevin diffusion is exponentially large
in~$1/\epsilon$, and running Langevin Monte Carlo until the law of~$X^\epsilon_t$ is close
to~$\pi_\epsilon$ is prohibitively expensive at low temperatures. By contrast, when~$U$
is convex, or more generally when~$\pi_\epsilon$ satisfies a logarithmic Sobolev inequality with a constant of
order one, Langevin Monte Carlo and its discretizations converge rapidly~\cite{VempalaWibisono19}.

Annealing-based algorithms circumvent the obstruction by interpolating between a
high temperature, at which the target is nearly unimodal and mixing is fast, and the
low target temperature. Classical examples are parallel tempering~\cite{SwendsenWang86},
simulated tempering~\cite{MarinariParisi92,GeyerThompson95} and tempered
transitions~\cite{Neal96a}. Decomposition techniques give conditions for rapid mixing of
tempering chains~\cite{MadrasRandall02,WoodardSchmidlerEA09}, polynomial bounds are known
for simulated tempering Langevin Monte Carlo when the modes have the same
shape~\cite{LeeRisteskiEA18,GeLeeEA20}, and spectral gap bounds for parallel and simulated
tempering that are polynomial in~$1/\epsilon$ were obtained in~\cite{Son26}. Annealed
importance sampling was introduced in~\cite{Neal01} and is the basis of many
sequential Monte Carlo methods~\cite{DelMoralDoucetEA06,Liu08,ChopinOmiros20}. Error bounds
for sequential Monte Carlo on multimodal targets were obtained
in~\cite{Schweizer12,PaulinJasraEA19,MathewsSchmidler24,LeeSantanaGijzen24} under structural
assumptions on the target or on the transition kernels. Annealed Langevin Monte Carlo, which
uses neither importance weights nor resampling, has also been analyzed~\cite{GuoTaoEA25}.

More recently, explicit time-complexity bounds have been obtained for annealing algorithms
that only evaluate~$U$ and~$\grad U$. It is shown in~\cite{HanIyerEA26} that
annealed sequential Monte Carlo targeting a multimodal Gibbs measure on the
torus~$\T^d$ has time complexity of order~$\epsilon^{-4}$, and a bound of order~$\epsilon^{-2}$
for Langevin AIS, again on~$\T^d$, is proved in~\cite{AgarwalIyerEA26}. In both cases the
dependence on the accuracy~$\delta$ is of order~$\delta^{-2}$. The sequential Monte Carlo result
is extended in~\cite{Han25} to~$\R^d$ under the additional assumption that~$\Delta U$ is bounded. For a double
well with wells of equal depth the time complexity is of order~$\epsilon^{-7}$, and the
multi-well case is sketched. In all of these results the exponent of~$1/\epsilon$ does
not depend on the dimension~$d$. This is the relevant benchmark, because rejection
sampling from the uniform distribution on~$\T^d$ already has cost polynomial
in~$1/\epsilon$, namely of order~$\epsilon^{-d/2}$ per sample, as its acceptance
probability is~$Z_\epsilon\asymp\epsilon^{d/2}$. Moreover, none of these results
requires prior knowledge of the shape of~$U$, such as the number or the locations of
its wells. This is in contrast with analyses that assume a known decomposition of the
state space~\cite{MathewsSchmidler24} or known masses of the modes~\cite{LeeSantanaGijzen24}.

The present paper removes the compactness of the state space from the analysis
of~\cite{AgarwalIyerEA26}. Compared with~\cite{Han25}, our result concerns AIS rather
than sequential Monte Carlo and has the smaller exponent~$2$ in~$1/\epsilon$. It also
replaces the boundedness of~$\Delta U$ by the growth conditions of
Assumption~\ref{a:U-growth-control}, which allow superquadratic growth and, through
Remark~\ref{r:extension}, Gaussian mixtures with distinct covariance matrices. Finally, it
treats potentials with any finite number of wells, with saddles of arbitrary Morse
index, in full (Section~\ref{s:multi-well}). Removing compactness is not a formality. Two ingredients of the compact
theory are unavailable on~$\R^d$. The generator~$\mathcal L_\epsilon$ need not have
discrete spectrum, and the eigenfunctions need not be bounded. We recover both under
the growth conditions of Assumption~\ref{a:U-growth-control}, the first through a
Cwikel--Lieb--Rozenblum and Lieb--Thirring eigenvalue count (Lemma~\ref{l:ef-weighted-sum-Linfty-bds})
and the second through a combination of Agmon estimates, which control the eigenfunctions
at infinity (Lemma~\ref{l:weighted-eigen-vanish-infty}), with a local maximum principle
near the critical points (Lemma~\ref{l:eigenfunctions-weighted-Linfty-bdd}). A further
difficulty, absent when~$U$ has only two wells, arises for general multi-well potentials.
The comparison between the~$L^2(\pi_\epsilon)$ and~$L^2(dx)$ norms that underlies the
eigenfunction bound degrades at saddles lying far above the smallest barrier. Section~\ref{s:multi-well}
resolves this by a perturbation argument in the spirit of~\cite{MenzSchlichting14}
and~\cite{Son26}.

\subsection{Plan of the paper}\label{ss:plan}

In Section~\ref{s:main}, we state the precise assumptions (Section~\ref{ss:assumptions}) and the main results for double-well potentials, Proposition~\ref{p:single-particle-proposition} and Theorem~\ref{t:main-theorem} (Section~\ref{ss:main-result}). We also list the three key lemmas from which they follow (Section~\ref{ss:key-lemmas}). Lemma~\ref{l:lyapunov-weighted-Linftybds-and-their-sums} bounds the eigenfunctions of the generator in a weighted~$L^\infty$ norm, and Lemmas~\ref{l:bias-bound} and~\ref{l:squared-weight-bound} bound the bias and the second moment of the unnormalized weight in terms of these eigenfunction bounds.

In Section~\ref{s:ef-bounds}, we prove Lemma~\ref{l:lyapunov-weighted-Linftybds-and-their-sums}, splitting it into Lemmas~\ref{l:eigenfunctions-weighted-Linfty-bdd} and~\ref{l:ef-weighted-sum-Linfty-bds}. Most of the difficulties caused by the noncompactness of~$\R^d$ are dealt with there. In Section~\ref{s:bias-variance}, we prove the bias bound, Lemma~\ref{l:bias-bound}, in Section~\ref{ss:bias} and the variance bound, Lemma~\ref{l:squared-weight-bound}, in Section~\ref{ss:variance}. In Section~\ref{s:proof-main}, we combine the three key lemmas to prove Proposition~\ref{p:single-particle-proposition} and Theorem~\ref{t:main-theorem}.

In Section~\ref{s:multi-well}, we extend these results to potentials with any finite number of wells (Proposition~\ref{p:single-particle-proposition-multiwell} and Theorem~\ref{t:main-theorem-multiwell}). Only one step of Sections~\ref{s:ef-bounds}--\ref{s:proof-main} requires a new argument, namely a uniform bound on the low-lying eigenfunctions near the saddles of~$U$. We reduce this bound to Lemmas~\ref{l:lyapunov-uniform-boundedness} and~\ref{l:lyapunov-existence}, which we prove in Sections~\ref{ss:proof-uniform-bound} and~\ref{ss:proof-lyapunov-existence}, respectively. The latter relies on a perturbation of~$U$ near its saddles, Lemma~\ref{l:perturb-prop}, which we construct in Section~\ref{ss:higher-index}. Finally, Appendix~\ref{s:appendix} contains the proofs of Lemmas~\ref{l:dEF-small} and~\ref{l:partition-function-estimates}.

\section{Assumptions, main result and key lemmas}\label{s:main}

In this section, we state the assumptions and the main result of the paper for double-well potentials. The assumptions are given in Section~\ref{ss:assumptions} and the main result, Proposition~\ref{p:single-particle-proposition} and Theorem~\ref{t:main-theorem}, in Section~\ref{ss:main-result}. In Section~\ref{ss:key-lemmas}, we state the key lemmas from which the main result is deduced, Lemmas~\ref{l:lyapunov-weighted-Linftybds-and-their-sums}, \ref{l:bias-bound} and~\ref{l:squared-weight-bound}, and we explain how they are used.

\subsection{Assumptions}\label{ss:assumptions}

We begin by stating the assumptions that are required in the main theorem. The assumptions are very similar to those in~\cite{HanIyerEA26} except that here the domain is $\R^d$. The first assumption requires~$U$ to be a regular double-well function with nondegenerate critical points. 

\begin{assumption}\label{a:criticalpts}
 The function~$U \in C^6(\R^d,\R)$ has a nondegenerate Hessian at all of its critical points and has exactly two local minima, located at~$x_{\min, 1}$ and~$x_{\min, 2}$.
\end{assumption}

Without loss of generality, we normalize~$U$ so that
\begin{equation}\label{e:Upositive}
  0 = U(x_{\min, 1} )  \leq U( x_{\min, 2} )\,.
\end{equation}

Our next assumption concerns the saddle separating the two local minima.
The \emph{saddle height}~$\hat U$ between~$x_{\min,1}$ and~$x_{\min,2}$ is the lowest energy level that a continuous path joining the two minima must reach, namely
\begin{equation}\label{e:UHatDef}
    \hat{U} = \hat{U}(x_{\min,1},x_{\min,2}) \defeq \inf_\omega \sup\limits_{t\in [0,1]}U(\omega(t))
    .
\end{equation}
Here, $\omega \in C( [0, 1]; \mathbb R^d)$ is any path such that~$\omega(0) =x_{\min,1}$,~$\omega(1) = x_{\min,2}$.
\medskip

We also need to assume a nondegeneracy condition on the saddle.
\begin{assumption}\label{a: nondegeneracy}
   The saddle height between~$x_{\min,1}$ and~$x_{\min,2}$ is attained at a unique critical
   point~$s_{1,2}$, and this point has Morse index one, that is, exactly one eigenvalue
   of~$D^2U(s_{1,2})$ is negative and the remaining~$d-1$ eigenvalues are positive.
   After a translation we assume throughout that~$s_{1,2}=0$.
\end{assumption}

Whereas~$\hat U$ is an absolute energy level, the \emph{energy barrier}~$\hat\gamma$ is the energy \emph{increase} that must be paid to move from the shallower minimum~$x_{\min, 2}$ to the global minimum~$x_{\min, 1}$. Under Assumption~\ref{a: nondegeneracy} both quantities are expressed in terms of~$s_{1,2}$ by
\begin{equation}\label{e:gammaHatDef}
    \hat{\gamma}\defeq U(s_{1,2})-U(x_{\min,2})= U(0)-U(x_{\min, 2})
    \quad\text{and}\quad
    \hat U = U(s_{1,2})=U(0)\,.
\end{equation}

Finally, we require~$\pi_\epsilon$ to be genuinely multimodal throughout the temperature range of interest, in the sense that both \emph{basins of attraction} carry a non-negligible share of the mass.
The basin of attraction around~$x_{\min, i}$, denoted by~$\Omega_i$, is the set of all initial points for which the gradient flow of~$U$ eventually reaches~$x_{\min, i}$.
Precisely,~$\Omega_{i}$ is defined by
\begin{equation}\label{e:basin-def}
  \Omega_i
    \defeq
    \set[\Big]{
      y\in\R^d \st
      \lim_{t\to\infty}y_t=x_{\min,i}, \text{ where }\dot{y}_t=-\nabla U(y_t)
	\text{ with } y_0=y
      }
    ,
\end{equation}
and our multimodality condition is as follows.

\begin{assumption}\label{a:massRatioBound}
  There exist~$\kappa \in [0,1/2)$ and a constant~$C_m \geq 1$ such that
   \begin{equation}\label{e:massRatioBound}
     \min_{i\in \set{1,2}}\inf_{\substack{
       \epsilon\in (\kappa, 1]\\
     }}
      \pi_{\epsilon}(\Omega_i)
	\geq \frac{1}{C_m^2}
	.
   \end{equation}
\end{assumption}

The parameter~$\kappa$ in Assumption~\ref{a:massRatioBound} is the lowest temperature at which the mass condition is required. The case of interest is~$\kappa=0$, in which case~\eqref{e:massRatioBound} must hold at every temperature in~$(0,1]$. By~\cite[Lemma 4.4]{HanIyerEA26} this is automatic, with~$C_m$ depending only on~$U$, when the two wells have equal depth, that is when~$U(x_{\min,2})=0$. Allowing~$\kappa>0$ accommodates wells of unequal depth, for which the shallower well loses its mass as~$\epsilon\to0$ and~\eqref{e:massRatioBound} can only hold down to a positive temperature.

Since the state space~$\R^d$ is noncompact, we need further assumptions on~$U$ in order to guarantee that the Langevin diffusion mixes and that~$\mathcal L_\epsilon$ has discrete spectrum.

\begin{assumption}\label{a:U-growth-control}
There exists a constant~$C_L>0$ such that 
    \begin{align}
        \label{a:DU-blowup}
        \lim_{\abs{x}\to \infty}&\abs*{DU} = \infty\,,\\
        \label{a:DU^2-dominates-D^2U}
        \lim_{\abs{x}\to \infty} &\frac{\abs{D^2U}}{\abs{DU}^2} = 0 \,,\\
        \label{a:U-lyapunov}
        0\leq U &\leq C_L(1+\abs{DU}^2)\,.
    \end{align}
\end{assumption}

\begin{remark}[Extension of the results]\label{r:extension}
The assumption~\eqref{a:DU^2-dominates-D^2U} implies the existence of a constant~$C_\Delta$ such that
\begin{equation}
\tag{\ref*{a:DU^2-dominates-D^2U}*}
\label{a:DU-exp-growth-at-most}
    \abs{D^2 U} \leq C_\Delta (1+\abs{DU}^2)\,.
\end{equation}
Repeating the proof in this paper shows that the main theorem also holds under the weaker set of assumptions~\eqref{a:DU-blowup}, \eqref{a:DU-exp-growth-at-most}, and~\eqref{a:U-lyapunov}, in place of~\eqref{a:DU-blowup}--\eqref{a:U-lyapunov}, provided that we start the annealing from a \emph{sufficiently small} highest temperature. The smallness of the initial temperature guarantees that the Langevin diffusion~\eqref{e:langevin} also mixes at the highest temperature. Since the proof in this case follows along the same lines, we omit the details.

We also note that we assume~$U$ has double wells, instead of multiple wells, in Assumption~\ref{a:criticalpts} only for simplicity and clarity of presentation. The main result extends to potentials with multiple wells satisfying Assumptions~\ref{a:criticalpts-multiwell}, \ref{a:massRatioBound-multiwell}, and~\ref{a:U-growth-control}, at the cost of additional technical work to prove the uniform boundedness of the eigenfunctions. The proof is presented in Section~\ref{s:multi-well} below.
\end{remark}

\begin{remark}[Examples]
    Every potential~$U \geq 0$ such that, for some~$p\in [2,\infty)$,
    \begin{equation}\label{e:U-poly-growth}
        U \lesssim \abs{x}^p\,,\quad \abs{DU} \gtrsim \abs{x}^{p-1}\,,\quad\text{and}\quad \abs{D^2U}\lesssim \abs{x}^{p-2}
    \end{equation} 
    satisfies the conditions~\eqref{a:DU-blowup}--\eqref{a:U-lyapunov}. Here, for two non-negative functions~$f$ and~$g$, we write~$f\lesssim g$ (equivalently, $g \gtrsim f$) if there exists a constant~$C>0$ such that~$f(x) \leq C g(x)$ for all sufficiently large~$\abs{x}$. Indeed, \eqref{e:U-poly-growth} gives~$\abs{D^2U}/\abs{DU}^2 \lesssim \abs{x}^{-p} \to 0$ and $U \lesssim \abs{x}^p \lesssim \abs{DU}^{p/(p-1)} \lesssim 1+\abs{DU}^2$, where the last step uses~$p\geq 2$.

    In particular, any nonnegative potential~$U$ such that~$U(x) = \abs{x}^p$ for all sufficiently large~$\abs{x}$ and some~$p\in [2,\infty)$ satisfies~\eqref{e:U-poly-growth}. Moreover, the potential of a Gaussian mixture on~$\R^d$ with a common covariance matrix satisfies~\eqref{e:U-poly-growth} with~$p=2$. That is, suppose that
    \begin{equation}\label{e:U-potential-for-GM-with-same-covariance}
        \exp\paren*{-U(x)} = \sum_{j=1}^{N} \alpha_j \exp\paren*{-\frac{1}{2}\paren*{x-\mu_j}^T \Gamma^{-1}\paren*{x-\mu_j}}
    \end{equation} 
    for some~$N \in \N$, points~$(\mu_j)_{j=1}^{N} \subset \R^d$, a symmetric positive-definite matrix~$\Gamma \in \R^{d\times d}$, and weights~$\alpha_j \geq 0$ with~$\sum_{j=1}^{N} \alpha_j = 1$. Then~$U \geq 0$, and~$U$ satisfies~\eqref{e:U-poly-growth} with~$p=2$. Suppose, in addition, that the parameters are such that~$U$ satisfies Assumptions~\ref{a:criticalpts-multiwell} and~\ref{a:massRatioBound-multiwell}. This imposes conditions on the parameters. For instance, the means~$\mu_j$ must be sufficiently separated for~$U$ to have distinct wells. Then our main result holds for~$U$ in~\eqref{e:U-potential-for-GM-with-same-covariance}, as discussed in Remark~\ref{r:extension} above.
    
    For more general Gaussian mixtures with varying covariance matrices, where~$U$ satisfies
    \begin{equation}\label{e:U-potential-for-GM}
        \exp\paren*{-U(x)} = \sum_{j=1}^{N} \alpha_j \exp\paren*{-\frac{1}{2}\paren*{x-\mu_j}^T \Gamma_j^{-1}\paren*{x-\mu_j}}
    \end{equation} 
    with a collection~$(\mu_j, \Gamma_j)_{j=1}^{N}$, the growth condition~\eqref{e:U-poly-growth} generally fails with~$p=2$, since~$\abs{D^2U}$ can grow as fast as~$\abs{x}^2$. Nevertheless, such~$U \geq 0$ still satisfies
    \begin{equation}
        U \lesssim \abs{x}^2\,,\quad \abs{DU} \gtrsim \abs{x}\,,\quad\text{and}\quad \abs{D^2U}\lesssim \abs{x}^2\,,
    \end{equation}
    and hence~$\abs{D^2U} \lesssim 1+\abs{DU}^2$. Therefore~$U$ satisfies~\eqref{a:DU-blowup}, \eqref{a:DU-exp-growth-at-most}, and~\eqref{a:U-lyapunov}. If, in addition, the parameters are such that~$U$ satisfies Assumptions~\ref{a:criticalpts-multiwell} and~\ref{a:massRatioBound-multiwell}, then combining this with Remark~\ref{r:extension} shows that our main result holds for~$U$ in~\eqref{e:U-potential-for-GM}, provided that the initial highest temperature is sufficiently small.
\end{remark}

\begin{remark}[Role of the assumptions]
    We briefly explain why the conditions in Assumption~\ref{a:U-growth-control} are natural.
    
    First, the growth conditions~\eqref{a:DU-blowup} and~\eqref{a:DU^2-dominates-D^2U} guarantee that the Langevin generator~$\mathcal L_\epsilon$ has discrete spectrum (see, e.g.,~\cite{Pavliotis14, Kolokoltsov00}). They also give other desirable properties at every temperature, or at every sufficiently small temperature if we only assume~\eqref{a:DU-exp-growth-at-most}. These include the boundedness of the eigenfunctions in appropriate norms and Weyl's law, which are essential for obtaining the bias and variance bounds of the AIS sampler.

    Second, \eqref{a:U-lyapunov} is a Lyapunov-type condition. Under~\eqref{a:DU-blowup} and~\eqref{a:DU^2-dominates-D^2U} it is equivalent to the usual Lyapunov drift condition for the generator~$\mathcal L_\epsilon$ defined in~\eqref{e:generator-def} below, namely that there exist~$c, R > 0$ with~$\mathcal L_1 U \leq - c U$ on~$\set{\abs{x} > R}$. Similarly, under~\eqref{a:DU-blowup} and~\eqref{a:DU-exp-growth-at-most}, the assumption~\eqref{a:U-lyapunov} is equivalent to the existence of~$c, R > 0$ such that~$\mathcal L_\epsilon U \leq - cU$ on~$\set{\abs{x} > R}$ for all sufficiently small~$\epsilon$.

    Finally, we record a consequence of~\eqref{a:DU-blowup} that is used repeatedly below. There exist constants~$D_1, D_2 > 0$ such that
    \begin{equation}\label{e:U-coercive}
        U(x) \geq D_1\abs{x} - D_2\,, \qquad \forall x\in \R^d\,,
    \end{equation}
    as shown, e.g., in~\cite[Lemma 3.14]{MenzSchlichting14}. In particular~$U(x)\to\infty$ as~$\abs{x}\to\infty$, so all sublevel sets of~$U$ are compact.
\end{remark}

\subsection{Statement of the main result}\label{ss:main-result}

In this subsection, we state the main result of the paper, Theorem~\ref{t:main-theorem}, together with Proposition~\ref{p:single-particle-proposition}, from which it follows.

\begin{proposition}[Bias and variance bounds]\label{p:single-particle-proposition}
Make the Assumptions~\ref{a:criticalpts}--\ref{a:U-growth-control}. Then, for all~$\nu \in (0, 1/8)$ and~$\bar \epsilon \in (\kappa, 1]$, there exist constants~$C_w(\nu)
, C_{T, 1}(\nu)$, and~$C_{T,2}(\nu, \bar\epsilon)$ such that  the following holds. 
For any~$\underline\epsilon \in (\kappa, \bar\epsilon/2)$ and~$\delta\in (0,1]$, choose~$K$ and~$T$ such that 
    \begin{align}
        \label{e:single-particle-Kchoice}
        K &= \floor*{ \frac{1}{\nu}\paren*{ \frac{1}{\underline\epsilon} - \frac{1}{\bar\epsilon}  } }\,,\\
        \label{e:single-particle-Tchoice}
        T &\geq   \max\set*{C_{T, 1}\paren*{\frac{1}{\underline\epsilon} + \log K}, C_{T,2}\paren*{1 + \log\paren*{\frac{1}{\delta}} +  2\nu U(X_0)}, \frac{\log 2}{3\nu}}\,,
    \end{align} 
    and define the sequence of temperatures~$(\epsilon_k)_{k=1}^{K+1}$ such that
    \begin{equation}
        \bar\epsilon = \epsilon_1 > \epsilon_2 > \cdots > \epsilon_{K+1} = \underline\epsilon\,,
    \end{equation}
    and their inverse temperatures~$\set{1/\epsilon_k}_{k=1}^{K+1}$ are linearly spaced. Let~$\tilde w_K$ and~$X_K$ denote the output of Algorithm~\ref{a:aisLMC} run with these parameters.
    Then for any bounded test function~$f$, we have
    \begin{align}
        \label{e:single-particle-bias-bound}
        \abs*{\frac{\E\brak*{\tilde w_K f(X_K)}}{\E\brak*{\tilde w_K}} - \ip*{f, \pi_{\epsilon_{K+1}}}}&\leq  \delta \norm{f}_\infty \,,\\
        \label{e:single-particle-variance-bound}
        \var\paren*{\frac{\tilde w_K f(X_K)}{ \E\brak*{\tilde w_K} }} &\leq C_w \norm{f}_\infty^2\,.
    \end{align}
\end{proposition}

Proposition~\ref{p:single-particle-proposition} concerns a single run of Algorithm~\ref{a:aisLMC}. In practice, one performs~$N$ independent runs and weights their outputs by the normalized weights. Combining the bias and variance bounds of Proposition~\ref{p:single-particle-proposition}, we show that the resulting weighted empirical measure approximates~$\pi_{\epsilon_{K+1}}$ to accuracy~$\delta$ as soon as~$N$ is of order~$\delta^{-2}$.

\begin{theorem}[Rapid convergence of Langevin AIS on~$\R^d$]
    \label{t:main-theorem}
    Make the Assumptions~\ref{a:criticalpts}--\ref{a:U-growth-control}. Let~$\nu \in (0,1/8)$ and~$\bar\epsilon \in (\kappa, 1]$, and let~$C_w, C_{T, 1}$ and~$C_{T,2}$ be the constants from Proposition~\ref{p:single-particle-proposition}. 
    Then, for any~$\underline\epsilon \in (\kappa, \bar\epsilon/2)$ and~$\delta \in (0,1]$, we choose~$K, T, N$ as in~\eqref{e:multiple-particles-Kchoice}, \eqref{e:multiple-particles-Tchoice}, and~\eqref{e:multiple-particles-Nchoice}, respectively, and define the sequence of temperatures~$(\epsilon_k)_{k=1}^{K+1}$ such that
    \begin{equation}
        \bar\epsilon = \epsilon_1 > \epsilon_2 > \cdots > \epsilon_{K+1} = \underline\epsilon\,,
    \end{equation}
    and their inverse temperatures~$\set{1/\epsilon_k}_{k=1}^{K+1}$ are linearly spaced. Perform~$N$ independent runs of Algorithm~\ref{a:aisLMC} with these parameters and let~$(X_K^1, \tilde w_K^1)$, \ldots, $(X_K^N, \tilde w_K^N)$ be the resulting points and unnormalized weights. Define the empirical measure~$\mu_N$ by~\eqref{e:muN-wKi-def}.
    Then for any bounded test function~$f$, we have~\eqref{e:main-theorem-small-error}. 
    Consequently, for any~$s > d/2$, there exists an explicit dimensional constant~$C_s$ (independent of~$\nu, \bar\epsilon, \underline\epsilon, \delta$) such that~\eqref{e:langevin-ais-convergence-sobolev} holds. 
\end{theorem}

\begin{remark}[Time complexity]\label{r:time-complexity}
We measure the cost of the algorithm by the total time for which the Langevin diffusion~\eqref{e:langevin} is simulated. Theorem~\ref{t:main-theorem} uses~$N$ independent runs of Algorithm~\ref{a:aisLMC}, each of which simulates~\eqref{e:langevin} for time~$T$ at each of the~$K$ levels, so the total cost is~$NKT$. Choose~$K$ and~$N$ as in~\eqref{e:multiple-particles-Kchoice} and~\eqref{e:multiple-particles-Nchoice}, with equality in the latter up to rounding, and choose~$T$ with equality in~\eqref{e:multiple-particles-Tchoice}. Since~$K \leq 1/(\nu\underline\epsilon)$ and hence~$\log K \leq C/\underline\epsilon$, we obtain
\begin{equation}\label{e:time-complexity}
  NKT \leq C\paren*{\frac{1}{\underline\epsilon^{2}\delta^{2}} + \frac{1}{\underline\epsilon\,\delta^{2}}\paren*{\log\frac{1}{\delta} + \max_{1\leq i\leq N}U(X_0^i)}}
\end{equation}
for a constant~$C$ that depends on~$\nu$, $\bar\epsilon$, $d$ and~$U$, but not on~$\underline\epsilon$, $\delta$ or the initial points. In particular, if the initial points lie in a fixed bounded set, the total cost is of order~$\underline\epsilon^{-2}\delta^{-2}$, up to a factor logarithmic in~$1/\delta$, and the exponents of~$1/\underline\epsilon$ and~$1/\delta$ do not depend on the dimension~$d$. The bound~\eqref{e:time-complexity} also makes the dependence on the initial points explicit. It enters only through the term~$2\nu\max_{i}U(X_0^i)$ in~\eqref{e:multiple-particles-Tchoice}, which accounts for the time the Langevin diffusion at the initial temperature~$\bar\epsilon$ needs to forget its initial position, and the total cost grows only linearly in the initial energies.
\end{remark}

\subsection{Key lemmas}\label{ss:key-lemmas}

In this subsection we state the key lemmas from which Proposition~\ref{p:single-particle-proposition} and Theorem~\ref{t:main-theorem} are deduced. To this end, we first collect some facts about the spectral decomposition associated with the Langevin diffusion~\eqref{e:langevin}.

Let~$\mathcal L_\epsilon$ be the generator of~\eqref{e:langevin}, given by
\begin{equation}\label{e:generator-def}
  \mathcal L_\epsilon f = -\grad U \cdot \grad f + \epsilon \lap f
  \,,
\end{equation}
and let~$(P_{\epsilon, t})_{t\geq0}$ denote the associated Markov semigroup. For any test function~$f$ the function~$u_t \defeq P_{\epsilon,t}f$ satisfies the \emph{Kolmogorov backward equation} (see, for instance,~\cite[Chapter 2]{Pavliotis14}). That is, if
\begin{equation}
  u_t(x) = P_{\epsilon, t} f(x) \defeq \int_{\R^d} P_{\epsilon, t}(x, dy) f(y)
  \,,
\end{equation}
then
\begin{equation}\label{e:KBackward}
  \partial_t u = \mathcal L_\epsilon u
  \quad\text{and}\quad
  u_0 = f\,.
\end{equation}
Let~$L^2(\pi_\epsilon)$ denote the weighted~$L^2$ space with inner-product and norm defined by
\begin{equation}
  \ip{f, g}_{L^2(\pi_\epsilon)}
    \defeq \int_{\R^d} f g \, \pi_\epsilon \, dx
  \quad\text{and}\quad
  \norm{f}^2_{L^2(\pi_\epsilon)} \defeq 
    \int_{\R^d} \abs{f}^2 \pi_\epsilon \, dx
    \,,
\end{equation}
respectively.
It is well known (see for instance~\cite[Chapter 8]{Kolokoltsov00}) that the operator~$-\mathcal L_\epsilon$ is self-adjoint on the weighted space~$L^2(\pi_\epsilon)$ and has discrete spectrum consisting of nonnegative eigenvalues of finite multiplicity, accumulating only at~$+\infty$.
We denote the eigenvalues by
\begin{equation}\label{e:evalOrder}
  0 = \lambda_{1, \epsilon}
    < \lambda_{2, \epsilon}
    \leq \lambda_{3, \epsilon}
    \leq \cdots
    \,,
\end{equation}
and the corresponding~$L^2(\pi_\epsilon)$ normalized eigenfunctions by~$\psi_{1, \epsilon}$, $\psi_{2, \epsilon}$, etc.
We note that the smallest eigenvalue~$\lambda_{1, \epsilon}$ is~$0$, and the second smallest eigenvalue~$\lambda_{2, \epsilon}$ is strictly positive (i.e.\ $\mathcal L_\epsilon$ has a spectral gap).
The first key lemma bounds the eigenfunctions~$\psi_{i, \epsilon}$ in a weighted~$L^\infty$ norm, together with sums of these bounds damped by the factors~$\exp(-\lambda_{i,\epsilon}T)$. These sums are the quantities that control the bias and the variance of Langevin AIS in Lemmas~\ref{l:bias-bound} and~\ref{l:squared-weight-bound} below. On a compact state space the eigenfunctions are bounded, whereas on~$\R^d$ they may grow at infinity, and the function~$V$ in~\eqref{e:V-normV-def} compensates for this growth. This lemma is the main new ingredient of the paper, and we prove it in Section~\ref{s:ef-bounds}.

\begin{lemma}\label{l:lyapunov-weighted-Linftybds-and-their-sums}
Make the Assumptions~\ref{a:criticalpts}--\ref{a:U-growth-control}.
Fix any~$\nu \in (0, 1/8)$ and define
\begin{equation}\label{e:V-normV-def}
    V \defeq \exp(2\nu U) \quad\text{and}\quad \norm{\cdot}_V \defeq \norm{\cdot / V}_{L^\infty(\R^d)}\,.
\end{equation}
Then, the following holds.
\begin{enumerate}
    \item There exist constants~$\bar C_{\psi_2}, C_d$, and~$\bar C_\nu(\nu, \Lambda)$ such that for all~$\epsilon\in (\kappa,1]$ and~$i\geq 3$, we have
\begin{align}\label{e:second-ef-weighted-norm-unif-ub}
\norm*{\psi_{2, \epsilon}}_{V} &\leq \bar C_{\psi_2},\\
\norm*{\psi_{i, \epsilon}}_{V} &\leq C_d \exp\paren*{\bar C_\nu\lambda_{i, \epsilon}/\epsilon}.
\end{align}
    \item For any~$\Lambda_1,T_0>0$, there exists a constant~$\tilde C_\psi(\Lambda_1, T_0)\geq 1$ such that the following holds. Let~$\epsilon\in (\kappa, 1]$, and suppose that for some~$m\in\N$,
\begin{equation}
    \lambda_{i,\epsilon}\geq\Lambda_1,\qquad \forall i\geq m.
\end{equation}
Then, for every~$T$ satisfying
\begin{equation}\label{e:T-large-for-weighted-sum}
    T\geq \frac{2\bar C_\nu}{\epsilon}+T_0,
\end{equation}
we have
\begin{equation}
    \sum_{i=m}^\infty
    \exp\paren*{-\lambda_{i,\epsilon}T}
    \norm{\psi_{i,\epsilon}}_V^2
    \leq
    \tilde C_\psi
    \epsilon^{-1-\frac d2}
    \exp\paren*{-\Lambda_1\paren*{T-\frac{2\bar C_\nu}{\epsilon}}}\,.
\end{equation}
\end{enumerate}
\end{lemma}

The next two lemmas bound the bias of Langevin AIS and the second moment of the unnormalized weight~$\tilde w_K$ in terms of the sums~$S_b(T)$ and~$S_{v,k}(T)$, defined in~\eqref{e:SbT-def} and~\eqref{e:Svk-def} below, which Lemma~\ref{l:lyapunov-weighted-Linftybds-and-their-sums} controls. To bound the bias, we combine a Lyapunov drift condition for~$V$ with the mixing of the Langevin diffusion at the first level, measured in norms weighted by~$V$ (Section~\ref{ss:bias}). To bound the variance, we use the linear spacing of the inverse temperatures to show that~$\tilde r_k^{-2}$ does not depend on~$k$ and can serve as a Lyapunov function at every level. This yields a closed recurrence for the second moments of~$\tilde w_k$ (Section~\ref{ss:variance}). In Section~\ref{s:proof-main}, we choose~$T$ so large that~$S_b(T)$ and~$S_{v,k}(T)$ are small by Lemma~\ref{l:lyapunov-weighted-Linftybds-and-their-sums}, and the main result follows.

\begin{lemma}[Bias bound]\label{l:bias-bound}
    Make the Assumptions~\ref{a:criticalpts}--\ref{a:U-growth-control}. Then, for all~$\nu \in (0, 1/4)$, there exists a constant~$C_B \geq 1$ such that the following holds.    
    Let~$\kappa < \underline\epsilon < \bar\epsilon \leq 1$, $K\in \N$, and
    \begin{equation}\label{e:Tlarge-bias}
        T \geq \frac{\log 2}{3\nu}\,.
    \end{equation}
    Define the sequence of temperatures~$(\epsilon_k)_{k=1}^{K+1}$ such that
    \begin{equation}
        \bar\epsilon = \epsilon_1 > \epsilon_2 > \cdots > \epsilon_{K+1} = \underline\epsilon\,,
    \end{equation}
    and their inverse temperatures~$\set{1/\epsilon_k}_{k=1}^{K+1}$ are linearly spaced.
    Let~$\tilde w_K$ and~$X_K$ denote the output of Algorithm~\ref{a:aisLMC} run with these parameters. For any bounded measurable function~$f$, we have
    \begin{equation}\label{e:unnormalized-bias-decay}
        \abs*{\E\brak*{\tilde w_K f(X_K)} - \frac{Z_{\epsilon_{K+1}}}{Z_{\epsilon_1}} \ip*{f, \pi_{\epsilon_{K+1}}}}\leq C_B \frac{Z_{\epsilon_{K+1}}}{Z_{\epsilon_1}}  V(X_0)  S_b(T) \norm{f}_\infty  \,,
    \end{equation}
    where~$S_b$ is defined as 
    \begin{equation}\label{e:SbT-def}
        S_b(T) \defeq \sum_{i=2}^\infty \norm{\psi_{i,\epsilon_1}}_V^2 \exp\paren*{-\lambda_{i, \epsilon_1}T}\,.
    \end{equation}
    Here, $V$ and~$\norm{\cdot}_V$ are defined as in~\eqref{e:V-normV-def}.
\end{lemma}

\begin{lemma}[Variance bound]\label{l:squared-weight-bound}
    Make the Assumptions~\ref{a:criticalpts}--\ref{a:U-growth-control}. Then for all~$\nu \in (0,1/8)$,  there exist constants~$\bar C_Z, \bar C_\Theta > 0$ such that the following holds. 
    For any~$\kappa < \underline\epsilon < \bar\epsilon \leq 1$ such that $\underline \epsilon < \bar \epsilon / 2$, choose the number of levels~$K\in \N$ as 
    \begin{equation}\label{e:K-choice-variance}
        K = \floor*{ \frac{1}{\nu}\paren*{ \frac{1}{\underline\epsilon} - \frac{1}{\bar\epsilon}  } }\,,
    \end{equation}
    and let~$T >0$.
    Define the sequence of temperatures~$(\epsilon_k)_{k=1}^{K+1}$ such that
    \begin{equation}
        \bar\epsilon = \epsilon_1 > \epsilon_2 > \cdots > \epsilon_{K+1} = \underline\epsilon\,,
    \end{equation}
    and their inverse temperatures~$\set{1/\epsilon_k}_{k=1}^{K+1}$ are linearly spaced. Extend this sequence by one further level, setting~$1/\epsilon_{K+2}\defeq 2/\epsilon_{K+1}-1/\epsilon_K$, so that~$\set{1/\epsilon_k}_{k=1}^{K+2}$ is still linearly spaced. The level~$\epsilon_{K+2}$ plays no role in the algorithm. Let~$\tilde w_K$ and~$X_K$ denote the output of Algorithm~\ref{a:aisLMC} run with these parameters.
    Then, we have
    \begin{equation}\label{e:unnormalized-variance-ub}
        \E\brak*{\tilde w_K^2} \leq  \exp\paren*{\bar C_\Theta \paren*{1+\sum_{k=2}^K S_{v,k}(T)}   } \frac{Z_{\epsilon_{K+1}} Z_{\epsilon_{K+2}} }{Z_{\epsilon_1} Z_{\epsilon_2}}\paren*{1+\bar C_ZV(X_0)S_b(T)}\,,
    \end{equation}
    where~$V$ and~$S_b$ are defined as in~\eqref{e:V-normV-def} and~\eqref{e:SbT-def}, respectively, and~$S_{v,k}$ is defined as
    \begin{equation}\label{e:Svk-def}
        S_{v,k}(T) \defeq \sum_{i=3}^\infty \exp\paren*{-\lambda_{i, \epsilon_k}T} \norm{\psi_{i,\epsilon_k}}_V^2 \,.
    \end{equation}
\end{lemma}

\section{Eigenfunction bounds (Proof of Lemma~\ref{l:lyapunov-weighted-Linftybds-and-their-sums})}\label{s:ef-bounds}

In this section, we prove Lemma~\ref{l:lyapunov-weighted-Linftybds-and-their-sums}. We split it into two parts, Lemma~\ref{l:eigenfunctions-weighted-Linfty-bdd}, which bounds each eigenfunction~$\psi_{i,\epsilon}$ in the norm~$\norm{\cdot}_V$, and Lemma~\ref{l:ef-weighted-sum-Linfty-bds}, which bounds the sums over~$i$. The proof of Lemma~\ref{l:eigenfunctions-weighted-Linfty-bdd} is based on the maximum principle. To apply it on all of~$\R^d$, we first show in Lemma~\ref{l:weighted-eigen-vanish-infty} that the weighted eigenfunctions vanish at infinity. We then collect in Lemmas~\ref{l:dEF-small}--\ref{l:partition-function-estimates} the spectral properties of~$\mathcal L_\epsilon$ and the estimates on the partition function that are needed near the critical points of~$U$, where we use the local maximum principle. Lemma~\ref{l:ef-weighted-sum-Linfty-bds} then follows from Lemma~\ref{l:eigenfunctions-weighted-Linfty-bdd} together with a bound on the number of eigenvalues of~$\mathcal L_\epsilon$. The proofs of Lemmas~\ref{l:dEF-small} and~\ref{l:partition-function-estimates} are deferred to Appendix~\ref{s:appendix}.

The first lemma shows that~$e^{-2\nu U}\psi_{i,\epsilon}$ vanishes at infinity, so that its absolute value attains its maximum. The proof applies the Agmon estimate to the Schr\"odinger operator obtained from~$\mathcal L_\epsilon$ by the ground-state transformation.

\begin{lemma}\label{l:weighted-eigen-vanish-infty}
    Make the Assumptions~\ref{a:criticalpts}, \ref{a: nondegeneracy}, and~\ref{a:U-growth-control}. Let~$\epsilon, \nu > 0$. Denote the~$i$-th eigenfunction of~$\mathcal L_\epsilon$ as~$\psi_{i, \epsilon}$ and define
    \begin{equation}
        g_{i, \epsilon} = \exp(-2\nu U)\psi_{i, \epsilon}\,.    
    \end{equation}
        Then~$g_{i, \epsilon}(x) \to 0$ as~$\abs{x} \to \infty$.
\end{lemma}
\begin{proof}
Fix any~$\epsilon > 0$. It is well known that~$\phi_{i, \epsilon} = \exp(-U/2\epsilon)\psi_{i,\epsilon}$ satisfies
\begin{equation}\label{e:schrodinger-def}
    \mathcal H_\epsilon \phi_{i, \epsilon} = \frac{\lambda_{i, \epsilon}}{\epsilon} \phi_{i, \epsilon}\,,\quad\text{where}\quad \mathcal H_\epsilon = -\Delta + V_\epsilon\quad\text{and}\quad V_\epsilon = \frac{1}{4\epsilon^2}\abs{DU}^2 - \frac{1}{2\epsilon}\Delta U\,.
\end{equation}
The Agmon estimate (see, e.g.,~\cite{Agmon82}) states that for a Schr\"{o}dinger operator~$\mathcal H=-\Delta + V$ whose potential~$V$ satisfies~$V \to \infty$ as~$\abs{x}\to\infty$, if~$\mathcal H\phi = \lambda \phi$, then for every~$\delta\in (0,1)$ there exists a constant~$C_\delta > 0$ such that, for all~$x\in \R^d$, 
\begin{equation}\label{e:agmon-pointwise-ub}
    \abs{\phi(x)} \leq C_\delta \exp\paren*{-(1-\delta)\rho(x)}\,, 
\end{equation}
where
\begin{align}
    \rho(x) &= \inf_{y \colon V(y) \leq \lambda} d_\mathrm{A}(x,y)\,,\\
    d_\mathrm{A}(x,y) &= \inf_{\gamma\in\Gamma(x,y)}\int_0^1 \sqrt{\paren{V-\lambda}_+(\gamma(t))}\,\abs*{\dot \gamma(t)}\,dt\,.
\end{align}    
Here~$\Gamma(x,y) \defeq \set*{\gamma\in C^1([0,1]; \R^d) \st \gamma(0)=x\,,\ \gamma(1)=y}$ denotes the set of~$C^1$ paths from~$x$ to~$y$.
By the assumptions~\eqref{a:DU-blowup} and~\eqref{a:DU^2-dominates-D^2U}, we see that~$V_\epsilon \to \infty$ as~$\abs{x} \to \infty$.
Moreover, the same assumptions imply that there exists~$R_{i, \epsilon} > 0$ such that for all~$x$ with~$\abs{x} > R_{i, \epsilon}$,
\begin{equation}
    \frac{\abs*{\Delta U + 2\lambda_{i, \epsilon}}}{\abs{DU}^2} \leq \nu_\epsilon \defeq \nu \wedge \frac{1}{4\epsilon}\,,
\end{equation}
and hence
\begin{equation}
    V_\epsilon - \frac{\lambda_{i, \epsilon}}{\epsilon} = \frac{1}{4\epsilon^2}\abs{DU}^2\paren*{1 - 2\epsilon\frac{\Delta U + 2\lambda_{i, \epsilon}}{\abs{DU}^2}} \geq \frac{1}{4\epsilon^2}\abs{DU}^2\paren*{1 - 2\epsilon \nu_\epsilon} > 0\,.
\end{equation}

Fix~$x$ and~$y$ with~$\abs{x} > R_{i,\epsilon}$ and~$V_\epsilon(y) \leq \lambda_{i, \epsilon} / \epsilon$, and let~$\gamma$ be a~$C^1$ path with~$\gamma(0)=x$ and~$\gamma(1)=y$. The previous display shows that~$V_\epsilon > \lambda_{i,\epsilon}/\epsilon$ outside~$B(0,R_{i,\epsilon})$, so necessarily~$\abs{y} \leq R_{i, \epsilon}$. We may therefore let~$t_0 \in (0,1]$ be the first time at which~$\abs{\gamma(t_0)}=R_{i, \epsilon}$.
Then, for any~$t\in [0, t_0)$, $\abs{\gamma(t)} > R_{i, \epsilon}$ and
\begin{align}
    \int_0^1 \sqrt{\paren{V_\epsilon-\lambda_{i,\epsilon}/\epsilon}_+}\abs*{\dot \gamma(t)}dt &\geq \frac{(1-2\epsilon \nu_\epsilon)^\frac{1}{2}}{2\epsilon} \int_0^{t_0} \abs*{DU(\gamma(t))} \abs{\dot \gamma(t)} dt \\
    &\geq \frac{(1-2\epsilon \nu_\epsilon)^\frac{1}{2}}{2\epsilon} \abs*{\int_0^{t_0}  DU(\gamma(t)) \cdot \dot\gamma(t)\, dt}\\
    \label{e:agmon-geodesic-lb}
    &\geq \frac{(1-2\epsilon \nu_\epsilon)^\frac{1}{2}}{2\epsilon} \paren*{U(x) - \max_{\partial B(0,R_{i, \epsilon})}U}\,.
\end{align}
The lower bound~\eqref{e:agmon-geodesic-lb} holds for every admissible~$y$ and~$\gamma$, and hence bounds~$\rho(x)$ from below. Thus the Agmon estimate~\eqref{e:agmon-pointwise-ub} implies that for any~$\delta \in (0,1)$, there exists~$C_{\delta, i, \epsilon, \nu} >0$ such that for all~$x$ with~$\abs{x} > R_{i, \epsilon}$,
\begin{equation}
    \abs{\phi_{i,\epsilon} (x)} \leq C_{\delta, i, \epsilon, \nu}\exp\paren*{-(1-\delta)(1-2\epsilon \nu_\epsilon)^\frac{1}{2} \frac{U(x)}{2\epsilon}} \,,
\end{equation}
and hence
\begin{equation}\label{e:g-ie-agmon-bound}
    \abs{\exp(-2\nu U(x)) \psi_{i,\epsilon} (x) } \leq C_{\delta, i, \epsilon, \nu}\exp\paren*{ \eta U(x) } \,,
\end{equation}
where
\begin{equation}
    \eta = \frac{1}{2\epsilon}\paren*{1-(1-\delta)(1-2\epsilon \nu_\epsilon)^\frac{1}{2}} -2\nu   \,.
\end{equation}

We can find~$\delta(\epsilon) > 0$ such that~$\eta < 0$ if and only if
\begin{equation}\label{e:condition-for-negative-eta}
    1-4\epsilon \nu < (1-2\epsilon \nu_\epsilon)^\frac{1}{2}\,,
\end{equation}
which is satisfied for every~$\nu > 0$. Indeed, if~$\nu\epsilon \geq 1/4$ then the left-hand side of~\eqref{e:condition-for-negative-eta} is nonpositive while the right-hand side is positive, so the inequality is immediate. If instead~$\nu\epsilon < 1/4$, then~$\nu_\epsilon \leq \nu$ and~$16\epsilon^2\nu^2 < 6\epsilon\nu$ give
\begin{equation}
    1-2\epsilon\nu_\epsilon \geq 1-2\epsilon\nu > 1-8\epsilon\nu + 16\epsilon^2\nu^2 = (1-4\epsilon\nu)^2\,,
\end{equation}
which is~\eqref{e:condition-for-negative-eta}. Fix such a~$\delta=\delta(\epsilon)$, so that~$\eta<0$. Then~\eqref{e:g-ie-agmon-bound} reads~$\abs{g_{i,\epsilon}(x)}\leq C_{\delta, i, \epsilon, \nu}\exp(\eta U(x))$ for all~$\abs{x}>R_{i,\epsilon}$. Since~$U(x)\to\infty$ as~$\abs{x}\to\infty$ by~\eqref{e:U-coercive} above, and~$\eta<0$, the right-hand side tends to~$0$, and the claim follows.
\end{proof}

Before proving Lemma~\ref{l:lyapunov-weighted-Linftybds-and-their-sums} we collect some facts from the spectral theory of~$\mathcal L_\epsilon$, together with estimates on the partition function~$Z_\epsilon$. The first of these says that the eigenfunctions associated with the low-lying eigenvalues are \emph{almost} indicator functions of the basins, up to an error that is exponentially small in~$L^2(\pi_\epsilon)$ (Lemma~\ref{l:dEF-small}). The second says that~$\lambda_{2,\epsilon}$ is exponentially small in~$1/\epsilon$, whereas all the remaining eigenvalues stay bounded away from~$0$ uniformly in~$\epsilon$ (Lemma~\ref{l:eigenvalue-properties}). To preserve the flow of the argument, we defer the proofs of Lemmas~\ref{l:dEF-small} and~\ref{l:partition-function-estimates} to the appendix.

\begin{lemma}\label{l:dEF-small}
    Make the Assumptions~\ref{a:criticalpts}--\ref{a:U-growth-control}.
    For any~$\gamma < \hat \gamma$, there exists~$C_\gamma$ such that for any~$\epsilon\in (\kappa, 1]$, the following holds. 
    Let
    \begin{equation}\label{e:EeFe-def}
        E_\epsilon=\mathrm{span}\set{\one_{\Omega_1},\one_{\Omega_2}}\quad\text{and}\quad F_\epsilon=\mathrm{span}\set{1,\psi_{2,\epsilon}}\,,
    \end{equation}
    and let~$P_{E_\epsilon}, P_{F_\epsilon}$ denote the orthogonal projections onto~$E_\epsilon, F_\epsilon$ in~$L^2(\pi_\epsilon)$, respectively. Then, the distance~$d(E_\epsilon, F_\epsilon)$ between~$E_\epsilon$ and~$F_\epsilon$, defined as 
    \begin{equation}\label{e:dEeFe-def}
        d(E_\epsilon,F_\epsilon)
        =
        \norm{(I-P_{E_\epsilon})P_{F_\epsilon}}_{L^2(\pi_\epsilon)\to L^2(\pi_\epsilon)}\,,
    \end{equation} 
    satisfies
    \begin{equation}\label{e:ef-almost-indicators}
        d(E_\epsilon,F_\epsilon)
        \leq
        C_\gamma\exp\paren*{-\frac{\gamma}{\epsilon}}\,,
    \end{equation}
    and consequently,
        \begin{equation}\label{e:eigen-residual-exp-small}
        \norm{(I-P_{E_\epsilon})\psi_{2,\epsilon}}_{L^2(\pi_\epsilon)} \leq C_\gamma \exp\paren*{-\frac{\gamma}{\epsilon}}\,.
    \end{equation}

    Moreover, there exist constants~$a_{1,\epsilon}, a_{2,\epsilon}$ such that, for~$i\in\set{1,2}$,
    \begin{equation}\label{e:PEe-decompose}
        P_{E_\epsilon}\psi_{2,\epsilon}
        =
        a_{1,\epsilon}\one_{\Omega_1}
        +
        a_{2,\epsilon}\one_{\Omega_2}\quad\text{and}\quad
        \abs{a_{i,\epsilon}}
        \leq
        C_m\,.
    \end{equation}
\end{lemma}

The next lemma is a direct consequence of~\cite[Chapter 8, Proposition 2.2]{Kolokoltsov00}, which holds on~$\R^d$ under Assumption~\ref{a:U-growth-control}, so we do not reproduce its proof.

\begin{lemma}[{\cite[Chapter 8, Proposition 2.2]{Kolokoltsov00}}]\label{l:eigenvalue-properties}
    Make the Assumptions~\ref{a:criticalpts}--\ref{a:U-growth-control}. Then there exists~$\Lambda>0$ such that, for any~$\epsilon \in (\kappa, 1]$ and~$i\geq 3$, 
    \begin{equation}\label{e:third-O(1)-away}
        \lambda_{i, \epsilon} \geq \Lambda\,.
    \end{equation}
    Moreover, for any~$\gamma < \hat\gamma$, there exists~$\tilde C$ such that for any~$\epsilon \in (\kappa, 1]$,
    \begin{equation}\label{e:second-exp-small}
        \lambda_{2,\epsilon} \leq \tilde C \exp\paren*{-\frac{\gamma}{2\epsilon}}\,.
    \end{equation}
\end{lemma}

Finally, we need two-sided bounds on the partition function~$Z_\epsilon$.

\begin{lemma}\label{l:partition-function-estimates}
    Make the Assumptions~\ref{a:criticalpts}--\ref{a:U-growth-control}. Then, there exists a constant~$C_Z(d,U, C_m)$ such that for any~$\epsilon \in (\kappa, 1]$,
    \begin{equation}\label{e:Zepsilon-estimate}
        C_Z^{-1}\epsilon^\frac{d}{2} \leq Z_\epsilon \leq C_Z \epsilon^\frac{d}{2}\exp\paren*{-\frac{U(x_{\min, 2})}{\epsilon}}\,.
    \end{equation}
\end{lemma}

We now split Lemma~\ref{l:lyapunov-weighted-Linftybds-and-their-sums} into two parts, proved in Lemmas~\ref{l:eigenfunctions-weighted-Linfty-bdd} and~\ref{l:ef-weighted-sum-Linfty-bds} respectively.

Lemma~\ref{l:eigenfunctions-weighted-Linfty-bdd} is the first part of Lemma~\ref{l:lyapunov-weighted-Linftybds-and-their-sums} and the core of this section. Using the maximum principle, we reduce the bound on~$e^{-2\nu U}\psi_{i,\epsilon}$ to a bound on a sublevel set of~$\abs{DU}^2$, which for small~$\epsilon$ lies close to the critical points of~$U$. There we apply the local maximum principle together with the normalization of~$\psi_{i,\epsilon}$ in~$L^2(\pi_\epsilon)$. For~$i=2$ we also use that~$\psi_{2,\epsilon}$ is close to a combination of the indicator functions of the basins (Lemma~\ref{l:dEF-small}), and this yields a bound that is uniform in~$\epsilon$. For~$i\geq3$ the same argument yields a bound that grows like~$\exp(\bar C_\nu\lambda_{i,\epsilon}/\epsilon)$.

\begin{lemma}\label{l:eigenfunctions-weighted-Linfty-bdd}
Make the Assumptions~\ref{a:criticalpts}--\ref{a:U-growth-control}.
Fix any~$\nu \in (0, 1/8)$ and define~$V$ and~$\norm{\cdot}_V$ as in~\eqref{e:V-normV-def}.
Then, there exist constants~$\bar C_{\psi_2}, C_d$, and~$\bar C_\nu(\nu, \Lambda)$ such that for all~$\epsilon\in (\kappa,1]$ and~$i\geq 3$, we have
\begin{align}
\label{e:second-ef-weightedLinfty}
\norm*{\psi_{2, \epsilon}}_{V} &\leq \bar C_{\psi_2},\\
\label{e:higher-ef-weightedLinfty}
\norm*{\psi_{i, \epsilon}}_{V} &\leq C_d \exp\paren*{\bar C_\nu\lambda_{i, \epsilon}/\epsilon}.
\end{align}
\end{lemma}
\begin{proof}
Fix~$\nu \in (0,1/8)$ and~$\epsilon\in (\kappa, 1]$, and define
\begin{equation}
    g_{i,\epsilon} = \exp(-2\nu U)\psi_{i,\epsilon}.
\end{equation}
The proof consists of two steps. First, we use the maximum principle to reduce the problem to estimating~$g_{i,\epsilon}$ on a suitable sublevel set of~$\abs{DU}^2$. We then estimate this set separately in the cases~$i=2$ and~$i\geq3$. In each of these two cases we treat small and large~$\epsilon$ separately. We fix a threshold~$\tilde \epsilon \in (0,1]$ and prove the statement first for~$\epsilon \in (\kappa, \tilde \epsilon]$ and then for~$\epsilon \in (\tilde\epsilon, 1]$. We start with an arbitrary~$\tilde\epsilon \in (0,1]$ and decrease it finitely many times in the course of the proof.

We begin with the reduction step. A direct computation shows that~$g_{i,\epsilon}$ satisfies
\begin{equation}
    \epsilon\Delta g_{i,\epsilon} + \paren*{4\epsilon\nu-1} DU\cdot Dg_{i,\epsilon} + W_{i,\epsilon}g_{i,\epsilon} = 0\,,
\end{equation}
where
\begin{equation}
    W_{i,\epsilon} = 2\epsilon\nu \Delta U + 2\nu(2\epsilon\nu -1)\abs{DU}^2 + \lambda_{i,\epsilon}\,.
\end{equation}
By Lemma~\ref{l:weighted-eigen-vanish-infty} the function~$\abs{g_{i,\epsilon}}$ vanishes at infinity, so it attains its maximum at some point~$x_\ast\in\R^d$. At such a point~$Dg_{i,\epsilon}(x_\ast)=0$ and~$\Delta\abs{g_{i,\epsilon}}(x_\ast)\leq 0$, and the displayed equation forces~$W_{i,\epsilon}(x_\ast)\geq0$. (If~$g_{i,\epsilon}\equiv0$ there is nothing to prove.) Moreover, as pointed out in Remark~\ref{r:extension}, \eqref{a:DU^2-dominates-D^2U} implies~\eqref{a:DU-exp-growth-at-most}. Hence, if we choose~$\tilde \epsilon \leq 1/(4C_\Delta)$ then for all~$\epsilon\in (\kappa, \tilde\epsilon]$ and~$\nu \in(0,1/8)$,
we have
\begin{align}
    W_{i, \epsilon} &\overset{\eqref{a:DU-exp-growth-at-most}}{\leq} 2C_\Delta\epsilon\nu + 2\nu(C_\Delta\epsilon + 2\nu - 1)\abs{DU}^2 + \lambda_{i, \epsilon} \\
    &\leq 2C_\Delta \epsilon \nu + 2\nu\paren*{2\nu - \frac{3}{4}}\abs{DU}^2 + \lambda_{i, \epsilon}\\
    &\leq 2C_\Delta \epsilon\nu - \nu \abs{DU}^2 + \lambda_{i, \epsilon}\,,
\end{align}
where the first inequality also uses~$\epsilon\leq1$ and~$\abs{\Delta U}\leq d\abs{D^2U}$ (the factor~$d$ being absorbed into~$C_\Delta$), the second uses~$C_\Delta\epsilon\leq1/4$, and the third uses~$\nu\leq1/8$. This yields
\begin{equation}\label{e:maximum-set-in-DU^2-sublevel}
    \set*{W_{i,\epsilon}\geq0}
    \subset
    \set*{\abs*{DU}^2 \leq 2C_\Delta\epsilon+\frac{\lambda_{i,\epsilon}}{\nu}}
    \defeq A_\mathrm{small}.
\end{equation}
Moreover, for any~$\epsilon \in (\tilde\epsilon, 1]$, we use the fact that assumption~\eqref{a:DU^2-dominates-D^2U} implies the existence of some constant~$c \geq 1$ such that for all~$\abs{x} > c$, 
\begin{equation}
    \abs{\Delta U} \leq \frac{1}{4}\abs{DU}^2\,,
\end{equation}
so that on~$\set*{\abs{x} > c}$ we have, for any~$\nu \in (0, 1/8)$,
\begin{align}
    W_{i, \epsilon} &\leq 2\nu \abs*{\Delta U} + 2\nu (2\nu - 1) \abs{DU}^2 + \lambda_{i, \epsilon}\\
    &\leq 2\nu \paren*{2\nu - \frac{3}{4}}\abs{DU}^2 + \lambda_{i, \epsilon}\\
    &\leq -\nu \abs{DU}^2 + \lambda_{i, \epsilon}\,.
\end{align}
Thus, we have
\begin{equation}\label{e:maximum-set-in-DU^2-sublevel-big}
    \set*{W_{i, \epsilon} \geq 0} \subset B(0, c) \cup \set*{\abs*{DU}^2 \leq  \frac{\lambda_{i, \epsilon}}{\nu}} \defeq A_\mathrm{large} \,.
\end{equation}

We now bound~$g_{i, \epsilon}$ on~$A_\mathrm{small}$ and on~$A_\mathrm{large}$, for small and for large~$\epsilon$ respectively.
 
\medskip

\noindent\textbf{The case $i=2$, small~$\epsilon \in (\kappa, \tilde\epsilon]$.}
We prove the existence of a constant~$\bar C_{\psi_2}$ such that
\begin{equation}\label{e:small-epsilon-i2}
    \sup_{\epsilon \in (\kappa, \tilde\epsilon]} \norm{g_{2, \epsilon}}_{L^\infty(\R^d)} \leq \bar C_{\psi_2}\,.
\end{equation}
We begin with the observation that the set~$A_\mathrm{small}$ is contained in a small neighborhood of the critical points of~$U$. More precisely, we apply~\eqref{e:second-exp-small} at~$\gamma = \hat\gamma / 2$ to obtain that for every~$\epsilon\in(\kappa,1]$,
\begin{equation}\label{e:second-ev-exp-small}
\lambda_{2,\epsilon}
\leq
\tilde C\exp\paren*{-\hat\gamma/(4\epsilon)}
\leq
\frac{4\tilde C}{\hat\gamma}\epsilon.
\end{equation}
Combining this estimate with~\eqref{e:maximum-set-in-DU^2-sublevel} gives
\begin{equation}\label{e:A-subset-tildeA}
    A_\mathrm{small}
    \subset
    \set*{\abs{DU}^2 \leq \paren*{2C_\Delta + \frac{4\tilde C}{\nu\hat\gamma}}\epsilon}
    \defeq \tilde A\,.
\end{equation}
Hence, there exists a constant~$C_A>0$ such that, for sufficiently small~$\epsilon$,
\begin{equation}\label{e:tildeA-subset-crtpts}
    \tilde A \subset B_1 \cup B_2 \cup B_3\,,
\end{equation}
where
\begin{equation}
    B_1=B(x_{\min,1},C_A\sqrt\epsilon)\,,\quad
    B_2=B(x_{\min,2},C_A\sqrt\epsilon)\,,\quad
    B_3=B(0,C_A\sqrt\epsilon)\,.
\end{equation}

We first control~$\psi_{2,\epsilon}$ on~$B_1\cup B_2$. The proof of~\cite[Lemma 9.8, Equation (9.37)]{HanIyerEA26} extends verbatim from the compact setting~$\T^d$ to~$\R^d$, and therefore there exists a constant~$C_{\psi_2,1}$ such that
\begin{equation}\label{e:ef-B1B2-bound}
    \sup_{\epsilon\in (\kappa,\tilde \epsilon]}
    \norm{\psi_{2,\epsilon}}_{L^\infty(B_1\cup B_2)}
    \leq
    C_{\psi_2,1}\,.
\end{equation}
It therefore remains to obtain a uniform bound on~$B_3$.

To this end we apply the local maximum principle~\cite[Corollary 9.21]{GilbargTrudinger01}. Here and in each later application of that corollary, we write~$\lambda_*, \Lambda_*, \nu_*, \gamma_*, n_*, p_*$ and~$\Omega_*$ for the quantities denoted~$\lambda, \Lambda, \nu, \gamma, n, p$ and~$\Omega$ in~\cite[Chapter 9]{GilbargTrudinger01}, so as to keep them distinct from the notation of this paper. We take
\begin{gather}
    R=C_A\sqrt\epsilon\,,\quad
    \Omega_*=B(0,2C_A\sqrt\epsilon)\,,\quad
    \lambda_*=\Lambda_*=\epsilon\,,\quad
    \gamma_*=1\,,\quad
    n_*=p_*=2\,,\\
    \mathcal L=\mathcal L_\epsilon\,,\quad
    \nu_*=\norm{U}_{C^2(B(0,2))}/\epsilon\,,\quad
    u=\psi_{2,\epsilon}\,,\quad
    f=-\lambda_{2,\epsilon}\psi_{2,\epsilon}\,.
\end{gather}
This yields
\begin{equation}\label{e:B3-lmp}
    \norm{\psi_{2,\epsilon}}_{L^\infty(B_3)}
    \leq
    C\paren*{\epsilon^{-d/4}+\epsilon^{-1/2}\lambda_{2,\epsilon}}
    \norm{\psi_{2,\epsilon}}_{L^2(\Omega_*)}\,.
\end{equation}

It remains to estimate the~$L^2(\Omega_*)$ norm of~$\psi_{2, \epsilon}$. 
Using~\eqref{e:PEe-decompose}, we obtain
\begin{align}
    \norm{\psi_{2,\epsilon}}_{L^2(\Omega_*)}
    &\leq
    \norm{(I-P_{E_\epsilon})\psi_{2,\epsilon}}_{L^2(\Omega_*)}
    +
    \norm{P_{E_\epsilon}\psi_{2,\epsilon}}_{L^2(\Omega_*)}\\
    \label{e:Omega-L^2-decompose}
    & \leq 
    \norm{(I-P_{E_\epsilon})\psi_{2,\epsilon}}_{L^2(\Omega_*)}
    +
    C_m\abs{\Omega_*}^{1/2}\,.
\end{align}
Moreover, 
\begin{align*}
    \norm{(I-P_{E_\epsilon})\psi_{2,\epsilon}}_{L^2(\Omega_*)}^2
    &\leq
    Z_\epsilon
    \exp\paren*{\frac{1}{\epsilon}\max_{\Omega_*} U}
    \norm{(I-P_{E_\epsilon})\psi_{2,\epsilon}}_{L^2(\pi_\epsilon)}^2
    \\
    &\leq
    C_Z \epsilon^{d/2}
    \exp\paren*{\frac{U(0)-U(x_{\min,2})}{\epsilon}}
    \norm{(I-P_{E_\epsilon})\psi_{2,\epsilon}}_{L^2(\pi_\epsilon)}^2\\
    \refstepcounter{equation}\tag{\theequation}\label{e:orthogonal-estimate}
    &\overset{\eqref{e:gammaHatDef}}{=}
    C_Z \epsilon^{d/2}
    \exp\paren*{\frac{\hat \gamma}{\epsilon}}
    \norm{(I-P_{E_\epsilon})\psi_{2,\epsilon}}_{L^2(\pi_\epsilon)}^2
\end{align*}
where the second inequality uses the estimate for~$Z_\epsilon$ in~\eqref{e:Zepsilon-estimate}, together with the Taylor expansion
\begin{equation}
    U(x)
    \leq
    U(0)+2C_A^2\norm{U}_{C^2(B(0,1))}\epsilon\,,
    \qquad x\in\Omega_*.
\end{equation}
Choosing~$\gamma = 3\hat\gamma/4$ in Lemma~\ref{l:dEF-small} and combining~\eqref{e:eigen-residual-exp-small} with~\eqref{e:Omega-L^2-decompose} and~\eqref{e:orthogonal-estimate}, we obtain
\begin{equation}
    \norm{\psi_{2,\epsilon}}_{L^2(\Omega_*)}
    \leq
    C\epsilon^{d/4}
    \paren*{1+\exp\paren*{-\frac{\hat\gamma}{4\epsilon}}}.
\end{equation}
Substituting this estimate into~\eqref{e:B3-lmp}, together with~\eqref{e:second-ev-exp-small}, yields the existence of a constant~$C_{\psi_2,2}$ such that
\begin{equation}\label{e:ef-B3-bound}
    \sup_{\epsilon\in(\kappa,\tilde\epsilon]}
    \norm{\psi_{2,\epsilon}}_{L^\infty(B_3)}
    \leq
    C_{\psi_2,2}.
\end{equation}

Finally, by~\eqref{e:maximum-set-in-DU^2-sublevel}, \eqref{e:A-subset-tildeA}, and~\eqref{e:tildeA-subset-crtpts}, together with~\eqref{e:ef-B1B2-bound}, \eqref{e:ef-B3-bound}, and the non-negativity of~$U$, we obtain
\begin{equation}
    \norm{g_{2,\epsilon}}_{L^\infty(\R^d)}
    \leq
    \norm{g_{2,\epsilon}}_{L^\infty(B_1\cup B_2\cup B_3)}
    \leq
    \norm{\psi_{2,\epsilon}}_{L^\infty(B_1\cup B_2\cup B_3)}
    \leq
    C_{\psi_2,1}\vee C_{\psi_2,2}
    \defeq
    \bar C_{\psi_2}.
\end{equation}
This completes the proof of~\eqref{e:small-epsilon-i2}.
\medskip

\noindent\textbf{The case $i=2$, large~$\epsilon \in (\tilde\epsilon, 1]$.}
We prove the existence of a constant~$\bar C_{\psi_2}$ such that
\begin{equation}\label{e:big-epsilon-i2}
    \sup_{\epsilon \in (\tilde\epsilon, 1]} \norm{g_{2, \epsilon}}_{L^\infty(\R^d)} \leq \bar C_{\psi_2}\,,
\end{equation}
by bounding~$g_{2, \epsilon}$ on~$A_\mathrm{large}$.
By the definition~\eqref{e:maximum-set-in-DU^2-sublevel-big} of~$A_\mathrm{large}$ and the bound~\eqref{e:second-ev-exp-small}, which gives~$\lambda_{2,\epsilon}\leq\tilde C e^{-\hat\gamma/(4\epsilon)}\leq4\tilde C/\hat\gamma$ for~$\epsilon\leq1$, the set~$A_\mathrm{large}$ is contained in~$B(0,c)\cup\set{\abs{DU}^2\leq4\tilde C/(\nu\hat\gamma)}$, and this set is bounded by~\eqref{a:DU-blowup}. Hence there is an~$\epsilon$-independent constant~$R_\nu \geq 1$ such that
\begin{equation}\label{e:Abig-O(1)-ball-inclusion}
    A_\mathrm{large} \subset B(0,c) \cup \set*{\abs{DU}^2 \leq \frac{4\tilde C}{\nu\hat\gamma}} \subset B(0, R_{\nu})\,.
\end{equation}

For any~$x\in B(0, R_\nu)$, we apply the local maximum principle~\cite[Corollary 9.21]{GilbargTrudinger01} with 
\begin{gather}
    R=\epsilon\,,\quad 
    B = B(x, 2\epsilon)\,, \quad
    \Omega_*=B(0,3R_\nu)\,,\quad
    \lambda_*=\Lambda_*=\epsilon\,,\quad
    \gamma_*=1\,,\quad
    n_*=p_*=2\,,\\
    \mathcal L=\mathcal L_\epsilon\,,\quad
    \nu_*=\norm{U}_{C^1(\Omega_*)}^2/\epsilon^2\,,\quad
    u=\psi_{2,\epsilon}\,,\quad
    f=-\lambda_{2,\epsilon}\psi_{2,\epsilon}\,,
\end{gather}
and obtain
\begin{align}
    \abs*{\psi_{2,\epsilon}(x)}^2 &\leq C\paren*{\epsilon^{-d} + \lambda_{2, \epsilon}^2 } \norm{\psi_{2, \epsilon}}_{L^2(B)}^2\\
    &\overset{\eqref{e:second-ev-exp-small}}{\leq} C(\tilde \epsilon^{-d} + 1)\norm{\psi_{2, \epsilon}}_{L^2(B)}^2\\
    &\leq C Z_\epsilon\exp\paren*{ \frac{1}{\epsilon} \max_{\Omega_*} U }  (\tilde \epsilon^{-d} + 1) \int_{B}\psi_{2, \epsilon}^2 \pi_\epsilon(x) \, dx\\
    &\leq CZ_1 \exp\paren*{\frac{1}{\tilde \epsilon} \max_{\Omega_*} U}(\tilde \epsilon^{-d} + 1) = C(\tilde\epsilon, \nu)\,.
\end{align}
We chose arbitrary~$x\in B(0, R_\nu)$ so combining this with~\eqref{e:maximum-set-in-DU^2-sublevel-big} and~\eqref{e:Abig-O(1)-ball-inclusion} yields~\eqref{e:big-epsilon-i2}.

\noindent\textbf{The case $i\geq 3$, small~$\epsilon \in (\kappa, \tilde\epsilon]$.}
We prove the existence of constants~$C_d, \bar C_\nu > 0$ such that for any~$\epsilon \in (\kappa, \tilde \epsilon]$, 
\begin{equation}\label{e:small-epsilon-i3}
    \norm{g_{i, \epsilon}}_{L^\infty(\R^d)} \leq C_d\exp(\bar C_\nu \lambda_{i, \epsilon} / \epsilon)\,.
\end{equation}

As in the case of~$i = 2$ and small~$\epsilon$, we prove~\eqref{e:small-epsilon-i3} by bounding~$g_{i, \epsilon}$ on the set~$A_\mathrm{small}$.
First, for~$\epsilon\leq 1/(2C_\Delta)$, define
\begin{equation}\label{e:Rie-A1-A2-def}
    R_{i,\epsilon}\defeq 1+\frac{\lambda_{i,\epsilon}}{\nu},
    \qquad
    A_1\defeq\set{\abs{DU}^2\leq R_{i,\epsilon}},
    \qquad
    A_2\defeq\set{\abs{DU}^2\leq16R_{i,\epsilon}}.
\end{equation}
Then~\eqref{e:maximum-set-in-DU^2-sublevel} gives
\begin{equation}
    A_\mathrm{small}\subset A_1\subset A_2.
\end{equation}
We next show that, for every~$x\in A_\mathrm{small}$, a ball of radius~$\bar R_{i,\epsilon}$ around~$x$ is contained in~$A_2$, where~$\bar R_{i, \epsilon}$ shall be chosen soon.

To this end, we first record a consequence of~\eqref{a:DU-exp-growth-at-most}. For any~$x,y\in\R^d$,
\begin{equation}\label{e:DU-tan-growth}
    \tan^{-1}\paren*{\abs{DU(y)}}
    \leq
    \tan^{-1}\paren*{\abs{DU(x)}} + C_\Delta\abs{y-x}.
\end{equation}
Indeed, if~$v\in S^{d-1}$ and~$f(t)=\abs{DU(x+tv)}$, then
\begin{equation}
    \abs{f'(t)}
    \leq
    \abs{D^2U(x+tv)}
    \overset{\eqref{a:DU-exp-growth-at-most}}{\leq}
    C_\Delta\paren*{1+f(t)^2},
\end{equation}
and solving this differential inequality yields~\eqref{e:DU-tan-growth} with the choices~$v=(y-x)/\abs{y-x}$ and~$t=\abs{y-x}$.

Choose~$c'>0$ such that
\begin{equation}\label{e:c-def}
    \inf_{x\in[1,\infty)}
    x\paren*{\tan^{-1}(4x)-\tan^{-1}(x)}
    \geq c',
\end{equation}
and decrease~$\epsilon$ further if necessary so that~$\epsilon\leq c'/C_\Delta$. Define
\begin{equation}
    \bar R_{i,\epsilon}
    =
    \frac{\epsilon}{2\sqrt{R_{i,\epsilon}}}.
\end{equation}
Let~$x\in A_1$ and~$y\in B(x,2\bar R_{i,\epsilon})$. Then~\eqref{e:DU-tan-growth} gives
\begin{equation}
    \tan^{-1}\paren*{\abs{DU(y)}}
    \overset{\eqref{e:DU-tan-growth}}{\leq}
    \tan^{-1}\paren*{\sqrt{R_{i,\epsilon}}}
    +\frac{C_\Delta\epsilon}{\sqrt{R_{i,\epsilon}}}
    \overset{\eqref{e:c-def},\,R_{i,\epsilon}\geq1}{\leq}
    \tan^{-1}\paren*{4\sqrt{R_{i,\epsilon}}}.
\end{equation}
Therefore~$y\in A_2$, and hence
\begin{equation}
    B(x,2\bar R_{i,\epsilon})\subset A_2\,,
    \qquad\text{for every }x\in A_\mathrm{small}\,.
\end{equation}

Now we are ready to apply the local maximum principle. Fixing~$x\in A_\mathrm{small}$ and applying~\cite[Corollary 9.21]{GilbargTrudinger01} with
\begin{gather}
    R=\bar R_{i,\epsilon}\,,\quad B=B(x,2R)\,,\quad \Omega_*=A_2\,,\quad \lambda_*=\Lambda_*=\epsilon\,,\quad \gamma_*=1\,,\quad n_*=p_*=2\,,\\
    \mathcal L=\mathcal L_\epsilon\,,\quad \nu_*=16R_{i,\epsilon}/\epsilon^2\,,\quad
    u=\psi_{i,\epsilon}\,,\quad
    f=-\lambda_{i,\epsilon}\psi_{i,\epsilon}\,,
\end{gather}
we obtain, for some~$\epsilon$-independent constants~$C_d$ and~$\bar C_\nu(\nu,\Lambda)$ (in the final step we use~\eqref{e:third-O(1)-away} in the form~$1\leq\lambda_{i,\epsilon}/\Lambda$ to absorb the factors~$\epsilon^{-d}$, $R_{i,\epsilon}^{d/2}$ and~$\exp(16C_L/\epsilon)$ into~$\exp(\bar C_\nu\lambda_{i,\epsilon}/\epsilon)$),
\begin{align}
    \abs{\psi_{i,\epsilon}(x)}^2
    &\leq
    C_d\paren*{\frac{1}{\abs{B}}+\frac{\lambda_{i,\epsilon}^2}{4R_{i,\epsilon}}}
    \int_B\abs{\psi_{i,\epsilon}(y)}^2\,dy\\
    &\leq
    C_d Z_\epsilon
    \exp\paren*{\frac{1}{\epsilon}\max_B U}
    \paren*{\frac{1}{\abs{B}}+\frac{\nu\lambda_{i,\epsilon}}{4}}
    \int_B\abs{\psi_{i,\epsilon}(y)}^2\pi_\epsilon(y)\,dy\\
    &\overset{\eqref{a:U-lyapunov},\,B\subset\Omega_*,\,\eqref{e:Zepsilon-estimate}}{\leq}
    C_d\epsilon^{\frac d2}
    \exp\paren*{\frac{C_L}{\epsilon}\paren*{1+16R_{i,\epsilon}}}
    \paren*{\epsilon^{-d}R_{i,\epsilon}^{\frac d2}
    +\frac{\nu\lambda_{i,\epsilon}}{4}}\\
    &\overset{\eqref{e:third-O(1)-away}}{\leq}
    C_d\exp\paren*{\bar C_\nu\frac{\lambda_{i,\epsilon}}{\epsilon}}\,.
\end{align}

Since~$x\in A_\mathrm{small}$ was arbitrary, we conclude that
\begin{equation}
    \norm*{\psi_{i,\epsilon}}_{L^\infty(A_\mathrm{small})}
    \leq
    C_d\exp\paren*{\bar C_\nu\frac{\lambda_{i,\epsilon}}{\epsilon}}.
\end{equation}
Finally, because~$U\geq0$, we have~$\abs{g_{i,\epsilon}}\leq\abs{\psi_{i,\epsilon}}$, and therefore
\begin{equation}
    \norm*{g_{i,\epsilon}}_{L^\infty(\R^d)}
    \leq
    \norm*{g_{i,\epsilon}}_{L^\infty(A_\mathrm{small})}
    \leq
    \norm*{\psi_{i,\epsilon}}_{L^\infty(A_\mathrm{small})}
    \leq
    C_d\exp\paren*{\bar C_\nu\frac{\lambda_{i,\epsilon}}{\epsilon}}.
\end{equation}
This completes the proof for~\eqref{e:small-epsilon-i3}.

\noindent\textbf{The case $i\geq 3$, large~$\epsilon \in (\tilde \epsilon, 1]$.}
We prove the existence of constants~$C_d, \bar C_\nu > 0$ such that for any~$\epsilon \in (\tilde \epsilon, 1]$, 
\begin{equation}\label{e:big-epsilon-i3}
    \norm{g_{i, \epsilon}}_{L^\infty(\R^d)} \leq C_d\exp(\bar C_\nu \lambda_{i, \epsilon} / \epsilon)\,.
\end{equation}

As in the case of~$i = 2$ and large~$\epsilon$, we prove~\eqref{e:big-epsilon-i3} by bounding~$g_{i, \epsilon}$ on the set~$A_\mathrm{large}$. In particular, we bound~$g_{i, \epsilon}$ on~$B(0, c)$ and~$A_0\defeq \set*{\abs{DU}^2 \leq \lambda_{i, \epsilon} / \nu}$, separately.

For any~$x\in B(0, c)$, we apply the local maximum principle~\cite[Corollary 9.21]{GilbargTrudinger01} with 
\begin{gather}
    R=\epsilon\,,\quad 
    B = B(x, 2\epsilon)\,, \quad
    \Omega_*=B(0, 3c)\,,\quad
    \lambda_*=\Lambda_*=\epsilon\,,\quad
    \gamma_*=1\,,\quad
    n_*=p_*=2\,,\\
    \mathcal L=\mathcal L_\epsilon\,,\quad
    \nu_*=\norm{U}_{C^1(\Omega_*)}^2/\epsilon^2\,,\quad
    u=\psi_{i,\epsilon}\,,\quad
    f=-\lambda_{i,\epsilon}\psi_{i,\epsilon}\,,
\end{gather}
and obtain
\begin{align}
    \abs*{\psi_{i,\epsilon}(x)}^2 &\leq C\paren*{\epsilon^{-d} + \lambda_{i, \epsilon}^2 } \norm{\psi_{i, \epsilon}}_{L^2(B)}^2\\
    &\leq C(\tilde \epsilon^{-d} + \lambda_{i, \epsilon}^2)\norm{\psi_{i, \epsilon}}_{L^2(B)}^2\\
    &\leq C Z_\epsilon\exp\paren*{ \frac{1}{\epsilon} \max_{\Omega_*} U }  (\tilde \epsilon^{-d} + \lambda_{i, \epsilon}^2) \int_{B}\psi_{i, \epsilon}^2 \pi_\epsilon(x) \, dx\\
    &\leq CZ_1 \exp\paren*{\frac{1}{\tilde \epsilon} \max_{\Omega_*} U}(\tilde \epsilon^{-d} + \lambda_{i, \epsilon}^2) \\
    \label{e:psi-ie-bound-on-O(1)-ball}
    &\overset{\eqref{e:third-O(1)-away}}{\leq} C(\tilde\epsilon, \Lambda)\lambda_{i, \epsilon}^2\,.
\end{align}

Now we bound~$g_{i, \epsilon}$ on~$A_0$. Define~$R_{i, \epsilon}, A_1, A_2$ as in~\eqref{e:Rie-A1-A2-def} and we see that
\begin{equation}
    A_0 \subset A_1 \subset A_2\,.
\end{equation}
We first note that for any~$x\in A_0$, it holds that $B\paren*{x, 2\tilde R_{i, \epsilon}} \subset A_2$, where
\begin{equation}
    \tilde R_{i, \epsilon} = \frac{\tilde\epsilon}{2\sqrt{R_{i, \epsilon}}}\,.
\end{equation} 
Indeed, after decreasing~$\tilde \epsilon$ if necessary so that~$\tilde \epsilon \leq  c'/C_\Delta$ where~$c'$ is defined as in~\eqref{e:c-def}, we observe that for any~$x\in A_1$ and~$y \in B(x, 2\tilde R_{i, \epsilon})$,  
\begin{equation}
    \tan^{-1}\paren*{\abs{DU(y)}}
    \overset{\eqref{e:DU-tan-growth}}{\leq}
    \tan^{-1}\paren*{\sqrt{R_{i,\epsilon}}}
    +\frac{C_\Delta\tilde\epsilon}{\sqrt{R_{i,\epsilon}}}
    \overset{\eqref{e:c-def},\,R_{i,\epsilon}\geq1}{\leq}
    \tan^{-1}\paren*{4\sqrt{R_{i,\epsilon}}}\,,
\end{equation}
and hence, $y\in A_2$. Moreover, we also note that on~$A_2$, 
\begin{equation}
    \frac{\abs{DU}^2}{\epsilon^2} \leq \frac{16R_{i, \epsilon}}{\epsilon^2} \leq \frac{16R_{i, \epsilon}}{\tilde\epsilon^2}\,.
\end{equation}
Now we are ready to apply the local maximum principle.  Fixing~$x\in A_0$ and applying~\cite[Corollary 9.21]{GilbargTrudinger01} with
\begin{gather}
    R=\tilde R_{i,\epsilon}\,,\quad B=B(x,2R)\,,\quad \Omega_*=A_2\,,\quad \lambda_*=\Lambda_*=\epsilon\,,\quad \gamma_*=1\,,\quad n_*=p_*=2\,,\\
    \mathcal L=\mathcal L_\epsilon\,,\quad \nu_*=16R_{i,\epsilon}/\tilde\epsilon^2\,,\quad
    u=\psi_{i,\epsilon}\,,\quad
    f=-\lambda_{i,\epsilon}\psi_{i,\epsilon}\,,
\end{gather}
we obtain, for some~$\epsilon$-independent constants~$C_d$ and~$\bar C_\nu(\nu,\Lambda, \tilde\epsilon)$,
\begin{align}
    \abs{\psi_{i,\epsilon}(x)}^2
    &\leq
    C_d\paren*{\frac{1}{\abs{B}}+\frac{\lambda_{i,\epsilon}^2}{4R_{i,\epsilon}} \paren*{\frac{\tilde\epsilon}{\epsilon}}^2}
    \int_B\abs{\psi_{i,\epsilon}(y)}^2\,dy\\
    &\overset{\tilde\epsilon \leq \epsilon}{\leq}
    C_d Z_\epsilon
    \exp\paren*{\frac{1}{\epsilon}\max_B U}
    \paren*{\frac{1}{\abs{B}}+\frac{\nu\lambda_{i,\epsilon}}{4}}
    \int_B\abs{\psi_{i,\epsilon}(y)}^2\pi_\epsilon(y)\,dy\\
    &\overset{\eqref{a:U-lyapunov},\,B\subset\Omega_*}{\leq}
    C_d Z_1
    \exp\paren*{\frac{C_L}{\epsilon}\paren*{1+16R_{i,\epsilon}}}
    \paren*{\tilde\epsilon^{-d}R_{i,\epsilon}^{\frac d 2}
    +\frac{\nu\lambda_{i,\epsilon}}{4}}\\
    \label{e:psi-ie-bound-on-sublevel}
    &\overset{\eqref{e:third-O(1)-away}}{\leq}
    C_d\exp\paren*{\bar C_\nu\frac{\lambda_{i,\epsilon}}{\epsilon}}.
\end{align}
Combining~\eqref{e:maximum-set-in-DU^2-sublevel-big}, \eqref{e:psi-ie-bound-on-O(1)-ball}, and~\eqref{e:psi-ie-bound-on-sublevel} yields~\eqref{e:big-epsilon-i3}, which completes the proof.

\end{proof}

The second part of Lemma~\ref{l:lyapunov-weighted-Linftybds-and-their-sums} follows from Lemma~\ref{l:eigenfunctions-weighted-Linfty-bdd} once we count the eigenvalues of~$\mathcal L_\epsilon$. Since these eigenvalues are, up to the factor~$\epsilon$, those of the Schr\"odinger operator obtained by the ground-state transformation, we bound their number by the Cwikel--Lieb--Rozenblum inequality when~$d\geq3$ and by the Lieb--Thirring inequality when~$d=1,2$.

\begin{lemma}\label{l:ef-weighted-sum-Linfty-bds}
Make the Assumptions~\ref{a:criticalpts}--\ref{a:U-growth-control}. Fix any~$\nu \in (0, 1/8)$ and define~$V$ and~$\norm{\cdot}_V$ as in~\eqref{e:V-normV-def}. Then, for any~$\Lambda_1,T_0>0$, there exists a constant~$\tilde C_\psi$ such that the following holds. Let~$\epsilon\in (\kappa, 1]$, and suppose that for some~$m\in\N$,
\begin{equation}\label{e:lambda-large-from-mth}
    \lambda_{i,\epsilon}\geq\Lambda_1\,,\qquad \forall i\geq m.
\end{equation}
Then, for every~$T$ satisfying
\begin{equation}\label{e:T-large-for-sum-bound}
    T\geq \frac{2\bar C_\nu}{\epsilon}+T_0,
\end{equation}
we have
\begin{equation}\label{e:S-general-bound}
    \sum_{i=m}^\infty
    \exp\paren*{-\lambda_{i,\epsilon}T}
    \norm{\psi_{i,\epsilon}}_V^2
    \leq
    \tilde C_\psi
    \epsilon^{-1-\frac d2}
    \exp\paren*{-\Lambda_1\paren*{T-\frac{2\bar C_\nu}{\epsilon}}}\,.
\end{equation}
\end{lemma}

\begin{proof}
    Let~$\nu, \Lambda_1, T_0, \epsilon, m, T$ be as in the statement of the lemma. 
    By Lemma~\ref{l:eigenfunctions-weighted-Linfty-bdd} and~\eqref{e:T-large-for-sum-bound}, we have
    \begin{align}
        \sum_{i=m}^\infty \exp\paren*{-\lambda_{i, \epsilon}T} \norm{\psi_{i, \epsilon}}_V^2 &\leq \sum_{i \colon \lambda_{i, \epsilon} \geq \Lambda_1} \exp\paren*{-\lambda_{i, \epsilon}T} \norm{\psi_{i, \epsilon}}_V^2 \\
        &\overset{\eqref{e:second-ef-weightedLinfty}, \eqref{e:higher-ef-weightedLinfty}}{\leq} \max\set{\bar C_{\psi_2}, C_d}^2\sum_{i \colon \lambda_{i, \epsilon} \geq \Lambda_1} \exp\paren*{\lambda_{i, \epsilon}\paren*{\frac{2\bar C_\nu}{\epsilon} - T   }} \\
        \label{e:weighted-ef-sum-to-weyl}
        &\overset{\eqref{e:T-large-for-sum-bound}}{\leq} \max\set{\bar C_{\psi_2}, C_d}^2\exp\paren*{-\Lambda_1 \paren*{T-\frac{2\bar C_\nu}{\epsilon} - T_0 } }\\
        &\qquad\qquad\qquad\qquad\qquad \times \int_{\Lambda_1}^\infty \exp(-\lambda T_0) N_\epsilon(d\lambda)\,,
    \end{align}
    where~$N_\epsilon(d\lambda)$ denotes the measure induced by the eigenvalue counting function
    \begin{equation}
        N_\epsilon(\lambda)
        \defeq
        \#\set*{i\geq1 \colon \lambda_{i,\epsilon}\leq\lambda}\,.
    \end{equation}
    Moreover, if for any~$t\in (0,1]$, we define~$\lambda_t > 0$ as the constant such that
    \begin{equation}\label{e:lambdat-def}
        \set*{\lambda \geq 0 \st \exp(-\lambda T_0) \geq t} = [0, \lambda_t] \,,
    \end{equation}
    then
    \begin{align}\label{e:weyl-int-ub}
        \int_{\Lambda_1}^\infty \exp(-\lambda T_0) N_\epsilon(d\lambda) = \int_0^1 N_\epsilon \paren*{[\Lambda_1, \lambda_t]} dt \leq \int_0^1  \one_{\set*{\lambda_t \geq \Lambda_1}} N_\epsilon(\lambda_t)dt \,.
    \end{align}
    Now we estimate~$N_\epsilon(\lambda)$ for any~$\lambda \geq \Lambda_1$.
    If we define~$N(\lambda, \mathcal L)$ as the Weyl counting measure of a general self-adjoint operator~$\mathcal L$ 
    \begin{equation}
        N(\lambda, \mathcal L) = \#\set*{i \geq 1: \lambda_i (\mathcal L) \leq \lambda} = \sum_{\lambda_i(\mathcal L) \leq \lambda} \one\,,
    \end{equation}
    then the property~\eqref{e:schrodinger-def} implies that for any~$\lambda \in \R$,
    \begin{equation}\label{e:weyl-translations}
        N_\epsilon (\lambda) = N(\lambda, \mathcal L_\epsilon) = N(\lambda / \epsilon, \mathcal H_\epsilon) = N (0, \mathcal H_\epsilon - (\lambda / \epsilon ) I ) = N(0, -\Delta - \bar V_{\epsilon, \lambda})\,,
    \end{equation}
    where
    \begin{equation}
        \bar V_{\epsilon,\lambda} \defeq \lambda  / \epsilon -V_\epsilon\,,
    \end{equation}
    and~$\mathcal H_\epsilon, V_\epsilon$ are defined as in~\eqref{e:schrodinger-def}.
    Applying the Cwikel--Lieb--Rozenblum inequality~\cite[Theorem 4.31]{FrankLaptevEA23} yields that for any~$d \geq 3$, there exists a constant~$L_{0, d}$ such that for any~$\lambda \geq 0$,
    \begin{equation}\label{e:negative-eigenvalue-bound}
        N(0, -\Delta - \bar V_{\epsilon, \lambda})  \leq L_{0, d} \int_{\R^d}  \paren*{\bar V_{\epsilon, \lambda}}_+^\frac{d}{2}dx \,.
    \end{equation}
    Now, we estimate the integration on the right hand side of the inequality, for any given~$\lambda \geq \Lambda_1$.
    We first note that 
    \begin{equation}\label{e:Vbar-positive-set-simplify}
        \set*{\bar V_{\epsilon, \lambda} \geq 0}=\set*{\epsilon V_\epsilon \leq \lambda} = \set*{\frac{1}{4\epsilon}\abs{DU}^2 - \frac{1}{2}\Delta U \leq \lambda } \,.
    \end{equation}
    Then, by the assumption~\eqref{a:DU^2-dominates-D^2U}, there exists~$c>0$ such that
    \begin{equation}
        \abs{\Delta U} \leq \frac{1}{4}\abs{DU}^2 \quad\text{on}\quad \set*{\abs{x} > c}\,,
    \end{equation}
    and combining this with~$\epsilon \leq 1$, $\lambda \geq \Lambda_1$, and~\eqref{a:U-lyapunov} yields
    \begin{align}
        \set*{\frac{1}{4\epsilon}\abs{DU}^2 - \frac{1}{2}\Delta U \leq \lambda } &\subset B(0,c) \cup \set*{\frac{1}{8}\abs{DU}^2 \leq \lambda}  \\
        &\subset \set*{\abs{DU}^2 \leq \max\set*{8\lambda, \norm{U}_{C^1(B(0,c))}^2}}\\
        \label{e:barV-positive-in-DU-growth}
        &\overset{\lambda \geq \Lambda_1}{\subset} \set*{\abs{DU}^2 \leq C(\Lambda_1) \lambda}\\
        \label{e:confining-U-set}
        &\overset{\eqref{a:U-lyapunov}, \lambda \geq \Lambda_1}{\subset} \set*{ U \leq C(\Lambda_1)\lambda}\,,
    \end{align}
    where~$C(\Lambda_1)$ is a constant depending on~$\Lambda_1$ that might vary line by line. 
    Combining the coercivity estimate~\eqref{e:U-coercive} with~\eqref{e:Vbar-positive-set-simplify} and~\eqref{e:confining-U-set} yields that for any~$\lambda \geq \Lambda_1$,
    \begin{equation}\label{e:confining-U-in-a-ball}
        \set*{\bar V_{\epsilon, \lambda} \geq 0 } \subset \set*{ \abs{x} \leq C(\Lambda_1) \lambda}\,.
    \end{equation}
    Moreover, on the set~$\set*{\bar V_{\epsilon, \lambda} \geq 0}$, using~\eqref{a:DU-exp-growth-at-most} and~$\lambda \geq \Lambda_1$ yields
    \begin{align}
        \paren*{\bar V_{\epsilon, \lambda}}_+ = \bar V_{\epsilon, \lambda} = \frac{2\lambda + \Delta U}{2\epsilon} - \frac{1}{4\epsilon^2}\abs{DU}^2 &\overset{\eqref{a:DU-exp-growth-at-most}}{\leq} \frac{2\lambda + C_\Delta\paren*{1+\abs{DU}^2}}{2\epsilon} \\
        \label{e:barV-ub-on-positive-set}
        &\overset{\eqref{e:barV-positive-in-DU-growth}}{\leq} C(\Lambda_1) \frac{\lambda}{\epsilon}\,.
    \end{align}
    Combining this with~\eqref{e:weyl-translations}, \eqref{e:negative-eigenvalue-bound}, and~\eqref{e:confining-U-in-a-ball}, we obtain that for any~$\lambda \geq \Lambda_1$,
    \begin{equation}\label{e:weyl-higher-d-ub}
        N_\epsilon(\lambda) \leq C(\Lambda_1) L_{0, d} \epsilon^{-\frac{d}{2}} \lambda^\frac{3d}{2}\,.
    \end{equation}
    For~$d=1, 2$, we apply the Lieb--Thirring inequality~\cite[Theorem 4.38]{FrankLaptevEA23} to deduce that there exists a constant~$L_{1, d}$ such that for any~$\lambda \geq 0$,
    \begin{align}\label{e:Tr-bound-LT}
        \tr(-\Delta - \bar V_{\epsilon, \lambda})_- \leq L_{1, d} \int_{\R^d} \paren*{\bar V_{\epsilon, \lambda}}_+^{1+\frac{d}{2}} dx \,.
    \end{align}
    Here we use the notation 
    \begin{equation}
        \tr(A)_- = \sum_{j} \paren*{\mu_j}_- \,,
    \end{equation}
    where~$A$ is a self-adjoint, lower semibounded operator, $(\mu_j)$ denotes its negative eigenvalues, and~$(x)_- \defeq \max\set{-x,0}$.
    Applying~\cite[Lemma 1.39]{FrankLaptevEA23}, we obtain
    \begin{align}
        \tr(-\Delta - \bar V_{\epsilon, \lambda})_- &= \int_0^\infty N(-\tau, -\Delta - \bar V_{\epsilon, \lambda}) d \tau \geq \int_{\epsilon^{-1}-1}^{\epsilon^{-1}} N(-\tau, -\Delta - \bar V_{\epsilon, \lambda}) d \tau \\
        \label{e:Tr-weyl-relation}
        &\geq N(-\epsilon^{-1}, -\Delta - \bar V_{\epsilon, \lambda}) = N(0, -\Delta - \bar V_{\epsilon, \lambda - 1}) \overset{\eqref{e:weyl-translations}}{=} N_\epsilon(\lambda-1)\,.
    \end{align}
    Combining~\eqref{e:confining-U-in-a-ball}, \eqref{e:barV-ub-on-positive-set}, \eqref{e:Tr-bound-LT}, and~\eqref{e:Tr-weyl-relation} yields that for any~$\lambda \geq \Lambda_1$, 
    \begin{equation}
        N_\epsilon(\lambda-1) \leq C(\Lambda_1) L_{1,d}\epsilon^{-1-\frac{d}{2}}\lambda^{1+\frac{3d}{2}}\,,
    \end{equation}
    and hence
    \begin{equation}\label{e:weyl-lower-d-ub}
        N_{\epsilon}(\lambda) \leq C(\Lambda_1) L_{1,d}\epsilon^{-1-\frac{d}{2}}\paren*{1+\lambda}^{1+\frac{3d}{2}} \leq \tilde L_{1,d}(\Lambda_1) \epsilon^{-1-\frac{d}{2}}\lambda^{1+\frac{3d}{2}}\,.
    \end{equation}
    Combining~\eqref{e:weyl-higher-d-ub} and~\eqref{e:weyl-lower-d-ub}, we obtain that for any~$d \geq 1$, there exists a constant~$C(\Lambda_1, d)$ such that for any~$\lambda \geq \Lambda_1$,
    \begin{equation}
        N_\epsilon(\lambda) \leq C\epsilon^{-1-\frac{d}{2}}\lambda^{1+\frac{3d}{2}}\leq C\epsilon^{-1-\frac{d}{2}}\int_0^\infty \one_{\set*{\bar\lambda \leq \lambda}} \bar\lambda^\frac{3d}{2}    d\bar\lambda\,,
    \end{equation}
    and hence, combining this with~\eqref{e:weyl-int-ub}, exchanging the order of integration by Tonelli's theorem, and using the identity~$\abs{\set{t\in(0,1) \colon \lambda_t \geq \bar\lambda}} = \exp(-\bar\lambda T_0)$, we obtain
    \begin{equation}
        \int_{\Lambda_1}^\infty \exp(-\lambda T_0) N_\epsilon(d\lambda) \overset{\eqref{e:lambdat-def}}{\leq} C\epsilon^{-1-\frac{d}{2}}\int_0^\infty \bar\lambda^\frac{3d}{2}   \exp(-\bar\lambda T_0) d\bar\lambda = C(\Lambda_1, d, T_0)\epsilon^{-1-\frac{d}{2}}\,.
    \end{equation}
    Combining this with~\eqref{e:weighted-ef-sum-to-weyl} immediately proves~\eqref{e:S-general-bound}, which completes the proof.
\end{proof}

\section{Bias and variance bounds (Proof of Lemmas~\ref{l:bias-bound} and~\ref{l:squared-weight-bound})}\label{s:bias-variance}

In this section, we prove the bias bound, Lemma~\ref{l:bias-bound}, in Section~\ref{ss:bias} and the variance bound, Lemma~\ref{l:squared-weight-bound}, in Section~\ref{ss:variance}. The overall structure of both proofs is motivated by~\cite{AgarwalIyerEA26}, where Langevin AIS is analyzed on~$\T^d$. On~$\R^d$ the eigenfunctions are unbounded, and both proofs instead rest on Harris-type estimates in weighted norms. For a Lyapunov function~$W$, we use the spectral decomposition to bound the Langevin semigroup from~$L^1(W\pi_\epsilon)$ to the space normed by~$\norm{\cdot}_W$, defined as in~\eqref{e:V-normV-def}, treating the slowly decaying mode~$\psi_{2,\epsilon}$ separately when the temperature is low. For the bias, we take~$W=V$ and apply this bound at the first level (see~\eqref{e:first-level-mixing}). We must then estimate the~$L^1(V\pi_{\epsilon_1})$ norm of the function obtained by transporting the test function through the levels of the algorithm. To this end, we establish in Lemma~\ref{l:V-lyapunov} a Lyapunov drift condition for~$V$, uniformly in the temperature, and use it to pass from each level to the next (see~\eqref{e:r1f2-to-r2f3}). For the variance, we take the Lyapunov function~$W=\tilde r_k^{-2}$, which does not depend on~$k$ because the inverse temperatures are linearly spaced. Since~$\tilde r_k^2$ has norm one in~$L^1(\tilde r_k^{-2}\pi_{\epsilon_k})$, this gives the bound on~$P_k^T\tilde r_k^2$ in Lemma~\ref{l:Theta1k-Theta2k-bounds}, in which the error terms are multiplied by~$\tilde r_k^{-2}(x)$. Since~$\tilde r_{k-1}=\tilde r_k$, this factor evaluated at~$X_{k-1}$ cancels the squared incremental weight~$\tilde r_{k-1}^2(X_{k-1})$ of the previous level, and the second moments of~$\tilde w_k$ satisfy the closed recurrence~\eqref{e:second-recrurrence}.

From now on a temperature schedule~$(\epsilon_k)_k$ is fixed, and we abbreviate
\begin{equation}\label{e:index-abbrev}
  Z_k \defeq Z_{\epsilon_k}\,,\quad
  \pi_k \defeq \pi_{\epsilon_k}\,,\quad
  P_k^t \defeq P_{\epsilon_k, t}\,,\quad
  \psi_{i,k} \defeq \psi_{i,\epsilon_k}\,,\quad
  \lambda_{i,k} \defeq \lambda_{i,\epsilon_k}\,.
\end{equation}
We also write
\begin{equation}\label{e:rk-def}
  \tilde r_k \defeq \frac{\tilde\pi_{\epsilon_{k+1}}}{\tilde\pi_{\epsilon_k}}
  = \exp\paren*{-\paren*{\frac{1}{\epsilon_{k+1}}-\frac{1}{\epsilon_k}}U}
\end{equation}
for the incremental weight appearing in Algorithm~\ref{a:aisLMC}, so that~$\tilde w_K = \prod_{k=1}^K \tilde r_k(X_k)$.

\subsection{Proof of Lemma~\ref{l:bias-bound}}\label{ss:bias}

We first establish a Lyapunov drift condition for~$V$, which holds uniformly in the temperature.

\begin{lemma}[Lyapunov drift]\label{l:V-lyapunov}
    Make Assumption~\ref{a:U-growth-control}. Then, for all~$\nu \in (0,1/4)$, there exists a constant~$M(\nu) > 0$ such that for all~$\epsilon \in (0,1]$ and~$T \geq 0$, 
    \begin{equation}\label{e:V-lyapunov}
        P_\epsilon^T V \leq \exp\paren*{-3 \nu T} V + M(\nu)\,,
    \end{equation}
    where~$V$ is defined as in~\eqref{e:V-normV-def}.
\end{lemma}

\begin{proof}
    Direct computation shows that
    \begin{align}
        \mathcal L_\epsilon V &= 2\nu V \paren*{\paren*{2\epsilon\nu - 1} \abs{DU}^2 + \epsilon \Delta U}\\
        \label{e:LV-bound-without-epsilon}
        &\leq 2\nu V \paren*{(2\nu-1)\abs{DU}^2 + \abs{\Delta U}}\,.
    \end{align}
    By assumptions~\eqref{a:DU-blowup} and~\eqref{a:DU^2-dominates-D^2U}, there exists~$c>0$ such that 
    \begin{equation}
        \abs{\Delta U} \leq \frac{1}{8}\abs{DU}^2 \quad\text{and}\quad \abs{DU} \geq 2 \quad\text{on}\quad \set*{\abs{x} > c}\,.
    \end{equation}
    Combining this with~\eqref{e:LV-bound-without-epsilon} yields that for all~$\nu \in (0, 1/4)$, 
    \begin{equation}\label{e:Lyapunov-ub-outside}
        \mathcal L_\epsilon V \leq -3\nu V \quad\text{on}\quad \set{\abs{x} > c}\,.
    \end{equation}
    On the other hand, if we denote~$\tilde M  \defeq \norm{U}_{C^2(B(0,c))}$ then on the set~$B(0,c)$,
    \begin{align}
        \mathcal L_\epsilon V \overset{\eqref{e:LV-bound-without-epsilon}, \nu < 1/4}{\leq} 2\nu V \abs{\Delta U} &\leq 2\tilde M  \nu V \\
        &= -3\nu V + \paren*{2\tilde M+3}\nu V \\
        \label{e:Lyapunov-ub-inside}
        &\overset{\eqref{a:U-lyapunov}}{\leq} -3\nu V + \paren*{2\tilde M+3}\nu \exp\paren*{2\nu C_L\paren*{1+\tilde M^2}}\,.
    \end{align}
    Combining~\eqref{e:Lyapunov-ub-outside} and~\eqref{e:Lyapunov-ub-inside}, we obtain
    \begin{equation}
        \mathcal L_\epsilon V \leq -3\nu V + \paren*{2\tilde M+3}\nu \exp\paren*{2\nu C_L\paren*{1+\tilde M^2}} \quad\text{on}\quad \R^d\,.
    \end{equation}
    Since~$\mathcal L_\epsilon V \leq -3\nu V + b$ with~$b \defeq (2\tilde M+3)\nu\exp(2\nu C_L(1+\tilde M^2))$, Dynkin's formula together with Gr\"onwall's inequality gives~$P_\epsilon^T V \leq e^{-3\nu T}V + b/(3\nu)$, which is~\eqref{e:V-lyapunov} with 
    \begin{equation}
        M(\nu) = \paren*{\frac{2}{3}\tilde M +1}\exp\paren*{2\nu C_L\paren*{1+\tilde M^2}}\,,
    \end{equation}
    which completes the proof. 
\end{proof}

Lemma~\ref{l:bias-bound} now follows by combining the Lyapunov drift condition of Lemma~\ref{l:V-lyapunov} with the argument of~\cite[Lemma 4.1]{AgarwalIyerEA26}, which evaluates the expectation by repeated use of the Markov property.

\begin{proof}[Proof of Lemma~\ref{l:bias-bound}]
    For notational convenience, we define
    \begin{equation}\label{e:bias-def}
    \Bias_{i}(f)  \defeq 
    \abs*{\E\left[\tilde w_i f(X_i)\right]
    - \frac{Z_{i+1}}{Z_1}\ip{f,\pi_{i+1}}}\,.
    \end{equation}
    We notice
    \begin{equation}
        \Bias_{K}(f) \leq \Bias_{K}(f^+) + \Bias_{K}(f^-)    \,,
    \end{equation}
    and hence it suffices to establish~\eqref{e:unnormalized-bias-decay} for non-negative functions~$f$.  
    Thus, without loss of generality, we subsequently assume that $f \ge 0$.
    
    We observe that the mean of the estimator can be rewritten as
    \begin{align}\label{e:mean-estimator}
    \E\brak*{\tilde w_K f(X_K)} 
    &= \E\brak*{\tilde w_{K-1} (\tilde r_K f)(X_K)} \\
    &= \E\brak*{\tilde w_{K-1} P_K^T (\tilde r_K f)(X_{K-1})}\,.
    \end{align}
    On the other hand, since~$\pi_K$ is invariant for~$P_K^T$, the corresponding exact mean satisfies
    \begin{equation}\label{e:true-mean-change}
    \frac{Z_K}{Z_1} 
    \ip{P_K^T (\tilde r_K f), \pi_K}
    = \frac{Z_K}{Z_1} \ip{\tilde r_K f, \pi_K}
    = \frac{Z_{K+1}}{Z_1} \ip{f, \pi_{K+1}}\,.
    \end{equation}
    
    Combining~\eqref{e:mean-estimator}, \eqref{e:true-mean-change}, and the definition of the bias in~\eqref{e:bias-def}, we obtain
    \[
    \Bias_{K}(f) 
    = \Bias_{K-1}\!\paren*{P_K^T (\tilde r_K f)}.
    \]
    Iterating this identity down to level~$1$ yields
    \begin{equation}\label{e:biasM0-bias20}
    \Bias_{K}(f) = \Bias_{0}\paren{\tilde f_1},
    \end{equation}
    where
    \begin{equation}
    \tilde f_i \defeq \mathcal S_i \mathcal S_{i+1}\cdots \mathcal S_K f,
    \qquad
    \mathcal S_k h \defeq P_k^T (\tilde r_k h)\,.
    \end{equation}
     This implies that
    \begin{equation}\label{e:bias-reduce-to-first-level}
        \abs*{\E\left[\tilde w_K f(X_K)\right]
    - \frac{Z_{K+1}}{Z_1}\ip{f,\pi_{K+1}}} = \abs*{P_1^T \paren{\tilde r_1 \tilde f_2} (X_0) - \ip{\tilde r_1\tilde f_2, \pi_1}}\,.
    \end{equation}
    Then, using the spectral decomposition for~$P_1$ and
    \begin{equation}
        \abs*{\ip*{g, \psi_{i, 1}}_{L^2(\pi_1)} \psi_{i, 1}} = \abs*{\ip*{gV, \psi_{i, 1}V^{-1}}_{L^2(\pi_1)} \psi_{i, 1}V^{-1} V} \leq \ip*{\abs{gV}, \pi_1} \norm{\psi_{i, 1}}_V^2 V\,,
    \end{equation}
    we obtain that for any bounded measurable function~$g$ and~$T \geq 0$, 
    \begin{align}
        \abs*{P_1^T g (x) - \ip*{g, \pi_1}} &= \abs*{\sum_{i=2}^\infty \exp\paren*{-\lambda_{i, 1}T}  \ip*{g, \psi_{i, 1}}_{L^2(\pi_1)} \psi_{i, 1}} \\
        \label{e:first-level-mixing}
        &\leq V(x) \ip*{\abs*{gV}, \pi_1} \sum_{i=2}^\infty \norm{\psi_{i,1}}_V^2 \exp\paren*{-\lambda_{i, 1}T} = V(x) \ip*{\abs*{gV}, \pi_1} S_b \,.
    \end{align}
    To apply~\eqref{e:first-level-mixing} to~$g=\tilde r_1\tilde f_2$, it only remains to estimate~$\ip{ \tilde r_1 \tilde f_2 , V}_{L^2(\pi_1)}$.
    Observe
    \begin{align}
        \ip{\tilde r_1 \tilde f_2, V}_{L^2(\pi_1)} &= \frac{Z_2}{Z_1}\ip{\tilde f_2, V}_{L^2(\pi_2)} = \frac{Z_2}{Z_1}\ip{P_2^T \paren{\tilde r_2 \tilde f_3}, V}_{L^2(\pi_2)}\\
        &= \frac{Z_2}{Z_1}\ip{\tilde r_2 \tilde f_3, P_2^T V}_{L^2(\pi_2)} \\
        &\overset{\eqref{e:V-lyapunov}}{\leq} \frac{Z_2}{Z_1}\paren{\exp\paren*{-3\nu T}   \ip{\tilde r_2 \tilde f_3, V}_{L^2(\pi_2)}    + M \ip{\tilde r_2 \tilde f_3, \pi_2}}\\
        \label{e:r1f2-to-r2f3}
        &\leq \frac{Z_2}{Z_1}\exp\paren*{-3\nu T}   \ip{\tilde r_2 \tilde f_3, V}_{L^2(\pi_2)}    + M\frac{Z_{K+1}}{Z_1} \ip*{f, \pi_{K+1}}\,.
    \end{align}
    In the last inequality, we used
    \begin{equation}\label{e:reversible-bias-equiv}
        \frac{Z_2}{Z_1}\ip{\tilde r_2 \tilde f_3, \pi_2} = \frac{Z_3}{Z_1}\ip{\tilde r_3 \tilde f_4, \pi_3} = \ldots = \frac{Z_K}{Z_1}\ip*{\tilde r_K f, \pi_K} = \frac{Z_{K+1}}{Z_1}\ip*{f, \pi_{K+1}}\,.
    \end{equation}
    Arguing similarly to~\eqref{e:r1f2-to-r2f3} and using~\eqref{e:reversible-bias-equiv} yield that for any~$k=2, \ldots, K-1$, 
    \begin{equation}
        \ip{\tilde r_k \tilde f_{k+1}, V}_{L^2(\pi_k)} \leq \frac{Z_{k+1}}{Z_k}\exp\paren*{-3\nu T}   \ip{\tilde r_{k+1} \tilde f_{k+2}, V}_{L^2(\pi_{k+1})}    + M\frac{Z_{K+1}}{Z_k} \ip*{f, \pi_{K+1}}\,,
    \end{equation}
    where~$\tilde f_{K+1} \defeq f$. 
    Solving this recurrence relation, we obtain
    \begin{equation}\label{e:first-level-mean-decomposition}
        \ip{\tilde r_1 \tilde f_2, V}_{L^2(\pi_1)} \leq I_1 + I_2\,,
    \end{equation}
    where
    \begin{align}\label{e:I1-bias-def}
        I_1 &= \frac{Z_{K+1}}{Z_1}\exp\paren[\Big]{-3(K-1)\nu T} \ip*{fV, \pi_{K+1}}\,,\\
        \label{e:I2-bias-def}
        I_2 &= M\frac{Z_{K+1}}{Z_1}\paren*{\sum_{k=0}^{K-2}\exp\paren*{-3 k\nu T}} \ip*{f, \pi_{K+1}}\\
        &\leq M\frac{Z_{K+1}}{Z_1} \frac{\ip*{f, \pi_{K+1}}}{1-\exp\paren*{-3\nu T }}\,.
    \end{align}
    For any~$\nu \in (0,1/4)$, we have~$2\nu < 1/(2\bar\epsilon) < 1/(2\underline\epsilon)$ and hence
    \begin{equation}\label{e:V-integrable}
        \ip*{V, \pi_{K+1}} = \frac{1}{Z_{\underline\epsilon}} \int_{\R^d} \exp\paren*{\paren*{2\nu - \frac{1}{\underline\epsilon}} U}\,dx \leq \frac{1}{Z_{\underline\epsilon}} \int_{\R^d} \exp\paren*{- \frac{U}{2\underline\epsilon}}\,dx  = \frac{Z_{2\underline\epsilon}}{Z_{\underline\epsilon}}\,.
    \end{equation}
    Since~$2\underline\epsilon < \bar\epsilon \leq 1$, the estimate~\eqref{e:Zepsilon-estimate} may be applied both at~$\epsilon=\underline\epsilon$ and at~$\epsilon=2\underline\epsilon$, and doing so gives
    \begin{equation}
        \ip{V, \pi_{K+1}} \leq \bar C_d = C_Z^2 2^\frac{d}{2}\,,
    \end{equation}
    and this, together with the definition of~$I_1$ in~\eqref{e:I1-bias-def}, entails
    \begin{equation}\label{e:bias-I1-bound}
        I_1 \leq \bar C_d \frac{Z_{K+1}}{Z_1} \norm{f}_\infty\,.
    \end{equation}
    Moreover, the choice of~$T$ in~\eqref{e:Tlarge-bias} and the definition of~$I_2$ in~\eqref{e:I2-bias-def} ensures
    \begin{equation}\label{e:bias-I2-bound}
        I_2 \leq 2M\frac{Z_{K+1}}{Z_1}\norm{f}_\infty\,.
    \end{equation}
    Combining~\eqref{e:bias-reduce-to-first-level}, \eqref{e:first-level-mixing}, \eqref{e:first-level-mean-decomposition}, \eqref{e:bias-I1-bound}, and~\eqref{e:bias-I2-bound} concludes the proof for~\eqref{e:unnormalized-bias-decay}.
\end{proof}

\subsection{Proof of Lemma~\ref{l:squared-weight-bound}}\label{ss:variance}

In this subsection, we prove Lemma~\ref{l:squared-weight-bound}. Conditioning on the previous level, we derive a recurrence in~$k$ for the second moment of the weight~$\tilde w_k$ and then solve it. The key input is the following bound on~$P_k^T\tilde r_k^2$, whose proof we postpone to the end of this subsection.

\begin{lemma}\label{l:Theta1k-Theta2k-bounds}
    Under the same setting as in Lemma~\ref{l:squared-weight-bound}, there exist constants~$\bar C_Z$, $C_\gamma$, $C_\mathrm{geom}$ and~$C_\mathrm{BV}$, independent of~$\kappa, \underline\epsilon$ and~$\bar\epsilon$, such that for any~$k \in \set{1, \ldots, K}$,
    \begin{equation}\label{e:A_k-lb}
        \frac{Z_k}{Z_{k+2}} \leq \bar C_Z\,,
    \end{equation}
    and there exist terms~$\Theta_{1,k}$ and~$\Theta_{2,k}$ such that the bound
    \begin{equation}\label{e:PT-Theta12}
        P_k^T \tilde r_k^2 (x) \leq \frac{Z_{k+2}}{Z_k}\paren*{1 +  \tilde r_k^{-2}(x) \paren*{\Theta_{1, k}  + \Theta_{2, k}}}\,,
    \end{equation}
    holds, where
   \begin{align}
        \label{e:Theta1-ub}
        \sum_{k=2}^K \Theta_{1,k} &\leq \bar C_{\psi_2} C_\gamma \paren*{C_\mathrm{geom} + C_\mathrm{BV}}\,,\\
        \label{e:Theta2-ub}
        \Theta_{2, k} &\leq \bar C_Z S_{v, k}(T)\,.
    \end{align}
    Here, $\bar C_{\psi_2}$ and $S_{v,k}$ are defined as in~\eqref{e:second-ef-weighted-norm-unif-ub} and~\eqref{e:Svk-def}, respectively.
\end{lemma}

\begin{proof}[Proof of Lemma~\ref{l:squared-weight-bound}]
We first prove that~$K$ is well-defined.
Indeed, using~$\underline\epsilon < \bar\epsilon / 2$, $\bar\epsilon \leq 1$ and~$\nu < 1/2$, we obtain
\begin{equation}\label{e:K-well-def}
    \frac{1}{\nu} \paren*{ \frac{1}{\underline\epsilon} - \frac{1}{\bar\epsilon}} \geq  \frac{1}{\nu\bar\epsilon} \geq 2\,,
\end{equation}
which implies that~$K \in \N$ is well-defined.  
We also observe that
\begin{equation}
K \leq \frac{1}{\nu} \paren*{ \frac{1}{\underline\epsilon} - \frac{1}{\bar\epsilon}}\,,
\end{equation}
and hence, 
\begin{equation}\label{e:V-gamma-def}
    \bar \nu \defeq \frac{1}{\epsilon_{k+1}}-\frac{1}{\epsilon_k} = \frac{1}{K}\paren*{ \frac{1}{\underline\epsilon} - \frac{1}{\bar\epsilon}} \geq \nu\,.
\end{equation}
Since the inverse temperatures~$(1/\epsilon_k)_{k=1}^{K+1}$ are linearly spaced, \eqref{e:V-gamma-def} holds with any~$k\in \set{1, \ldots, K}$. Throughout this subsection we also use the additional level~$\epsilon_{K+2}$ of Lemma~\ref{l:squared-weight-bound}. By~\eqref{e:V-gamma-def},
\begin{equation}\label{e:epsK2-def}
    \frac{1}{\epsilon_{K+2}} = \frac{1}{\epsilon_{K+1}} + \bar\nu\,,
\end{equation}
so that~$Z_{\epsilon_{K+2}}$ and~$\tilde r_{K+1}$ are well defined. The level~$\epsilon_{K+2}$ plays no role in the algorithm and enters only through the identity~\eqref{e:Vinv-SD} below.

If we define 
\begin{equation}\label{e:barV-def}
    \bar V = \exp\paren*{2\bar \nu U}\,,
\end{equation}
then for any~$k \in \set{1, \ldots, K}$,
\begin{equation}\label{e:tilderk-barV}
    \tilde r_k^2 = \exp\paren*{-2\paren*{\frac{1}{\epsilon_{k+1}} - \frac{1}{\epsilon_k}}U} \overset{\eqref{e:V-gamma-def}}{=} \exp\paren*{-2\bar\nu U} = \bar V^{-1}\,.
\end{equation}
Now, we observe that for any~$k\in\set{2, \ldots, K}$, conditioning on level~$k-1$ in the expectation of~$\tilde w_k^2$ and using~\eqref{e:PT-Theta12} and~\eqref{e:tilderk-barV} yield
\begin{align}
    \E\brak*{\tilde w_k^2} &= \E\brak*{\paren*{\prod_{j=1}^{k-1} \bar V^{-1}(X_j)} P_k^T \tilde r_k^2 \paren*{X_{k-1}}}\\
    &\overset{\eqref{e:PT-Theta12}, \eqref{e:tilderk-barV}}{\leq} \frac{Z_{k+2}}{Z_k} \E\brak*{\paren*{\prod_{j=1}^{k-1} \bar V^{-1}(X_j)} \paren*{1+ \bar V (X_{k-1} )  \paren*{\Theta_{1,k} + \Theta_{2,k}} } }  \\
    \label{e:second-recrurrence}
    &\leq \frac{Z_{k+2}}{Z_k}\paren*{\E\brak*{\tilde w_{k-1}^2} + \paren*{\Theta_{1,k} + \Theta_{2, k}} \E\brak*{\tilde w_{k-2}^2}}\,.
\end{align}
To solve this recurrence, we introduce the notations 
\begin{equation}\label{e:EATheta-def}
E_k = \E\brak{\tilde{w}_k^2}\,,\quad A_k = Z_{k+2}/Z_k\,,\quad\text{and}\quad \Theta_k = \Theta_{1,k} + \Theta_{2,k}\,.
\end{equation}
For any $k \geq 2$, \eqref{e:second-recrurrence} implies that
\begin{equation}\label{e:E_k-recurrence}
    E_k \leq A_k E_{k-1} + A_k \Theta_k E_{k-2}.
\end{equation}
We define the normalized sequence 
\begin{equation}\label{e:uk-def}
u_k = E_k \bigg/ \prod_{j=1}^k A_j\,,
\end{equation}
and divide~\eqref{e:E_k-recurrence} by $\prod_{j=1}^k A_j$ to obtain
\begin{equation}
    u_k \leq u_{k-1} + \frac{\Theta_k}{A_{k-1}} u_{k-2}\,.
\end{equation}
Defining $M_k = \max_{0 \leq j \leq k} u_j$, we immediately see that
\begin{equation}
    M_k \leq \paren*{1 + \frac{\Theta_k}{A_{k-1}}} M_{k-1}\,,
\end{equation}
which implies that the sequence is bounded by the product
\begin{equation}\label{e:uK-ub}
    u_K \leq M_1 \prod_{k=2}^K \paren*{1 + \frac{\Theta_k}{A_{k-1}}} \leq M_1 \exp\paren*{\sum_{k=2}^K \frac{\Theta_k}{A_{k-1}}} \,.
\end{equation}
We note that for any~$k$, 
\begin{align}
    A_k^{-1} \overset{\eqref{e:A_k-lb}}{\leq} \bar C_Z\,,
\end{align}
and thus,
\begin{align}\label{e:sum-Thetak-Ak-ub}
    \sum_{k=2}^K \frac{\Theta_k}{A_{k-1}} &\leq \bar C_Z \paren*{\sum_{k=2}^K \Theta_{1, k} + \sum_{k=2}^K \Theta_{2, k}} \\
    &\overset{\eqref{e:Theta1-ub}, \eqref{e:Theta2-ub}}{\leq} \bar C_Z \paren*{\bar C_{\psi_2} C_\gamma \paren*{C_\mathrm{geom} + C_\mathrm{BV}} + \bar C_Z\sum_{k=2}^K S_{v,k}(T)  } \,.
\end{align}
Moreover,
\begin{equation}
    M_1 = \max\set*{u_0, u_1} = \max\set*{1, \E\brak*{\tilde w_1^2}/A_1}\,.
\end{equation}
Using~\eqref{e:first-level-mixing} for~$g=\bar V^{-1}$ and combining this with~$\bar V^{-1} V \overset{\eqref{e:V-gamma-def}}{\leq} 1$, we obtain that \begin{equation}
    \E\brak*{\tilde w_1^2} = P_1^T \tilde r_1^2 (X_0) = P_1^T \bar V^{-1} (X_0) \leq \frac{Z_3}{Z_1} + V(X_0) S_b(T)\,,
\end{equation}
and hence
\begin{equation}\label{e:M1-ub}
    M_1 \overset{\eqref{e:A_k-lb}}{\leq} 1 +  \bar C_ZV(X_0)S_b(T)\,. 
\end{equation}
Using the definition of~$u_K$ in~\eqref{e:uk-def} and the identity
\begin{equation}
    \prod_{k=1}^K A_k = \prod_{k=1}^K \frac{Z_{k+2}}{Z_k} = \frac{Z_{K+1} Z_{K+2}}{Z_1 Z_2}\,,
\end{equation}
and combining this with~\eqref{e:uK-ub}, \eqref{e:sum-Thetak-Ak-ub}, and~\eqref{e:M1-ub}, we obtain
\begin{equation}
    \E\brak{\tilde w_K^2} \leq  \exp\paren*{\bar C_\Theta \paren*{1+\sum_{k=2}^K S_{v,k}(T)}   } \frac{Z_{K+1} Z_{K+2}}{Z_1 Z_2} \paren*{1+\bar C_ZV(X_0)S_b(T)}\,,
\end{equation}
where
\begin{equation}
    \bar C_\Theta \defeq \bar C_Z\max\set*{\bar C_{\psi_2} C_\gamma \paren*{C_\mathrm{geom} + C_\mathrm{BV}}, \bar C_Z    }\,.
\end{equation}
This completes the proof for~\eqref{e:unnormalized-variance-ub}.
\end{proof}

The linear spacing of the inverse temperatures is used crucially in the bound for~$P_k^T\tilde r_k^2$ (see~\eqref{e:PT-Theta12-copied}), as in~\cite[Lemma 5.2]{AgarwalIyerEA26}. Combining that bound with Lemmas~\ref{l:lyapunov-weighted-Linftybds-and-their-sums} and~\ref{l:dEF-small} proves Lemma~\ref{l:Theta1k-Theta2k-bounds}.

\begin{proof}[Proof of Lemma~\ref{l:Theta1k-Theta2k-bounds}]
Throughout the proof, we define~$\bar \nu$ and~$\bar V$ as in~\eqref{e:V-gamma-def} and~\eqref{e:barV-def}, respectively, for the simplicity of the notation.
We first prove~\eqref{e:A_k-lb}.
Indeed, since
\begin{equation}
    K > \frac{1}{\nu} \paren*{ \frac{1}{\underline\epsilon} - \frac{1}{\bar\epsilon}} - 1 \overset{\eqref{e:K-well-def} }{\geq}  \frac{1}{2\nu} \paren*{ \frac{1}{\underline\epsilon} - \frac{1}{\bar\epsilon}}   \,,
\end{equation}
we have
\begin{equation}\label{e:gammabar-gamma-ub}
    \bar \nu \overset{\eqref{e:V-gamma-def}}{\leq} 2\nu\,,
\end{equation}
and combining this with~\eqref{e:Zepsilon-estimate} yields
\begin{align}
    \frac{Z_k}{Z_{k+2}} \overset{\eqref{e:Zepsilon-estimate}}{\leq} C_Z^2 \paren*{\frac{\epsilon_k}{\epsilon_{k+2}}}^\frac{d}{2} \overset{\eqref{e:V-gamma-def}}{=}   C_Z^2(1+ 2\bar\nu\epsilon_k)^\frac{d}{2} &\overset{\eqref{e:gammabar-gamma-ub}}{\leq} C_Z^2(1+ 4\nu)^\frac{d}{2} \\
    &\overset{\nu \leq 1/4}{\leq} 2^\frac{d}{2}C_Z^2 \defeq \bar C_Z\,.
\end{align}
It remains to prove~\eqref{e:PT-Theta12}--\eqref{e:Theta2-ub}. Using the spectral decomposition and~\eqref{e:tilderk-barV}, we obtain
\begin{multline}\label{e:Vinv-SD}
    P_k^T \tilde r_k^2 = P_k^T \bar V^{-1}= \\
    \ip{\bar V^{-1}, \pi_k} + e^{-\lambda_{2, k }T}\ip{\bar V^{-1}, \psi_{2,k}}_{L^2(\pi_k)}\psi_{2,k} + \sum_{i=3}^\infty e^{-\lambda_{i, k}T}\ip{\bar V^{-1}, \psi_{i, k}}_{L^2(\pi_k)}\psi_{i, k}\,.
\end{multline}
Combining this with the facts that
\begin{equation}
\ip*{\bar V^{-1}, \pi_k} = \ip*{\tilde r_k^2, \pi_k} = \frac{1}{Z_k}\ip*{\frac{\tilde \pi_{k+1}}{\tilde \pi_k}  , \tilde \pi_{k+1}} \overset{\tilde r_k = \tilde r_{k+1}}{=} \frac{1}{Z_k}\ip*{\frac{\tilde \pi_{k+2}}{\tilde \pi_{k+1}}  , \tilde \pi_{k+1}} = \frac{Z_{k+2}}{Z_k}\,,
\end{equation}
and similarly
\begin{equation}
    \ip*{\bar V^{-1}, \psi_{2,k}}_{L^2(\pi_k)} = \frac{Z_{k+2}}{Z_k}\ip*{\psi_{2,k}, \pi_{k+2}}\,,
\end{equation}
we obtain
\begin{equation}\label{e:PT-Theta12-copied}
    P_k^T \tilde r_k^2 (x) \leq \frac{Z_{k+2}}{Z_k}\paren*{1 +  \bar V(x)\paren*{\Theta_{1, k}  + \Theta_{2, k}}}\,,
\end{equation}
where
\begin{equation}
    \Theta_{1, k} = \abs*{\ip*{\psi_{2,k}, \pi_{k+2}}} \norm{\psi_{2,k}}_{\bar V}\,,\quad 
    \Theta_{2, k} = \frac{Z_k}{Z_{k+2}} \sum_{i=3}^\infty e^{-\lambda_{i,k}T} \norm*{\psi_{i, k}}_{\bar V}^2 \,, 
\end{equation}
and~$\quad \norm{\cdot}_{\bar V} = \norm{\cdot / \bar V}_{L^\infty(\R^d)}$.
Here, we have used~$\abs*{\ip*{\bar V^{-1}, \psi_{i,k}}_{L^2(\pi_k)}} \leq \norm{\psi_{i,k}}_{\bar V} $ to deduce~$\Theta_{2,k}$. 

We estimate~$\Theta_{1,k}$ and~$\Theta_{2,k}$, separately. Defining~$E_k = E_{\epsilon_k}$ and~$P_{E_k} = P_{E_{\epsilon_k }}$ as in Lemma~\ref{l:dEF-small}, and~$r_{k, k+2} = \pi_{k+2} / \pi_k$, we have
\begin{equation}\label{e:psi2k-pik-bound}
    \abs*{\ip{\psi_{2,k}, \pi_{k+2}}} = \abs*{\ip{\psi_{2,k}, r_{k, k+2} - 1}_{L^2(\pi_k)}} \leq I_1 + I_2\,,
\end{equation}
where
\begin{align}
    I_1 &= \abs*{\ip{(I-P_{E_k}) \psi_{2,k}, r_{k, k+2} - 1}_{L^2(\pi_k)}} \,, \\
    I_2 &= \abs*{\ip{P_{E_k} \psi_{2,k}, r_{k, k+2} - 1}_{L^2(\pi_k)}} \,.
\end{align}
We first note that
\begin{align}
    \label{e:rk-k+2-ub-const}
    r_{k, k+2} = \frac{Z_k}{Z_{k+2}} \exp\paren*{-\paren*{\frac{1}{\epsilon_{k+2}}  -\frac{1}{\epsilon_k} }U}\leq \frac{Z_k}{Z_{k+2}} \overset{\eqref{e:A_k-lb}}{\leq} \bar C_Z\,.
\end{align}
Then, combining this with the Cauchy--Schwarz inequality and~\eqref{e:eigen-residual-exp-small} at~$\gamma = \hat\gamma / 2$ yields 
\begin{equation}\label{e:I1-dEF-bound}
    I_1 \leq \norm{(I-P_{E_k}) \psi_{2,k}}_{L^2(\pi_k)}\paren*{1+ \norm{r_{k, k+2}}_{L^\infty}} \leq C(1+\bar C_Z)\exp\paren*{-\frac{\hat \gamma}{2\epsilon_k}}\,.
\end{equation}
Moreover, using~\eqref{e:PEe-decompose}, we obtain
\begin{align}
    I_2 &= \abs*{\ip{a_{1, \epsilon_k}\one_{\Omega_1} + a_{2, \epsilon_k} \one_{\Omega_2}, \pi_{k+2} - \pi_k}_{L^2(dx)}} \\
    &= \abs*{a_{1, \epsilon_k}\paren*{\pi_{k+2}(\Omega_1) - \pi_k(\Omega_1)} + a_{2, \epsilon_k}\paren*{\pi_{k+2}(\Omega_2) - \pi_k(\Omega_2)}} \\
    &= \abs*{(a_{1, \epsilon_k} - a_{2, \epsilon_k})\paren*{\pi_{k+2}(\Omega_1) - \pi_k(\Omega_1)}}\\
    \label{e:I2-CBV-bound}
    &\leq 2C_m \abs*{\pi_{k+2}(\Omega_1) - \pi_k(\Omega_1)}\,.
\end{align}
Finally, \eqref{e:V-gamma-def} implies that~$\bar V \geq V$ and hence
\begin{equation}\label{e:barV-V-norm-comparison}
\norm{\cdot}_{\bar V} \leq \norm{\cdot}_V    \,.
\end{equation}
 Thus, combining~\eqref{e:second-ef-weighted-norm-unif-ub}, \eqref{e:psi2k-pik-bound}, \eqref{e:I1-dEF-bound}, \eqref{e:I2-CBV-bound}, and~\eqref{e:barV-V-norm-comparison}, we obtain
\begin{equation}
    \Theta_{1,k} \leq \bar C_{\psi_2} C_\gamma \paren*{\exp\paren*{-\frac{\hat \gamma}{2\epsilon_k}} + \abs*{\pi_{k+2}(\Omega_1) - \pi_k(\Omega_1)}}\,,
\end{equation}
where
\begin{equation}
    C_\gamma =  \max\set*{C(1+\bar C_Z), 2C_m}\,.
\end{equation}
For the first term, we have
\begin{align}\label{e:Theta1k-first}
    \sum_{k=2}^K\exp\paren*{-\frac{\hat \gamma}{2\epsilon_k}} &\overset{\eqref{e:V-gamma-def}}{=}  \sum_{k=2}^K\exp\paren*{-\frac{\hat \gamma}{2}\paren*{\frac{1}{\epsilon_1} + \bar \nu (k-1)}} \\
    &\leq \frac{1}{1-\exp\paren*{-\hat\gamma \bar\nu / 2}} \overset{\eqref{e:V-gamma-def}}{\leq} \frac{1}{1-\exp\paren*{-\hat\gamma \nu / 2}} \defeq C_\mathrm{geom}\,,
\end{align}
and for the second term, we have
\begin{equation}\label{e:Theta1k-second}
    \sum_{k=2}^K\abs{\pi_{k+2}(\Omega_1) - \pi_k (\Omega_1)} \leq  2 \int_\kappa^1 \abs{\partial_{\epsilon'}\pi_{\epsilon'}(\Omega_1)}d\epsilon'\,.
\end{equation}
Assumptions~\ref{a:criticalpts}--\ref{a:U-growth-control} and~\cite[Lemma 8.2]{HanIyerEA26}, repeated for~$\R^d$, imply the existence of a constant~$C_\mathrm{BV}$ such that
\begin{equation}\label{e:well-mass-BV}
    2\sup_{\eta \in (0,1)}\int_\eta^1 \abs{\partial_{\epsilon'}\pi_{\epsilon'}(\Omega_1)}d\epsilon' \leq C_\mathrm{BV}\,,
\end{equation}
and therefore combining this with~\eqref{e:Theta1k-first} and~\eqref{e:Theta1k-second} deduces~\eqref{e:Theta1-ub}.
Moreover, using~\eqref{e:A_k-lb}, \eqref{e:barV-V-norm-comparison}, and the definition of~$S_{v,k}$ in~\eqref{e:Svk-def}, we obtain~\eqref{e:Theta2-ub}.
This and~\eqref{e:PT-Theta12-copied} complete the proof for the lemma.
\end{proof}

\section{Proof of Proposition~\ref{p:single-particle-proposition} and Theorem~\ref{t:main-theorem}}\label{s:proof-main}

In this section, we combine the key lemmas of Section~\ref{ss:key-lemmas} to prove Proposition~\ref{p:single-particle-proposition} and Theorem~\ref{t:main-theorem}. For the proposition, we use Lemma~\ref{l:lyapunov-weighted-Linftybds-and-their-sums} to show that the choice of~$T$ in~\eqref{e:single-particle-Tchoice} makes the sums~$S_b(T)$ and~$S_{v,k}(T)$ small, and we then apply Lemmas~\ref{l:bias-bound} and~\ref{l:squared-weight-bound}. The theorem follows from the proposition by bounding the error of the self-normalized estimator~$\mu_N$ in terms of the bias and the variance of a single run.

\begin{proof}[Proof of Proposition~\ref{p:single-particle-proposition}]
    Let~$\nu, \bar\epsilon, \underline\epsilon, \delta$ be as in the statement of the proposition and let the constants~$\bar C_\nu, \tilde C_\psi$ be as in Lemma~\ref{l:lyapunov-weighted-Linftybds-and-their-sums}, $C_B$ as in Lemma~\ref{l:bias-bound}, and~$\bar C_Z$ as in Lemma~\ref{l:squared-weight-bound}. We define
    \begin{equation}\label{e:CbCv-def}
        C_b \defeq \tilde C_{\psi} (\lambda_{2,1}, 1)\bar\epsilon^{-1-\frac{d}{2}}\exp\paren*{\frac{2\bar C_\nu \lambda_{2,1}}{\bar\epsilon}}\quad\text{and}\quad C_v \defeq \tilde C_\psi(\Lambda, 1)\,,
    \end{equation}
    and choose
    \begin{align}
        C_{T,1}(\nu) &= 2\bar C_\nu + \frac{1+\log C_v + d/2}{\Lambda}  +  1   \,,\\
        C_{T,2}(\nu, \bar\epsilon) &=  \frac{1}{\lambda_{2,1}}\max\set*{\log(3C_b C_B), \log(2\bar C_Z C_b), 1}    \,.
    \end{align}
    Suppose $T$ satisfies the bound~\eqref{e:single-particle-Tchoice}.
    Then the bound ensures that
    \begin{align}\label{e:Sb-Sv-T-large}
        T \geq \max \Bigg\{  &\frac{2\bar C_\nu}{\underline\epsilon} + \frac{1}{\Lambda}\log\paren*{\frac{C_v K}{\underline\epsilon^{1+\frac{d}{2}}}  } + 1\,,\quad \frac{1}{\lambda_{2,1}}\log\paren*{\frac{3 C_b C_B V(X_0)}{\delta}}\,,\\
        &\frac{1}{\lambda_{2,1}}\log\paren*{2\bar C_Z C_b V(X_0)}\,,  \quad \frac{\log 2}{3\nu}\Bigg\} \,. 
    \end{align}
    This bound for~$T$ guarantees that~$T$ satisfies~\eqref{e:T-large-for-weighted-sum} with~$\epsilon=\bar\epsilon$ and~$T_0=1$ so applying Lemma~\ref{l:lyapunov-weighted-Linftybds-and-their-sums} with~$\Lambda_1 = \lambda_{2, 1}, T_0=1, \epsilon = \epsilon_1 = \bar\epsilon$, and~$m=2$, we obtain that
    \begin{equation}\label{e:Sb-exp-decay}
        S_b(T) \leq C_b\exp\paren*{-\lambda_{2,1} T}\,,
    \end{equation}
    where~$S_b$ and~$C_b$ are defined in~\eqref{e:SbT-def} and~\eqref{e:CbCv-def}, respectively.
    Similarly, the bound for~$T$ in~\eqref{e:Sb-Sv-T-large} also guarantees that~$T$ satisfies~\eqref{e:T-large-for-weighted-sum} with~$\epsilon = \epsilon_k$ and~$T_0=1$, for any~$k\in \set{1,\ldots, K+1}$, so applying Lemma~\ref{l:lyapunov-weighted-Linftybds-and-their-sums} with~$\Lambda_1 = \Lambda, T_0=1, \epsilon=\epsilon_k$ and~$m=3$, we obtain that
    \begin{equation}\label{e:Svk-exp-decay}
        S_{v,k}(T) \leq C_v\epsilon_k^{-1-\frac{d}{2}}\exp\paren*{-\Lambda\paren*{T-\frac{2\bar C_\nu}{\epsilon_k}}}\,,
    \end{equation}
    where~$S_{v,k}$ and~$C_v$ are defined in~\eqref{e:Svk-def} and~\eqref{e:CbCv-def}, respectively.
    Using these decay estimates for~$S_b$ and~$S_{v,k}$, we prove~\eqref{e:single-particle-bias-bound} and~\eqref{e:single-particle-variance-bound} separately.

    \smallskip

    \noindent\textbf{Bias bound~\eqref{e:single-particle-bias-bound}.}
    We see that~$T$ in~\eqref{e:Sb-Sv-T-large} satisfies~\eqref{e:Tlarge-bias} and hence~\eqref{e:unnormalized-bias-decay} holds.
    Combining~\eqref{e:unnormalized-bias-decay} and~\eqref{e:Sb-exp-decay} and using~\eqref{e:Sb-Sv-T-large} yield that for any bounded test function~$f$,
    \begin{equation}\label{e:unnormalized-bias-after-long-time}
        \abs*{\E\brak*{\tilde w_K f(X_K)} - \frac{Z_{K+1}}{Z_1}\ip*{f, \pi_{K+1}}}\leq \frac{\delta}{3} \frac{Z_{K+1}}{Z_1}  \norm{f}_\infty \,.
    \end{equation}
    In particular, applying~\eqref{e:unnormalized-bias-after-long-time} to~$f=1$, 
    we obtain
    \begin{equation}\label{e:unnormalized-weight-after-long-time}
        \abs*{\E\brak*{\tilde w_K} - \frac{Z_{K+1}}{Z_1}}\leq \frac{\delta}{3} \frac{Z_{K+1}}{Z_1}  \,,
    \end{equation}
    and combining this with $\delta \leq 1$, we get
    \begin{equation}\label{e:weight-lb-after-long-time}
        \E\brak*{\tilde w_K} \geq \frac{2}{3}\frac{Z_{K+1}}{Z_1}\,.
    \end{equation}
    We also see from the triangle inequality that
    \begin{align}
        \abs*{\E\brak*{\tilde w_K f(X_K)} - \E\brak*{\tilde w_K}\ip*{f, \pi_{K+1}}} \leq I_1 + I_2 
    \end{align}
    where
    \begin{align}
        I_1 &= \abs*{\E\brak*{\tilde w_K f(X_K)} - \frac{Z_{K+1}}{Z_1}\ip*{f, \pi_{K+1}}} \overset{\eqref{e:unnormalized-bias-after-long-time}}{\leq} \frac{\delta}{3} \frac{Z_{K+1}}{Z_1}  \norm{f}_\infty \\
        I_2 &= \abs*{\E\brak*{\tilde w_K} - \frac{Z_{K+1}}{Z_1}}\abs*{\ip*{f, \pi_{K+1}}}\overset{\eqref{e:unnormalized-weight-after-long-time}}{\leq} \frac{\delta}{3} \frac{Z_{K+1}}{Z_1} \norm{f}_\infty \,,
    \end{align}
    and dividing this inequality with~\eqref{e:weight-lb-after-long-time} yields~\eqref{e:single-particle-bias-bound}. 

    \noindent\textbf{Variance bound~\eqref{e:single-particle-variance-bound}.}
    We first see that using~\eqref{e:Sb-Sv-T-large} and~\eqref{e:Svk-exp-decay} yields
    \begin{equation}
        S_{v,k}\leq \frac{1}{K}\,, \quad \forall k\in \set*{1, \ldots, K+1}\,,
    \end{equation}
    and using~\eqref{e:Sb-Sv-T-large} and~\eqref{e:Sb-exp-decay} yields
    \begin{equation}
        2\bar C_Z V(X_0) S_b(T) \leq 1\,.
    \end{equation}
    Combining this with~\eqref{e:unnormalized-variance-ub}, we obtain
    \begin{equation}\label{e:weight-squared-ub-after-long-time}
        \E\brak*{\tilde w_K ^2} \leq 2\exp\paren*{2\bar C_\Theta } \frac{Z_{K+1} Z_{K+2}}{Z_1Z_2}\,.
    \end{equation}
    Hence,
    \begin{align}
        \mathrm{Var}\paren*{\frac{\tilde w_K f(X_K)}{\E\brak*{\tilde w_K}}}&\leq \frac{\E\brak*{\tilde w_K^2 f(X_K)^2}}{ \E\brak*{\tilde w_K}^2} \leq \frac{\E\brak*{\tilde w_K^2}}{\E\brak*{\tilde w_K}^2}\norm{f}_\infty^2 \\
        &\overset{\eqref{e:weight-lb-after-long-time}, \eqref{e:weight-squared-ub-after-long-time}}{\leq} \frac{9}{2}\exp\paren*{2\bar C_\Theta} \frac{Z_1 Z_{K+2}}{Z_2Z_{K+1}}  \norm{f}_\infty^2 \overset{\eqref{e:A_k-lb}}{\leq} \frac{9}{2}\bar C_Z\exp\paren*{2\bar C_\Theta}\norm{f}_\infty^2\,,
    \end{align}
    which completes the proof for~\eqref{e:single-particle-variance-bound} with the choice of
    \begin{equation}
        C_w \defeq \tfrac{9}{2}\bar C_Z \exp(2\bar C_\Theta)\,.
    \end{equation}        
\end{proof}

With Proposition~\ref{p:single-particle-proposition} in hand, the proof of Theorem~\ref{t:main-theorem} is almost identical to that of~\cite[Theorem 2.5]{AgarwalIyerEA26}. The only difference is the definition of the negative fractional Sobolev norm~$H^{-s}$, now taken on~$\R^d$ rather than on~$\T^d$. We present the argument here for the reader's convenience.

\begin{proof}[Proof of Theorem~\ref{t:main-theorem}]
Define
\begin{equation}\label{e:R-def}
R \defeq \frac{1}{N} \sum_{i=1}^N \tilde w_K^i .
\end{equation}
From the definition of $w_K^i$ in~\eqref{e:muN-wKi-def}, we have
\[
\frac{1}{N} \tilde w_K^i = R\, w_K^i.
\]
Adding and subtracting the same quantity and applying the triangle inequality yields, for any $x_1,\dots,x_N \in \R^d$,
\begin{equation}\label{e:break-to-I1-I2}
\Bigl| \sum_{i=1}^N w_K^i f(x_i) - \ip{f,\pi_{K+1}} \Bigr|
\le I_1 + I_2,
\end{equation}
where
\begin{align}
I_1 &\defeq
\Bigl| \sum_{i=1}^N w_K^i f(x_i)
- \frac{1}{\E \tilde w_K^1}
  \sum_{i=1}^N R w_K^i f(x_i) \Bigr|, \\
I_2 &\defeq
\Bigl| \frac{1}{\E \tilde w_K^1}
  \sum_{i=1}^N \frac{1}{N} \tilde w_K^i f(x_i)
  - \ip{f,\pi_{K+1}} \Bigr|.
\end{align}

We observe that
\begin{align}
\label{e:I1-bound}
I_1
&= \Bigl| 1 - \frac{R}{\E \tilde w_K^1} \Bigr|
   \Bigl| \sum_{i=1}^N w_K^i f(x_i) \Bigr| \\
&\le
\Bigl| 1 - \frac{R}{\E \tilde w_K^1} \Bigr|
\norm{f}_\infty
= \Bigl| \frac{1}{N} \sum_{i=1}^N
\Bigl( \frac{\tilde w_K^i}{\E \tilde w_K^1} - 1 \Bigr) \Bigr|
\norm{f}_\infty ,
\\
\label{e:I2-bound}
I_2
&=
\Bigl|
\frac{1}{N} \sum_{i=1}^N
\Bigl(
\frac{\tilde w_K^i f(x_i)}{\E \tilde w_K^1}
- \ip{f,\pi_{K+1}}
\Bigr)
\Bigr|.
\end{align}

We recall that for any i.i.d.\ random variables $(Y_i)_{i=1}^N$ in $L^2(\P)$,
\begin{equation}\label{e:y-iid-sum-bias-variance}
\E\brak*{\Bigl( \frac{1}{N} \sum_{i=1}^N Y_i \Bigr)^2}
= \frac{1}{N} \var(Y_1)
+ \bigl( \E Y_1 \bigr)^2.
\end{equation}

Applying~\eqref{e:break-to-I1-I2} with $x_i = X_K^i$ and combining~\eqref{e:I1-bound}, \eqref{e:I2-bound}, and~\eqref{e:y-iid-sum-bias-variance}, we obtain that
    \begin{equation}
        \norm*{\sum_{i=1}^N w_K^i f(X_K^i) - \ip{f, \pi_{K+1}} }_{L^2(\P)} \leq I_3 + I_4\,,
    \end{equation}
    where
    \begin{align}
        I_3 &= \frac{1}{\sqrt{N}}\var\paren*{\frac{\tilde w_K^1 }{ \E\brak*{\tilde w_K^1} }}^\frac{1}{2} \norm{f}_\infty\,,\\
        I_4 &= \frac{1}{\sqrt{N}} \var\paren*{\frac{\tilde w_K^1 f(X_K^1)}{ \E\brak*{\tilde w_K^1} }}^\frac{1}{2} + \abs*{\frac{\E\brak*{\tilde w_K f(X_K)}}{\E\brak*{\tilde w_K}} - \ip*{f, \pi_{K+1}}}\,.
    \end{align}
    Choices of~$T$ and~$N$ in~\eqref{e:multiple-particles-Tchoice} and~\eqref{e:multiple-particles-Nchoice} imply that each of the three terms appearing in~$I_3$ and~$I_4$ is bounded above by~$\delta \norm{f}_\infty / 3$, which completes the proof for~\eqref{e:main-theorem-small-error}. To prove the closeness in Sobolev norm as in~\eqref{e:langevin-ais-convergence-sobolev}, we  first note that for any two probability measures~$\mu$ and~$\nu$ on~$\R^d$,
    \begin{equation}
        \norm{\mu-\nu}_{H^{-s}(\R^d)}^2 \defeq \int_{\R^d} \paren*{1+\abs{x}^2}^{-s}\abs*{\widehat{\mu-\nu} (x)}^2 dx \,,
    \end{equation}
    where for any tempered distribution~$\rho$ on~$\R^d$, $\hat \rho$ is defined as  
    \begin{equation}
        \widehat{\rho}(x) = \int_{\R^d} \exp\paren*{-2\pi i x \cdot y} \rho(dy)\,.
    \end{equation}
    Then,
    \begin{align}
        \E\brak*{\norm*{\mu_N - \pi_{K+1}}_{H^{-s}}^2} &= \int_{\R^d} \paren*{1+\abs{x}^2}^{-s} \E\brak*{ \abs*{\ip*{w_x, \mu_N - \pi_{K+1}}}^2 } \, dx \\
        &\overset{\eqref{e:main-theorem-small-error}}{\leq} 4\delta^2 \int_{\R^d}\paren*{1+\abs{x}^2}^{-s} \,dx = C_s^2\delta^2\,,
    \end{align}
    where~$w_x(y) \defeq \exp\paren*{-2\pi i x \cdot y}$, and~$C_s \defeq 2\paren*{\int_{\R^d} \paren*{1+\abs{x}^2}^{-s} \,dx }^{1/2} < \infty$ because~$s>d/2$. This completes the proof for~\eqref{e:langevin-ais-convergence-sobolev}.

\end{proof}

\section{The multi-well case}\label{s:multi-well}

In this section, we extend our results to potentials with any finite number of wells. We first state the assumptions and the main results, Proposition~\ref{p:single-particle-proposition-multiwell} and Theorem~\ref{t:main-theorem-multiwell}. As we explain below, the arguments of Sections~\ref{s:ef-bounds}--\ref{s:proof-main} carry over except for one step, the uniform bound~\eqref{e:psi-uniform-bound} on the low-lying eigenfunctions near the saddles of~$U$. We reduce this bound to two lemmas. Lemma~\ref{l:lyapunov-uniform-boundedness} derives it from a suitable supersolution, and Lemma~\ref{l:lyapunov-existence} constructs such a supersolution. We prove Lemma~\ref{l:lyapunov-uniform-boundedness} in Section~\ref{ss:proof-uniform-bound} and Lemma~\ref{l:lyapunov-existence} in Section~\ref{ss:proof-lyapunov-existence}. The latter relies on a perturbation of~$U$ near its saddles, Lemma~\ref{l:perturb-prop}, which we construct in Section~\ref{ss:higher-index}.

For the multi-well case we make the following assumptions, which are the natural extensions of Assumptions~\ref{a:criticalpts}--\ref{a:massRatioBound}. Assumption~\ref{a:U-growth-control} is unchanged.

\begin{assumption}\label{a:criticalpts-multiwell}
 The function~$U \in C^6(\R^d,\R)$ has a nondegenerate Hessian at all of its critical points and has exactly~$J\geq 2$ local minima, located at~$x_{\min, j}$, $j=1, \ldots, J$.
\end{assumption}

The basins of attraction~$\Omega_j$, $j=1,\ldots, J$, are defined as in~\eqref{e:basin-def}. We write
\begin{equation}\label{e:S-def}
    S \defeq \set*{x\in\R^d \st DU(x)=0 \text{ and } x \text{ is not a local minimum of } U}
\end{equation}
for the set of the remaining critical points of~$U$. These are the saddles of~$U$, of arbitrary Morse index, together with its local maxima. Note that no assumption on the indices is made. We record for later use that~$S$ is finite. Indeed, by~\eqref{a:DU-blowup} there is~$R_*>0$ with~$\abs{DU}\geq1$ outside~$B(0,R_*)$, so every critical point of~$U$ lies in the compact set~$\overline{B(0,R_*)}$, and by Assumption~\ref{a:criticalpts-multiwell} the critical points are nondegenerate, hence isolated.

\begin{assumption}\label{a:massRatioBound-multiwell}
  There exist~$\kappa \in [0,1/2)$ and a constant~$C_m \geq 1$ such that
   \begin{equation}\label{e:massRatioBound-multiwell}
     \min_{i\in \set{1,\ldots, J}}\inf_{\substack{
       \epsilon\in (\kappa, 1]\\
     }}
      \pi_{\epsilon}(\Omega_i)
	\geq \frac{1}{C_m^2}\,.
   \end{equation}
\end{assumption}

Under these assumptions we obtain the following extensions of Proposition~\ref{p:single-particle-proposition} and Theorem~\ref{t:main-theorem}.

\begin{proposition}[Bias and variance bounds for multi-well potential]\label{p:single-particle-proposition-multiwell}
Make the Assumptions~\ref{a:U-growth-control}, \ref{a:criticalpts-multiwell}, and~\ref{a:massRatioBound-multiwell}.
Then, for all~$\nu \in (0, 1/8)$ and~$\bar \epsilon \in (\kappa, 1]$, there exist constants~$C_w(\nu)
, C_{T, 1}(\nu)$, and~$C_{T,2}(\nu, \bar\epsilon)$ such that  the following holds. 
For any~$\underline\epsilon \in (\kappa, \bar\epsilon/2)$ and~$\delta\in (0,1]$, choose~$K$ and~$T$ as in~\eqref{e:single-particle-Kchoice} and~\eqref{e:single-particle-Tchoice}, respectively, and define the sequence of temperatures~$(\epsilon_k)_{k=1}^{K+1}$ such that
    \begin{equation}
        \bar\epsilon = \epsilon_1 > \epsilon_2 > \cdots > \epsilon_{K+1} = \underline\epsilon\,,
    \end{equation}
    and their inverse temperatures~$\set{1/\epsilon_k}_{k=1}^{K+1}$ are linearly spaced. Let~$\tilde w_K$ and~$X_K$ denote the output of Algorithm~\ref{a:aisLMC} run with these parameters.
    Then for any bounded test function~$f$, we have~\eqref{e:single-particle-bias-bound} and~\eqref{e:single-particle-variance-bound}.
\end{proposition}

\begin{theorem}[Rapid convergence of Langevin AIS on~$\R^d$ for multi-well potential]
    \label{t:main-theorem-multiwell}
    Make the Assumptions~\ref{a:U-growth-control}, \ref{a:criticalpts-multiwell}, and~\ref{a:massRatioBound-multiwell}. Let~$\nu \in (0,1/8)$ and~$\bar\epsilon \in (\kappa, 1]$, and let~$C_w, C_{T, 1}$ and~$C_{T,2}$ be the constants from Proposition~\ref{p:single-particle-proposition-multiwell}. 
    Then, for any~$\underline\epsilon \in (\kappa, \bar\epsilon/2)$ and~$\delta \in (0,1]$, we choose~$K, T, N$ as in~\eqref{e:multiple-particles-Kchoice}--\eqref{e:multiple-particles-Nchoice}, respectively, 
    and define the sequence of temperatures~$(\epsilon_k)_{k=1}^{K+1}$ such that
    \begin{equation}
        \bar\epsilon = \epsilon_1 > \epsilon_2 > \cdots > \epsilon_{K+1} = \underline\epsilon\,,
    \end{equation}
    and their inverse temperatures~$\set{1/\epsilon_k}_{k=1}^{K+1}$ are linearly spaced. Perform~$N$ independent runs of Algorithm~\ref{a:aisLMC} with these parameters and let~$(X_K^1, \tilde w_K^1)$, \ldots, $(X_K^N, \tilde w_K^N)$ be the resulting points and unnormalized weights. Define the empirical measure~$\mu_N$ as in~\eqref{e:muN-wKi-def}.
    Then for any bounded test function~$f$, we have~\eqref{e:main-theorem-small-error}.
    Consequently, for any~$s > d/2$, we have~\eqref{e:langevin-ais-convergence-sobolev} with the same constant~$C_s$.
\end{theorem}

Proposition~\ref{p:single-particle-proposition-multiwell} and Theorem~\ref{t:main-theorem-multiwell} follow once the key Lemmas~\ref{l:lyapunov-weighted-Linftybds-and-their-sums}--\ref{l:squared-weight-bound} are established in the multi-well setting, and for all but one step the double-well arguments of Sections~\ref{s:ef-bounds}--\ref{s:bias-variance} apply verbatim. The basic properties of the eigenfunctions, the eigenvalues, the partition function and the total variation of~$\epsilon\mapsto\pi_\epsilon(\Omega_j)$ --- that is, the analogues of Lemmas~\ref{l:dEF-small}--\ref{l:partition-function-estimates} and of~\eqref{e:well-mass-BV} --- all persist, and we record them in Lemmas~\ref{l:dEF-small-multiwell}--\ref{l:mass-BV-multiwell} for the reader's convenience. Their proofs are identical to the double-well case and are omitted. We first redefine the energy barrier~$\hat \gamma$ as
\begin{equation}\label{e:gammaHatDef-multiwell}
    \hat \gamma \defeq \min_{j\in\set{1,\ldots, J}} \Delta_j\,,\qquad \Delta_j \defeq \inf_{x\in \partial\Omega_j} \paren*{U(x)-U(x_{\min, j})}\,.
\end{equation}

\begin{lemma}\label{l:dEF-small-multiwell}
    Make the Assumptions~\ref{a:U-growth-control}, \ref{a:criticalpts-multiwell}, and~\ref{a:massRatioBound-multiwell}.    
    For any~$\gamma < \hat \gamma$, there exists~$C_\gamma$ such that for any~$\epsilon\in (\kappa, 1]$, the following holds. 
    Let
    \begin{equation}
        E_\epsilon=\mathrm{span}\set{\one_{\Omega_1},\ldots, \one_{\Omega_J} }\quad\text{and}\quad F_\epsilon=\mathrm{span}\set{1,\psi_{2,\epsilon}, \ldots, \psi_{J, \epsilon}}\,,
    \end{equation}
    and let~$P_{E_\epsilon}, P_{F_\epsilon}$ denote the orthogonal projections onto~$E_\epsilon, F_\epsilon$ in~$L^2(\pi_\epsilon)$, respectively. Then, the distance~$d(E_\epsilon, F_\epsilon)$ between~$E_\epsilon$ and~$F_\epsilon$, defined as in~\eqref{e:dEeFe-def},
    satisfies~\eqref{e:ef-almost-indicators}
    and consequently, for any~$i\in \set{1,\ldots, J}$,
        \begin{equation}\label{e:eigen-residual-exp-small-multiwell}
        \norm{(I-P_{E_\epsilon})\psi_{i,\epsilon}}_{L^2(\pi_\epsilon)} \leq C_\gamma \exp\paren*{-\frac{\gamma}{\epsilon}}\,.
    \end{equation}
    Moreover, for any~$i\in \set{1,\ldots, J}$ there exist constants~$a^{(i)}_{1,\epsilon}, \ldots,  a^{(i)}_{J,\epsilon}$ such that
    \begin{equation}\label{e:PEe-decompose-multiwell}
        P_{E_\epsilon}\psi_{i,\epsilon}
        =
        \sum_{j=1}^J a^{(i)}_{j,\epsilon}\one_{\Omega_j}
        \quad\text{and}\quad
        \abs{a^{(i)}_{j,\epsilon}}
        \leq
        C_m \quad\text{for all } j \in \set{1,\ldots,J}\,.
    \end{equation}
\end{lemma}

\begin{lemma}[{\cite[Chapter 8, Proposition 2.2]{Kolokoltsov00}}]\label{l:eigenvalue-properties-multiwell}
    Make the Assumptions~\ref{a:U-growth-control}, \ref{a:criticalpts-multiwell}, and~\ref{a:massRatioBound-multiwell}. Then there exists~$\Lambda>0$ such that, for any~$\epsilon \in (\kappa, 1]$ and~$i\geq J+1$, 
    \begin{equation}
        \lambda_{i, \epsilon} \geq \Lambda\,.
    \end{equation}
    Moreover, for any~$\gamma < \hat\gamma$, there exists~$\tilde C$ such that for any~$\epsilon \in (\kappa, 1]$ and~$i\in \set*{1, \ldots, J}$,
    \begin{equation}
        \lambda_{i,\epsilon} \leq \tilde C \exp\paren*{-\frac{\gamma}{2\epsilon}}\,.
    \end{equation}
\end{lemma}

\begin{lemma}\label{l:partition-function-estimates-multiwell}
    Make the Assumptions~\ref{a:U-growth-control}, \ref{a:criticalpts-multiwell}, and~\ref{a:massRatioBound-multiwell}. Then there exists a constant~$C_Z(d,U, C_m)$ such that, for any~$\epsilon \in (\kappa, 1]$ and any~$j\in\set{1, \ldots, J}$,
    \begin{equation}
        C_Z^{-1}\epsilon^\frac{d}{2} \leq Z_\epsilon \leq C_Z \epsilon^\frac{d}{2}\exp\paren*{-\frac{U(x_{\min, j})}{\epsilon}}\,.
    \end{equation}
\end{lemma}

Repeating the proof of~\cite[Proposition 10.3(4)]{HanIyerEA26} on~$\R^d$ gives the following.

\begin{lemma}\label{l:mass-BV-multiwell}
    Make the Assumptions~\ref{a:U-growth-control}, \ref{a:criticalpts-multiwell}, and~\ref{a:massRatioBound-multiwell}. Then, there exists a constant~$C_\mathrm{BV}$ such that
    \begin{equation}
        \max_{j\in \set{1, \ldots, J}}\sup_{\eta\in (0,1)} \int_\eta^1 \abs{\partial_{\epsilon'}\pi_{\epsilon'}(\Omega_j)}\,d\epsilon' \leq C_\mathrm{BV}\,.
    \end{equation}
\end{lemma}

With Lemmas~\ref{l:dEF-small-multiwell}--\ref{l:mass-BV-multiwell} at hand, the only step that does not carry over directly is the uniform-in-$\epsilon$ bound, at small temperatures, on the eigenfunctions~$\psi_{i, \epsilon}$ with~$i=2, \ldots, J$ associated with the low-lying eigenvalues, namely the step labeled `\textbf{The case~$i=2$, small~$\epsilon \in (\kappa, \tilde\epsilon]$}' in the proof of Lemma~\ref{l:eigenfunctions-weighted-Linfty-bdd}. The remainder of this section is devoted to recovering that bound.

To see why, note that in the multi-well case the barrier~$\hat \gamma$ becomes~\eqref{e:gammaHatDef-multiwell} and~\eqref{e:eigen-residual-exp-small-multiwell}, the analogue of~\eqref{e:eigen-residual-exp-small}, still holds, but the estimate~\eqref{e:orthogonal-estimate} may fail for a saddle~$s$ such that
\begin{equation}
  U(s) - \max_{i\in \set{1,\ldots, J}} U(x_{\min, i}) - 2 \hat \gamma  > 0\,.
\end{equation}
This cannot happen in the double-well case, since there
\begin{equation}
    U(s)-\max_{i\in\set{1,2}}U(x_{\min, i}) \overset{\eqref{e:Upositive}}{=} U(0)-U(x_{\min, 2}) \overset{\eqref{e:gammaHatDef}}{=} \hat\gamma\,,    
\end{equation}
but it can happen in the multi-well case when some saddle lies far above the smallest barrier~$\hat\gamma$.

Consequently, it suffices to prove that there exist~$\tilde \epsilon \in (0,1]$ and~$C_{\psi, 2}>0$ such that
\begin{equation}\label{e:psi-uniform-bound}
    \sup_{\epsilon \in (\kappa, \tilde\epsilon]} \max_{i\in \set{1, \ldots, J}}\norm{\psi_{i, \epsilon}}_{L^\infty(\tilde S_1)} \leq C_{\psi, 2} \,, 
\end{equation}
where~$S$ is the set of non-minimum critical points of~$U$, as in Assumption~\ref{a:criticalpts-multiwell}, and
\begin{equation}\label{e:tildeS1-def}
    \tilde S_1 \defeq \bigcup_{s\in S} B\paren*{s, C_A\sqrt{\epsilon}}\,.
\end{equation}
Indeed, given~\eqref{e:psi-uniform-bound}, the rest of the proof of Lemma~\ref{l:eigenfunctions-weighted-Linfty-bdd} goes through, and the proofs of the key Lemmas~\ref{l:lyapunov-weighted-Linftybds-and-their-sums}--\ref{l:squared-weight-bound} remain unchanged. Together, these prove Proposition~\ref{p:single-particle-proposition-multiwell} and Theorem~\ref{t:main-theorem-multiwell}.

We state two lemmas which together yield~\eqref{e:psi-uniform-bound}. The first converts a suitable supersolution into the desired bound, and the second constructs such a supersolution. To state them we introduce, for a~$C^2$ function~$G$ on~$\R^d$, the functional
\begin{equation}
    Q(G) \defeq \epsilon \Delta G - DU \cdot DG + \abs{DG}^2\,.
\end{equation}
The quantity~$Q(G)$ is normalized so that~$\mathcal L_\epsilon \exp(G/\epsilon) = \epsilon^{-1}Q(G)\exp(G/\epsilon)$ (see~\eqref{e:w-lyapunov-identity} below). We further note that~\cite[Lemma 9.8]{HanIyerEA26} extends verbatim from the compact setting~$\T^d$ to~$\R^d$, so that there exist~$C, r, \tilde\epsilon> 0$ with
\begin{equation}\label{e:ef-bound-around-minima}
    \sup_{\epsilon \in (\kappa, \tilde\epsilon]}\max_{i \in \set*{1, \ldots, J}}    \norm{\psi_{i, \epsilon}}_{L^\infty( S_1  )} \leq C\,,
\end{equation}
where
\begin{equation}\label{e:S1-def}
    S_1 \defeq \bigcup_{j \in \set{1, \ldots, J}}  B(x_{\min, j}, r)\,.
\end{equation}

\begin{lemma}\label{l:lyapunov-uniform-boundedness}
    Suppose there exist constants~$\theta > 0$, $\delta \in (0,1/2)$, $L > 3\max_{s\in S} U(s) / \delta$ and~$C_G > 0$ such that, for all sufficiently small~$\epsilon > 0$, there exist a smooth function~$G$ and a smooth open set~$S_2$, both of which may depend on~$\epsilon$, satisfying
    \begin{gather}
    \label{e:S123-inclusion}
        S_1 \cup \tilde S_1  \subset S_2 \subset S_3 \,,\\
    \label{e:QG-property}
        Q(G) \leq -\theta \epsilon \quad\text{on}\quad S_2 \setminus \bar S_1\quad\text{and}\quad G = 0\quad\text{on}\quad \partial S_2\,,\\
    \label{e:G-small}
        \max_{\tilde S_1} G \leq \paren*{\frac{1}{2}+\delta} \paren*{\max_S U - L} + C_G\epsilon\,,\\
    \label{e:G-osc-bound}
        \norm{G}_{\mathrm{osc}} \leq C_G\,.
    \end{gather}
    Here~$S_1$ and~$\tilde S_1$ are defined as in~\eqref{e:S1-def} and~\eqref{e:tildeS1-def} respectively,
    \begin{align}\label{e:S3-def}
        S_3 \defeq \set*{U < L}\,,
    \end{align}
    and~$\norm{G}_{\mathrm{osc}}$ denotes the oscillation of~$G$ over~$\bar S_2 \setminus S_1$,
    \begin{equation}\label{e:osc-def}
        \norm{G}_{\mathrm{osc}} \defeq \sup_{\bar S_2 \setminus S_1} G - \inf_{\bar S_2 \setminus S_1} G\,.
    \end{equation}
    Then, there exists~$C_{\psi, 2}$ such that for all sufficiently small~$\epsilon$,
    \begin{equation}\label{e:ef-uniform-bound-around-saddle}
        \max_{i\in \set{1, \ldots, J}}\norm{\psi_{i, \epsilon}}_{L^\infty(\tilde S_1)} \leq C_{\psi, 2} \,.
    \end{equation}
\end{lemma}

\begin{lemma}\label{l:lyapunov-existence}
    Make the Assumptions~\ref{a:U-growth-control}, \ref{a:criticalpts-multiwell}, and~\ref{a:massRatioBound-multiwell}. Then there exist constants~$\theta > 0$, $\delta \in (0,1/2)$, $L > 3\max_{s\in S} U(s) / \delta$ and~$C_G > 0$ such that, for all sufficiently small~$\epsilon > 0$, there is a function~$\hat U \in C^3(\R^d, \R)$, depending on~$\epsilon$, for which
    \begin{equation}\label{e:G-S2-def}
        G \defeq \paren*{\frac{1}{2}+\delta} \paren*{\hat U - L} \quad\text{and}\quad S_2 \defeq \set*{\hat U < L}
    \end{equation}
    satisfy~\eqref{e:S123-inclusion}, \eqref{e:QG-property}, \eqref{e:G-small}, and~\eqref{e:G-osc-bound}.
\end{lemma}

With these two lemmas in hand we can already close the argument for the
multi-well case. The remainder of this section is devoted to proving them.

\begin{proof}[Proof of Proposition~\ref{p:single-particle-proposition-multiwell}
and Theorem~\ref{t:main-theorem-multiwell}]
Lemma~\ref{l:lyapunov-existence} produces constants~$\theta>0$,
$\delta\in(0,1/2)$, $L>3\max_{s\in S}U(s)/\delta$ and~$C_G>0$ and, for all
sufficiently small~$\epsilon$, a function~$G$ and a smooth open set~$S_2$
satisfying the hypotheses~\eqref{e:S123-inclusion}, \eqref{e:QG-property},
\eqref{e:G-small} and~\eqref{e:G-osc-bound} of
Lemma~\ref{l:lyapunov-uniform-boundedness}. That lemma therefore applies and
yields~\eqref{e:ef-uniform-bound-around-saddle} for all sufficiently
small~$\epsilon$. After decreasing~$\tilde\epsilon$ if necessary, this is
precisely~\eqref{e:psi-uniform-bound}.

    As explained in the paragraph preceding~\eqref{e:psi-uniform-bound}, this
    bound is the only ingredient of the key
    Lemmas~\ref{l:lyapunov-weighted-Linftybds-and-their-sums}--\ref{l:squared-weight-bound}
    that does not already carry over to the multi-well setting. With it in hand,
    the small-temperature step for~$i=2$ in the proof of
    Lemma~\ref{l:eigenfunctions-weighted-Linfty-bdd} is complete for
    every~$i \in \set{1,\ldots,J}$, and the remaining arguments of
    Sections~\ref{s:ef-bounds}--\ref{s:bias-variance} apply verbatim, with
    Lemmas~\ref{l:dEF-small}--\ref{l:partition-function-estimates}
    and~\eqref{e:well-mass-BV} replaced by
    Lemmas~\ref{l:dEF-small-multiwell}--\ref{l:mass-BV-multiwell}. This proves
    Proposition~\ref{p:single-particle-proposition-multiwell}, and
    Theorem~\ref{t:main-theorem-multiwell} follows from it exactly as
    Theorem~\ref{t:main-theorem} follows from
    Proposition~\ref{p:single-particle-proposition}.
\end{proof}

\subsection{Proof of Lemma~\ref{l:lyapunov-uniform-boundedness}}\label{ss:proof-uniform-bound}
In this subsection, we prove Lemma~\ref{l:lyapunov-uniform-boundedness}. The idea is to represent~$\psi_{i,\epsilon}$ near the saddles by the optional stopping theorem, applied at the first time the Langevin diffusion reaches either the neighborhood~$S_1$ of the minima, where~$\psi_{i,\epsilon}$ is bounded by~\eqref{e:ef-bound-around-minima}, or the boundary of~$S_2$, where~$\psi_{i,\epsilon}$ may be large. The supersolution~$G$ controls both the exponential moment of this time and the probability of reaching~$\partial S_2$ first. These two estimates are stated in the following lemma, whose proof we give at the end of this subsection.
\begin{lemma}\label{l:tau-moment-hitting-prob}
    Let~$\theta > 0$, and suppose that for all sufficiently small~$\epsilon > 0$ there exist a smooth function~$G$ and a smooth bounded open set~$S_2$, both of which may depend on~$\epsilon$, such that~$S_1 \subset S_2$ and~\eqref{e:QG-property} holds, where~$S_1$ is defined as in~\eqref{e:S1-def}.
    Fix such an~$\epsilon$, let~$X_t = X^\epsilon_t$ be the Langevin diffusion~\eqref{e:langevin} at temperature~$\epsilon$, and write~$\P_x$ and~$\E_x$ for the probability measure and the expectation under which~$X_0 = x$ almost surely.
    Let
    \begin{align}
        \label{e:tau1-def}
        \tau_1 &\defeq \inf\set*{t\geq 0\st X_t \in \partial S_1}\,,\\
        \label{e:tau2-def}
        \tau_2 &\defeq \inf\set*{t\geq 0\st X_t \in \partial S_2}\,,\\
        \label{e:tau-def}
        \tau &\defeq \inf\set*{t\geq 0\st X_t \not\in S_2\setminus \bar S_1} = \tau_1 \wedge \tau_2\,.   
    \end{align}
    Then, for any~$x \in S_2 \setminus \bar S_1$, 
    \begin{align}
        \label{e:tau-moment}
        \E_x\brak*{\exp\paren*{\theta \tau}} &\leq \exp\paren*{\frac{1}{\epsilon}\norm{G}_\mathrm{osc}}\,,\\
        \label{e:tau-hitting-prob}
        \P_x\brak*{\tau_2 < \tau_1} &\leq \exp\paren*{\frac{1}{\epsilon}G(x)} \,.
    \end{align}    
    Here~$\norm{G}_{\mathrm{osc}}$ is the oscillation~\eqref{e:osc-def} of~$G$ over~$\bar S_2 \setminus S_1$.
\end{lemma}
\begin{proof}[Proof of Lemma~\ref{l:lyapunov-uniform-boundedness}]
    Fix~$i\in \set*{1, \ldots, J}$. Throughout the proof we write~$\lambda$ and~$\psi$ in place of~$\lambda_{i, \epsilon}$ and~$\psi_{i, \epsilon}$.
    
    Since~$\psi_{1,\epsilon} \equiv 1$, we may assume~$i \geq 2$, so that~$\lambda > 0$.

    Let~$\theta, \delta, L, C_G$ be as in the statement of the lemma. By Lemma~\ref{l:eigenvalue-properties-multiwell} applied with~$\gamma=\hat\gamma/2$,
    \begin{equation}\label{e:ev-exp-small}
        \lambda \leq \tilde C\exp\paren*{-\frac{\hat\gamma}{4\epsilon}}
    \end{equation}
    for all~$\epsilon \in (\kappa,1]$. Consequently, for all sufficiently small~$\epsilon$,
    \begin{equation}\label{e:lambda-small}
        \lambda < \frac{1}{3}\frac{\delta}{1+2\delta}\theta\,,
    \end{equation}
    and in particular~$\theta/\lambda > 1$. Fix such an~$\epsilon$, let~$G$ and~$S_2$ be as in the statement of the lemma, and define~$\tau_1, \tau_2$ and~$\tau$ as in~\eqref{e:tau1-def}, \eqref{e:tau2-def} and~\eqref{e:tau-def}. The set~$S_2 \subset S_3$ is bounded, so all the hypotheses of Lemma~\ref{l:tau-moment-hitting-prob} hold, and the estimates~\eqref{e:tau-moment}--\eqref{e:tau-hitting-prob} are available. Combining them with H\"older's inequality, applied with exponent~$\theta/\lambda > 1$, yields that for any~$x\in S_2 \setminus \bar S_1$,
    \begin{align}
        \label{e:tau-moment-2}
        \E_x\brak*{\exp\paren*{\lambda \tau}} &\leq \E_x\brak*{\exp\paren*{\theta \tau}}^\frac{\lambda}{\theta} \leq \exp\paren*{ \frac{\lambda\norm{G}_{\mathrm{osc}}}{\epsilon \theta}   }\,,\\
        \E_x\brak*{\exp\paren*{\lambda \tau} \one_{\set{\tau_2 < \tau_1}}} &\leq \E_x\brak*{\exp\paren*{\theta \tau}}^\frac{\lambda}{\theta} \P_x\brak*{\tau_2 < \tau_1}^{1-\frac{\lambda}{\theta}} \\
        \label{e:tau-hitting-prob2}
        &\leq \exp\paren*{ \frac{\lambda\norm{G}_{\mathrm{osc}}}{\epsilon \theta} + \paren*{1-\frac{\lambda}{\theta}} \frac{G(x)}{\epsilon}  }\,.
    \end{align}
    Since~$\mathcal L_\epsilon \psi = -\lambda\psi$, It\^o's formula shows that
    \begin{equation}
        M_t \defeq \exp\paren*{\lambda (t\wedge \tau)}\, \psi(X_{t\wedge \tau})\,, \qquad t \geq 0\,,
    \end{equation}
    is a local martingale under~$\P_x$. Moreover~$S_3$ is bounded, so~$\norm{\psi}_{L^\infty(S_3)}<\infty$, and therefore
    \begin{equation}
        \sup_{t\geq0}\abs{M_t} \leq \exp(\lambda \tau) \norm{\psi}_{L^\infty(S_3)} \in L^1(\P_x) 
    \end{equation}
    by~\eqref{e:tau-moment-2}. Hence~$(M_t)_{t\geq0}$ is a uniformly integrable martingale. Since~$\tau<\infty$ almost surely by~\eqref{e:tau-moment}, applying the optional stopping theorem at~$\tau$ and treating the events~$\set{\tau_1 < \tau_2}$ and~$\set{\tau_2 < \tau_1}$ separately, we obtain
    \begin{equation}\label{e:tau1-tau2-casing}
        \abs{\psi(x)} = \abs*{\E_x\brak{\exp(\lambda \tau) \psi(X_\tau)}} \leq \E_x\brak{\exp(\lambda\tau)}\max_{\partial S_1} \abs{\psi} + \E_x\brak{\exp(\lambda \tau) \one_{\set{\tau_2 < \tau_1}}} \max_{\partial S_2} \abs{\psi}\,. 
    \end{equation}    
    We next claim that there exists a constant~$C > 0$ such that, for all sufficiently small~$\epsilon > 0$, 
    \begin{equation}
        \label{e:ef-bound-around-big-ball}
        \max_{\partial S_2} \abs{\psi} \leq C\exp\paren*{\frac{L}{2\epsilon}(1+\delta)}\,.
    \end{equation}
    We defer the proof of the claim and first show how it completes the proof of the lemma. Combining it with~\eqref{e:ef-bound-around-minima}, \eqref{e:tau-moment-2}--\eqref{e:tau-hitting-prob2} and~\eqref{e:tau1-tau2-casing} yields that for any~$x \in S_2 \setminus \bar S_1$, 
    \begin{equation}
        \abs{\psi(x)} \leq C \exp\paren*{ \frac{\lambda \norm{G}_{\mathrm{osc}}}{\epsilon\theta}} + C\exp\paren*{ \frac{\lambda\norm{G}_{\mathrm{osc}}}{\epsilon \theta} + \paren*{1-\frac{\lambda}{\theta}} \frac{G(x)}{\epsilon}  + \frac{L}{2\epsilon}(1+\delta)}\,.
    \end{equation}
    For sufficiently small~$\epsilon$, $S_1$ and~$\tilde S_1$ are disjoint so that~$\tilde S_1 \subset S_2 \setminus \bar S_1$, and thus, together with~\eqref{e:G-small},
    \begin{equation}\label{e:psi-bound-I1I2}
        \norm{\psi}_{L^\infty(\tilde S_1)} \leq I_1 + I_2\,, 
    \end{equation}
    where
    \begin{align}
        I_1 &=  C \exp\paren*{ \frac{\lambda \norm{G}_{\mathrm{osc}}}{\epsilon\theta}}\,,\\ 
        I_2 &= C\exp\paren*{ \frac{\lambda\norm{G}_{\mathrm{osc}}}{\epsilon \theta} + I_3 + C_G}\,,\\
        I_3 &= \frac{1}{\epsilon}\paren*{ \paren*{1-\frac{\lambda}{\theta}} \paren*{\frac{1}{2}+\delta} \paren*{\max_S U - L}  + \frac{L}{2}(1+\delta)}\,.
    \end{align}
    Since~$\norm{G}_\mathrm{osc} \leq C_G$ by~\eqref{e:G-osc-bound}, the bound~\eqref{e:ev-exp-small} gives~$\lambda\norm{G}_\mathrm{osc}/(\epsilon\theta)\to0$ as~$\epsilon\to0$, so that for sufficiently small~$\epsilon$,
    \begin{equation}\label{e:I1-I2-bounds}
        I_1 \leq 2C \quad\text{and}\quad I_2 \leq 2C\exp\paren*{C_G + I_3}\,.
    \end{equation}
    Moreover, splitting the terms of~$I_3$ according to whether or not they contain~$L$, and using~$\delta < 1/2$ together with~$\max_S U \geq 0$, we obtain
    \begin{equation}\label{e:I3-I4-ineq}
         I_3 \leq \frac{1}{\epsilon}\paren*{\max_S U-LI_4}\,,
    \end{equation}
    where
    \begin{equation}
        I_4 = \paren*{\frac{1}{2}+\delta}\paren*{1-\frac{\lambda}{\theta}}- \frac{1}{2}\paren*{1+\delta} = \frac{1}{2}\paren*{\delta - (1+2\delta)\frac{\lambda}{\theta}}\,.
    \end{equation}
    By~\eqref{e:lambda-small}, this gives
    \begin{equation}
        I_4 > \frac{\delta}{3}\,.
    \end{equation}
    Combining this with~\eqref{e:I3-I4-ineq} and using the assumption that~$L > 3\max_S U /\delta$, we get that for all sufficiently small~$\epsilon$,
    \begin{equation}
        I_3 \leq \frac{1}{\epsilon}\paren*{\max_S U - \frac{\delta}{3}L} < 0\,.
    \end{equation}
    This, together with~\eqref{e:psi-bound-I1I2} and~\eqref{e:I1-I2-bounds}, entails
    \begin{equation}
        \norm{\psi}_{L^\infty(\tilde S_1)} \leq 2C\paren*{1 + \exp\paren*{C_G}}\,,
    \end{equation}
    and since~$i\in \set{1, \ldots, J}$ was arbitrary, this completes the proof for~\eqref{e:ef-uniform-bound-around-saddle}.

    It remains to prove the claim~\eqref{e:ef-bound-around-big-ball}. 
    We first fix~$x \in \partial S_2$. By~\eqref{e:S123-inclusion}, $\partial S_2 \subset \bar S_3$, and hence~$U(x) \leq L$. Define
    \begin{equation}
        \tilde S_3 = \set*{y \in \R^d \st \dist(y, \bar S_3) < 1}\,.    
    \end{equation}
    For any~$\epsilon \in (0,1)$ we have~$B(x,\epsilon) \subset \tilde S_3$, and~$\tilde S_3$ is an open set independent of~$\epsilon$. We may therefore apply the local maximum principle~\cite[Corollary 9.21]{GilbargTrudinger01} with
    \begin{gather}
        R=\epsilon\,,\quad B_R = B(x, R/2)\,,\quad
        \Omega_*=B(x, R)\,,\quad
        \lambda_*=\Lambda_*=\epsilon\,,\quad \gamma_*=1\,,
        \\
        n_*=p_*=2\,,\quad
        \mathcal L=\mathcal L_\epsilon\,,\quad
        \nu_*=\norm{U}_{C^2(\tilde S_3)}^2/\epsilon^2 \,,\quad
        u=\psi_{i,\epsilon}\,,\quad
        f=-\lambda_{i,\epsilon}\psi_{i,\epsilon}\,.
    \end{gather}
    Then, we get
    \begin{align}
        \abs{\psi_{i, \epsilon}(x)}^2 &\leq C\paren*{\epsilon^{-d} + \lambda_{i, \epsilon}^2 }\int_{\Omega_*}\psi_{i,\epsilon}^2 \, dx\\
        \label{e:mp-near-big-ball}
        &\leq C\max_{\Omega_*}\pi_\epsilon^{-1}\paren*{\epsilon^{-d} + \lambda_{i, \epsilon}^2}  \int_{\Omega_*} \psi_{i, \epsilon}^2 \pi_\epsilon \, dx\,.
    \end{align}
    Using~$U \leq U(x) + \norm{U}_{C^1(\tilde S_3)}\epsilon \leq L + \norm{U}_{C^1(\tilde S_3)}\epsilon$ on~$\Omega_*$, we obtain
    \begin{equation}\label{e:max-pi-inverse-near-big-ball}
        \max_{\Omega_*} \pi_\epsilon^{-1} \leq CZ_\epsilon\exp\paren*{\frac{L}{\epsilon}} \leq C\epsilon^{\frac{d}{2}}\exp\paren*{\frac{L}{\epsilon}}\,.
    \end{equation}
    Setting~$\epsilon$ sufficiently small to satisfy
    \begin{equation}
        \epsilon^\frac{d}{2}\paren*{\epsilon^{-d} + \lambda_{i, \epsilon}^2} \overset{\eqref{e:ev-exp-small}}{\leq} 2\epsilon^{-\frac{d}{2}} \leq \exp\paren*{\frac{\delta L}{\epsilon}}\,,
    \end{equation}
    and combining this with~\eqref{e:mp-near-big-ball} and~\eqref{e:max-pi-inverse-near-big-ball} yields~\eqref{e:ef-bound-around-big-ball}. This completes the proof.
\end{proof}

\begin{proof}[Proof of Lemma~\ref{l:tau-moment-hitting-prob}]
    Let~$G, \theta, S_1, S_2$ be as in the statement of the lemma, and set~$w \defeq \exp(G/\epsilon)$. Since~$Dw = \epsilon^{-1}w\,DG$ and~$\Delta w = \epsilon^{-1}w\,\Delta G + \epsilon^{-2}w\abs{DG}^2$, a direct computation gives
    \begin{equation}\label{e:w-lyapunov-identity}
        \mathcal L_\epsilon w = \frac{1}{\epsilon} Q(G)\, w\,.
    \end{equation}
    In particular, \eqref{e:QG-property} implies
    \begin{equation}\label{e:w-lyapunov}
        \paren*{\mathcal L_\epsilon + \theta} w \leq 0 \quad\text{on}\quad S_2\setminus \bar S_1\,,
    \end{equation}
    and hence~$\big(\exp(\theta (t \wedge \tau))\, w(X_{t\wedge \tau})\big)_{t \geq 0}$ is a supermartingale. Consequently, for any~$x \in S_2 \setminus \bar S_1$ and any~$t\geq0$,
    \begin{equation}
        \paren*{\min_{\bar S_2\setminus S_1} w} \E_x\brak*{\exp\paren*{\theta (t \wedge \tau)}} \leq \E_x\brak*{\exp\paren*{\theta (t \wedge \tau)}\, w\paren*{X_{t \wedge \tau}}  }\leq w(x) \leq \max_{\bar S_2\setminus S_1} w\,.
    \end{equation}
    Rearranging, letting~$t \to \infty$ and invoking monotone convergence yields~\eqref{e:tau-moment}. 

    To prove~\eqref{e:tau-hitting-prob} we compare~$w$ with the function~$h(x) \defeq \P_x\brak*{\tau_2 <\tau_1}$, which solves
    \begin{equation}\label{e:h-dirichlet}
        \mathcal L_\epsilon h = 0 \quad\text{on}\quad S_2\setminus \bar S_1\,,\quad h = 1 \quad\text{on}\quad \partial S_2 \,,\quad\text{and}\quad h = 0 \quad\text{on}\quad \partial S_1\,.
    \end{equation}
    Combining~\eqref{e:QG-property}, \eqref{e:w-lyapunov}, and the fact that~$\theta, w > 0$, we obtain
    \begin{equation}
        \mathcal L_\epsilon w \leq 0 \quad\text{on}\quad S_2\setminus \bar S_1\,,\quad w = 1 \quad\text{on}\quad \partial S_2 \,,\quad\text{and}\quad w > 0 \quad\text{on}\quad \partial S_1\,.
    \end{equation}
    Since~$\mathcal L_\epsilon$ has no zeroth-order term, the minimum principle applied to~$w-h$ on~$S_2\setminus\bar S_1$ gives
    \begin{equation}
        \min_{\bar S_2\setminus S_1} (w-h) \geq \min_{\partial S_1 \cup \partial S_2} (w-h) \geq 0\,,
    \end{equation}
    which completes the proof for~\eqref{e:tau-hitting-prob}. 
\end{proof}

\subsection{Proof of Lemma~\ref{l:lyapunov-existence}}\label{ss:proof-lyapunov-existence}
We briefly describe the idea behind the proof of Lemma~\ref{l:lyapunov-existence}. If every saddle of~$U$ happened to be a local maximum, or merely satisfied~$\Delta U < 0$, then~$G = (\tfrac12+\delta)(U-L)$ would already satisfy~\eqref{e:QG-property}, since away from the critical points the term~$-\abs{DU}^2$ dominates, while near a critical point the sign of~$\epsilon\Delta U$ decides the matter. Under Assumption~\ref{a:criticalpts-multiwell}, however, $\Delta U$ may well be positive at a saddle, since the stable Hessian eigenvalues can outweigh the unstable ones. We therefore subtract from~$U$ a perturbation~$\hat P$, supported in balls of radius~$O(\sqrt\epsilon)$ around the saddles, which damps the stable directions until the trace of the Hessian becomes negative. The perturbed potential~$\hat U = U - \hat P$ then behaves at each saddle as if that saddle were a local maximum. The perturbation has amplitude~$O(\epsilon)$, so that it does not disturb the estimates of Lemma~\ref{l:lyapunov-uniform-boundedness}, and it is built from the local stable manifold rather than from the Hessian eigencoordinates, which is what allows the gradient lower bound of Lemma~\ref{l:DU-DhatP-difference-lb-gen} to hold on a ball of~$\epsilon$-independent radius. Similar perturbations appear in~\cite{MenzSchlichting14} and~\cite{Son26}.

\begin{lemma}\label{l:perturb-prop}
    There exist constants~$C_P, r_P, \theta > 0$ such that, for any~$a>0$ and all sufficiently small~$\epsilon > 0$, there is a perturbation~$\hat P \in C^3(\R^d, \R)$ with the following properties.
    \begin{enumerate}
        \item Define
        \begin{equation}\label{e:hatU-def}
            \hat U = U - \hat P\,.
        \end{equation}
        Then, for any saddle~$s\in S$, 
        \begin{equation}\label{e:P-DP-D^2P-at-saddle}
            \hat U(s) = U(s)\,,\quad D\hat U(s) = 0 \,,\quad\text{and}\quad \Delta \hat U (s) < -4\theta \,.
        \end{equation}
        \item The perturbation is supported locally around the saddles. Specifically,
        \begin{equation}
            \supp (\hat P) \subset \bigcup_{s\in S} B\paren*{s,C_Pa\sqrt{\epsilon}}\,.
        \end{equation}
        \item Derivatives of~$\hat P$ satisfy
        \begin{equation}\label{e:D^iP-norm}
            \norm{D^i\hat P}_{L^\infty(\R^d)} \leq C_P\paren*{a\sqrt{\epsilon}}^{2-i}\,,\qquad 0\leq i\leq 3\,. 
        \end{equation}
        \item The gradients obey the uniform lower bound
        \begin{equation}\label{e:DU-DP-unif-lb}
            \abs{DU(x)}^2 -  \abs{D\hat P(x)}^2 \geq C_P^{-1}\abs{x-s}^2 \,,\quad\forall s\in S \quad\text{and}\quad \forall x \in B(s, r_P)\,.
        \end{equation}
    \end{enumerate}
\end{lemma}

We momentarily postpone the proof of Lemma~\ref{l:perturb-prop} and show that these properties of the perturbation entail Lemma~\ref{l:lyapunov-existence}.

\begin{proof}[Proof of Lemma~\ref{l:lyapunov-existence}]
    Let~$C_P$, $r_P$ and~$\theta$ be the constants of Lemma~\ref{l:perturb-prop}. This~$\theta$ is the constant asserted in Lemma~\ref{l:lyapunov-existence}. Set
    \begin{equation}\label{e:C1-def}
        C_1 \defeq \norm{U}_{C^2(\cup_{s\in S} B(s,r_P))} + C_P\,,
    \end{equation}
    and let~$\eta \in (0,C_P)$ be the constant in~\eqref{e:negative-lap-near-saddle} below, which depends only on~$\theta$ and~$C_P$. The remaining constants are chosen once and for all, in the order~$\delta$, $L$, $a$, each depending only on~$U$ and on the constants chosen before it, and afterwards~$\epsilon$ is taken sufficiently small, depending on all of them. First, we fix~$\delta \in (0,1/4)$ so small that
    \begin{equation}\label{e:delta-choice}
        \max_{s\in S} U(s) + 1 < \frac{3}{\delta}\max_{s\in S} U(s) \quad\text{and}\quad \delta C_PC_1 + \delta^2C_1^2 \leq \frac{1}{8C_P}\,.
    \end{equation}
    Next, we fix~$L > 3\max_{s\in S} U(s)/\delta$ so large that
    \begin{equation}\label{e:L-choice}
        S_1 \cup \bigcup_{s\in S} B(s,r_P) \subset \set*{U < L} = S_3\,.
    \end{equation}
    Then, with the constant~$C_L \defeq \sup_{S_3} \Delta U + dC_P$, which depends on~$L$ but not on~$a$ or~$\epsilon$, we fix~$a > 0$ so large that
    \begin{equation}\label{e:a-choice}
        C_Pa > C_A \quad\text{and}\quad \frac{\eta^2a^2}{8C_P} \geq C_L + \theta\,.
    \end{equation}
    Finally, we define~$\hat P$ and~$\hat U$ as in Lemma~\ref{l:perturb-prop} with this~$a$, and set~$G$ and~$S_2$ as in~\eqref{e:G-S2-def}.
    
    We first prove~\eqref{e:S123-inclusion}.
    Using Taylor expansion for~$U$ on the perturbation region
    \begin{equation}
        W_P \defeq \bigcup_{s\in S} B\paren*{s,C_Pa\sqrt{\epsilon}}\,,
    \end{equation}
    we obtain that for all sufficiently small~$\epsilon$,
    \begin{equation}\label{e:U-on-Wp-ub}
        U \leq \max_{s\in S} U(s) + \norm{U}_{C^2(\cup_{s\in S} B(s,1))}
        C_P^2 a^2 \epsilon \leq \max_{s\in S} U(s) + 1 \quad\text{on}\quad W_P\,.
    \end{equation}
    Moreover, using~\eqref{e:P-DP-D^2P-at-saddle}, \eqref{e:D^iP-norm}, and Taylor expansion for~$\hat U$ on~$W_P$ yields that for all sufficiently small~$\epsilon$, 
    \begin{equation}\label{e:hatU-on-Wp-ub}
        \hat U \leq \max_{s\in S} U(s) + \paren*{\norm{U}_{C^2(\cup_{s\in S} B(s,1))} + C_P    }C_P^2a^2\epsilon \leq \max_{s\in S} U(s) + 1 \quad\text{on}\quad W_P\,.
    \end{equation}
    By~\eqref{e:delta-choice} and the choice of~$L$, we have~$\max_{s\in S} U(s) + 1 < 3\max_{s\in S} U(s)/\delta < L$. Hence~\eqref{e:U-on-Wp-ub} and~\eqref{e:hatU-on-Wp-ub} imply that~$U < L$ and~$\hat U < L$ on~$W_P$, respectively, and hence,
    \begin{equation}
        S_2 = \paren*{\set*{\hat U < L} \cap W_P}  \cup \paren*{\set*{\hat U < L} \setminus W_P} \subset W_P \cup \set*{U < L} = \set*{U < L} = S_3\,.
    \end{equation}
    Furthermore, $C_Pa > C_A$ by~\eqref{e:a-choice}, so that~$W_P \supset \tilde S_1$, and~$S_1 \cap W_P = \emptyset$ once~$a\sqrt{\epsilon}$ is sufficiently small. Since also~$S_1 \subset \set{U < L}$ by~\eqref{e:L-choice}, we obtain
    \begin{equation}
        S_2 = \paren*{\set*{\hat U < L} \cap W_P}  \cup \paren*{\set*{\hat U < L} \setminus W_P} = W_P \cup \paren*{\set*{U < L} \setminus W_P} \supset S_1 \cup \tilde S_1\,,
    \end{equation}
    which completes the proof for~\eqref{e:S123-inclusion}. Moreover, $S_2$ is a smooth bounded open set. Indeed, $S_2 \subset S_3$ is bounded. Since~$\hat U < L$ on~$\bar W_P$, the boundary~$\partial S_2$ lies outside~$\bar W_P$, where~$\hat U = U$, so that near~$\partial S_2$ the set~$S_2$ coincides with~$\set{U<L}$. Finally, $DU \neq 0$ on~$\set{U = L}$, because every critical point of~$U$ lies in~$S_1 \cup \bigcup_{s\in S}B(s,r_P) \subset \set{U<L}$ by~\eqref{e:L-choice}.

    We now prove~\eqref{e:QG-property}. Clearly, the continuity of~$\hat U$ implies~$G=0$ on~$\partial S_2$, the second part of the property. To prove the first property, we directly compute~$Q(G)$ with~$G$ and~$\hat U$ defined in~\eqref{e:G-S2-def} and~\eqref{e:hatU-def} and obtain
    \begin{equation}\label{e:QG-decompose-I1I2}
        Q(G) = I_1 - I_2\,, 
    \end{equation}
    where
    \begin{align}
        I_1 &= \paren*{\frac{1}{2}+\delta} \epsilon \Delta \hat U \,,\\
        I_2 &= \frac{1}{4}\paren*{\abs{DU}^2 - \abs{D\hat P}^2} + \delta D\hat U \cdot D\hat P - \delta^2 \abs{D\hat U}^2\,.
    \end{align}
    We claim that
    for any~$s\in S$ and~$x\in B(s, r_P)$,
    \begin{equation}\label{e:I2-claim}
        I_2 \geq (8C_P)^{-1}\abs{x-s}^2 \,.
    \end{equation}
    Deferring the proof of this claim for now, we complete the proof of~\eqref{e:QG-property} as follows. The estimate~\eqref{e:D^iP-norm} implies that for all sufficiently small~$\epsilon$, $\norm{\Delta \hat U}_{C^1(B(s,r_P))} \leq C_P(a\sqrt{\epsilon})^{-1}$ and hence, combining this with~\eqref{e:P-DP-D^2P-at-saddle}, we obtain that for some small constant~$\eta \in (0,C_P)$, independent of~$a$ and~$\epsilon$, it holds that
    \begin{equation}\label{e:negative-lap-near-saddle}
        \Delta \hat U \leq -2\theta \quad\text{on}\quad \bigcup_{s\in S} B(s,\eta a\sqrt{\epsilon})\,.
    \end{equation}
    Choosing~$\epsilon$ sufficiently small such that~$C_Pa\sqrt{\epsilon} < r_P$ (and thus~$W_P \subset \cup_{s\in S} B(s, r_P)$) and combining~\eqref{e:QG-decompose-I1I2}, \eqref{e:I2-claim}, and~\eqref{e:negative-lap-near-saddle} yield
    \begin{equation}\label{e:Q-lyapunov-near-saddle}
        Q(G) \leq -\theta\epsilon \quad\text{on}\quad \bigcup_{s\in S} B(s,\eta a\sqrt{\epsilon})\,.
    \end{equation}

    Moreover, $\bigcup_{s\in S}B(s,r_P) \subset S_3$ by~\eqref{e:L-choice}, and~$\abs{\Delta\hat P} \leq dC_P$ by~\eqref{e:D^iP-norm} with~$i=2$, so that~$\Delta \hat U = \Delta U - \Delta\hat P \leq C_L$ on~$\bigcup_{s\in S}B(s,r_P)$. Thus
    \begin{equation}\label{e:Q-lyapunov-far-saddle}
        Q(G) \overset{\eqref{e:I2-claim}}{\leq} \paren*{C_L - (8C_P)^{-1}\eta^2a^2}\epsilon \quad\text{on}\quad \bigcup_{s\in S} \paren*{B(s,r_P)\setminus B(s, \eta a\sqrt{\epsilon})}\,.
    \end{equation}
    Hence, by~\eqref{e:a-choice}, we obtain
    \begin{equation}
        Q(G) \leq -\theta \epsilon \quad\text{on}\quad \bigcup_{s\in S} \paren*{B(s,r_P)\setminus B(s, \eta a\sqrt{\epsilon})}\,.
    \end{equation}
    Finally, on~$S_2 \setminus \paren*{\cup_{s}B(s,r_P) \cup \bar S_1}$, $\hat U = U$ and~$\hat P =  0$. Therefore,
    \begin{equation}
        Q(G) = \paren*{\frac{1}{2}+\delta} \epsilon \Delta U - \paren*{\frac{1}{4}-\delta^2}\abs{DU}^2\,.
    \end{equation}
    Again, $\Delta U \leq C_L$ on~$S_3$. Moreover, every critical point of~$U$ lies in the open set~$S_1 \cup \bigcup_{s\in S}B(s,r_P)$, so by compactness of~$\set{U \leq L}$ there is a constant~$c>0$, depending on~$L$ but not on~$a$ or~$\epsilon$, such that
    \begin{equation}
        \abs{DU} \geq c \quad\text{on}\quad \set{U<L} \setminus \paren*{\cup_{s}B(s,r_P) \cup S_1}\,.
    \end{equation}
    Since~$\delta < 1/4$, this yields that for all sufficiently small~$\epsilon$,
    \begin{equation}\label{e:Q-lyapunov-out-crpts}
        Q(G) \leq C_L\epsilon - \frac{3}{16}c^2 \leq -\frac{3}{32}c^2 \quad\text{on}\quad S_2 \setminus \paren*{\cup_{s}B(s,r_P) \cup \bar S_1}\,.
    \end{equation}
    Combining~\eqref{e:Q-lyapunov-near-saddle}, \eqref{e:Q-lyapunov-far-saddle}, and~\eqref{e:Q-lyapunov-out-crpts} completes the proof for the first property in~\eqref{e:QG-property}.

    We now prove the claim~\eqref{e:I2-claim}. We note that~\eqref{e:P-DP-D^2P-at-saddle} implies that for all~$s\in S$, $D\hat P(s) = 0$ and, together with~\eqref{e:D^iP-norm}, 
    \begin{equation}
        \abs{D\hat P}\leq C_P\abs{x-s}\,, \quad \abs{D\hat U} \leq C_1\abs{x-s} \quad\text{on}\quad B(s,r_P)\,,
    \end{equation}
    with~$C_1$ as in~\eqref{e:C1-def}. Combining this with~\eqref{e:DU-DP-unif-lb} and~\eqref{e:delta-choice}, we obtain on~$B(s,r_P)$
    \begin{equation}
        I_2 \geq \paren*{\frac{1}{4C_P} - \delta C_PC_1 - \delta^2C_1^2}\abs{x-s}^2 \geq \frac{1}{8C_P}\abs{x-s}^2\,,
    \end{equation}
    which is~\eqref{e:I2-claim}.

    Finally, we verify~\eqref{e:G-small} and~\eqref{e:G-osc-bound} with the~$\epsilon$-independent constant
    \begin{equation}
        C_G \defeq \paren*{\frac{1}{2}+\delta} \max\set*{\paren*{\norm{U}_{C^2(\cup_{s\in S} B(s,1))} + C_P}C_P^2a^2,\; L + 1}\,.
    \end{equation}
    Since~$\tilde S_1 \subset W_P$, the bound~\eqref{e:hatU-on-Wp-ub} gives~$\hat U \leq \max_{s\in S}U(s) + C_G\epsilon/(\frac{1}{2}+\delta)$ on~$\tilde S_1$, which is~\eqref{e:G-small}. For~\eqref{e:G-osc-bound}, note that~$\hat U \leq L$ on~$\bar S_2$, so that~$G \leq 0$ there, while~\eqref{e:D^iP-norm} with~$i=0$ gives~$\hat U \geq U - C_Pa^2\epsilon \geq -1$ on~$\R^d$ for all sufficiently small~$\epsilon$, since~$U \geq 0$. Hence~$G \geq -(\frac{1}{2}+\delta)(L + 1)$, and so~$\norm{G}_{\mathrm{osc}} \leq C_G$.
\end{proof}

\subsection{Proof of Lemma~\ref{l:perturb-prop}}\label{ss:higher-index}

In this subsection we prove Lemma~\ref{l:perturb-prop}. The construction below works at a nondegenerate critical point of any Morse index, so no restriction is placed on the indices of the points of~$S$.

We construct
\begin{equation}
    \hat P = \sum_{s\in S}\hat P^{(s)}\,,
\end{equation}
where, for each~$s \in S$, there exist constants~$C_s, r_s>0$ such that, for any~$a > 0$ and all sufficiently small~$\epsilon >0$, the perturbation~$\hat P^{(s)} \in C^3(\R^d, \R)$ satisfies
\begin{gather}
    \label{e:per-saddle-at-s-gen}
    \hat P^{(s)}(s) = 0\,,\quad D\hat P^{(s)}(s) = 0\,,\quad  \Delta U (s) - \Delta \hat P^{(s)}(s) < 0\,,\\
    \label{e:per-saddle-support-gen}
    \supp \hat P^{(s)} \subset B(s, C_sa\sqrt{\epsilon})\,,\quad   \norm{D^i  \hat P^{(s)}}_{L^\infty(\R^d)} \leq C_s(a\sqrt{\epsilon})^{2-i}\,,\quad 0\leq i\leq 3\,,\\
    \label{e:per-saddle-lb-gen}
    \abs{DU(x)}^2 - \abs{D\hat P^{(s)}(x)}^2 \geq C_s^{-1}\abs{x-s}^2\,, \quad \forall x \in B(s, r_s)\,.
\end{gather}
Then we take
\begin{equation}
 C_P = \max_{s\in S} C_s\,,\quad r_P = \min\set*{\min_{s\in S}r_s,\ \frac{1}{4}\min_{s\neq s'}\abs{s-s'}}\,, \quad \text{and}\quad \theta = \frac{1}{8}\min_{s\in S} \paren*{\Delta \hat P^{(s)}(s) - \Delta U(s)}\,,
\end{equation}
which is possible since~$S$ is finite, and we decrease~$\epsilon$ until~$C_Pa\sqrt\epsilon < r_P$. Then the supports of the~$\hat P^{(s)}$ are pairwise disjoint, and on~$B(s, r_P)$ only~$\hat P^{(s)}$ is nonzero, so that~$\hat P = \hat P^{(s)}$ there. Hence~\eqref{e:per-saddle-at-s-gen}--\eqref{e:per-saddle-lb-gen} give Lemma~\ref{l:perturb-prop}. In particular, $\Delta\hat U(s) = \Delta U(s) - \Delta\hat P^{(s)}(s) \leq -8\theta < -4\theta$.

Fix~$s\in S$. After a translation we may assume~$s=0$. Let~$k \in \set{1,\ldots,d}$ be the Morse index of~$0$, that is, the number of negative eigenvalues of the Hessian~$H(0) \defeq D^2U(0)$. It suffices to prove the following two lemmas, for then~$C_s \defeq \max\set{C_0, c_0^{-1}}$ and~$r_s \defeq r_0$ satisfy~\eqref{e:per-saddle-at-s-gen}--\eqref{e:per-saddle-lb-gen}.

\begin{lemma}\label{l:hatP0-basic-prop-gen}
    There exists a constant~$C_0>0$ such that, for any~$a > 0$ and all sufficiently small~$\epsilon > 0$, there is a perturbation~$\hat P^{(0)}\in C^3(\R^d, \R)$ satisfying
    \begin{gather}
        \label{e:hatP-derivatives-eval-at0-gen}
        \hat P^{(0)}(0) = 0\,,\quad D\hat P^{(0)}(0) = 0\,,\quad  \Delta U (0) - \Delta \hat P^{(0)}(0) < 0\,,\\
        \label{e:hatP-support-C3norm-gen}
        \begin{gathered}
        \supp \hat P^{(0)} \subset B(0, C_0a\sqrt{\epsilon})\,,\\
        \norm{D^i  \hat P^{(0)}}_{L^\infty(\R^d)} \leq C_0(a\sqrt{\epsilon})^{2-i}\,,\quad 0\leq i\leq 3\,.
        \end{gathered}
    \end{gather}
\end{lemma}

\begin{lemma}\label{l:DU-DhatP-difference-lb-gen}
    Let~$\hat P^{(0)}$ be as in Lemma~\ref{l:hatP0-basic-prop-gen}. Then there exist~$r_0 > 0$ and~$c_0>0$ such that for any~$a>0$ and all sufficiently small~$\epsilon>0$,
    \begin{equation}
        \abs{DU(x)}^2 - \abs{D\hat P^{(0)}(x)}^2 \geq c_0\abs{x}^2\,, \quad \forall x \in B(0, r_0)\,.
    \end{equation}
\end{lemma}

We label the eigenvalues of~$H(0)$ as
\begin{equation}\label{e:eigenvalues-gen}
    -\lambda_{u,k} \leq \cdots \leq -\lambda_{u,1} < 0 < \lambda_1 \leq \cdots \leq \lambda_{d-k}\,,
\end{equation}
let~$n_1(0), \ldots, n_k(0)$ be orthonormal eigenvectors for~$-\lambda_{u,1},\ldots,-\lambda_{u,k}$, and let~$t_1(0), \ldots, t_{d-k}(0)$ be orthonormal eigenvectors for~$\lambda_1,\ldots,\lambda_{d-k}$. We write~$\lambda_{u,\min} \defeq \lambda_{u,1}$ for the smallest of the~$\lambda_{u,l}$. By the spectral theorem,
\begin{equation}
  H(0) = H_u + H_s\,, \qquad
  H_u \defeq -\sum_{l=1}^{k}\lambda_{u,l}\, n_l(0)n_l(0)^T\,, \qquad
  H_s \defeq \sum_{i=1}^{d-k} \lambda_i\, t_i(0)t_i(0)^T\,,
\end{equation}
and we write
\begin{equation}
    \Pi_s \defeq \sum_{i=1}^{d-k} t_i(0)t_i(0)^T\,, \qquad \Pi_u \defeq I - \Pi_s = \sum_{l=1}^{k} n_l(0)n_l(0)^T
\end{equation}
for the orthogonal projections onto the stable and unstable eigenspaces of~$H(0)$.

\emph{Local maxima.} If~$k = d$, then~$0$ is a local maximum, $H_s = 0$, and we simply take~$\hat P^{(0)} \defeq 0$. Then~\eqref{e:hatP-support-C3norm-gen} is trivial, and~\eqref{e:hatP-derivatives-eval-at0-gen} holds because~$\Delta U(0) = \tr H(0) = -\sum_{l=1}^d \lambda_{u,l} < 0$. Moreover, a Taylor expansion of~$DU$ at~$0$ gives~$\abs{DU(x)}^2 = \abs{H(0)x}^2 + O(\abs{x}^3) \geq \lambda_{u,\min}^2\abs{x}^2 - C\abs{x}^3$, so that Lemma~\ref{l:DU-DhatP-difference-lb-gen} holds with~$c_0 = \lambda_{u,\min}^2/2$ and~$r_0$ sufficiently small. \emph{From now on we assume~$1 \leq k \leq d-1$}, so that the stable eigenspace~$\Pi_s\R^d$ has dimension~$d-k\geq 1$.

We now collect the preliminary results needed for the construction. By the stable manifold theorem (see, e.g.,~\cite{Perko96}), applied to the~$C^5$ vector field~$-DU$, the local stable manifold~$W^s(0)$ of the gradient flow at~$0$ is a~$C^5$ embedded disc of dimension~$d-k$ whose tangent space at~$0$ is the stable eigenspace of~$H(0)$, that is,
\begin{equation}\label{e:T0Ws-gen}
    T_0W^s(0) = \Pi_s\R^d\,.
\end{equation}
Since~$W^s(0)$ is a~$C^5$ submanifold of~$\R^d$ (replacing it by a smaller closed disc around~$0$ if necessary), there exists~$r_0 > 0$ such that every~$x \in B(0,r_0)$ has a unique nearest point~$\xi(x) \in W^s(0)$. Moreover, by~\cite{Foote84} and~\cite[Theorem~2]{LeobacherSteinicke21}, the nearest-point projection~$\xi$ satisfies
\begin{equation}\label{e:d-xi-regularity-gen}
  \xi \in C^4\paren*{B(0,r_0); W^s(0)}\,.
\end{equation}
Furthermore, at a point of a~$C^2$ submanifold the differential of the nearest-point projection is the orthogonal projection onto the tangent space at that point~\cite[Theorem~C]{LeobacherSteinicke21}. Since~$\xi(0) = 0$, this together with~\eqref{e:T0Ws-gen} gives
\begin{equation}\label{e:dxi-is-projection-gen}
  D\xi(0) = \Pi_s\,, \qquad\text{and hence}\qquad D(I - \xi)(0) = \Pi_u\,.
\end{equation}

Now, we are ready to construct~$\hat P^{(0)}$.
\begin{definition}
Let
\begin{equation}\label{e:kappa-choice-K-def-gen}
    \rho < \min\set*{\frac{1}{2(d-k)}\sum_{l=1}^{k}\lambda_{u,l},\ \frac{1}{2}\lambda_1}\quad\text{and}\quad K\defeq H_s -  \rho \Pi_s\,,
\end{equation}
and fix a smooth cutoff function~$\chi_0: \R \to [0, 1]$ such that
\begin{equation}\label{e:cutoff-prop-gen}
    \chi_0 = 1 \quad\text{on}\quad (-\infty, 1/2]\,,\quad \chi_0 = 0 \quad\text{on}\quad [1,\infty)\,, \quad\text{and}\quad \chi_0' \leq 0 \quad\text{on}\quad \R\,.
\end{equation}
Then, fix~$R, \tilde R > 1$ such that
\begin{equation}\label{e:R-tildeR-choice-gen}
    \log R > 2\frac{\lambda_{d-k}}{\lambda_1}\norm{\chi_0'}_{L^\infty}\quad\text{and}\quad \tilde R > \frac{2  \lambda_{d-k} \norm{\chi_0'}_{L^\infty} R}{\lambda_{u,\min}}\,,
\end{equation}
and, for any given~$a, \epsilon > 0$, define~$P \colon \R^d \times \R^d \to \R$ by
\begin{multline}\label{e:P-def-gen}
    P(y,z) = \frac{1}{2}y^T K y\, \chi_1(\zeta_1)\, \chi_2(\zeta_2) \quad\text{where}\quad \zeta_1 \defeq \frac{\abs{y}^2}{a^2\epsilon}\,,\quad \zeta_2 \defeq \frac{\abs{z}^2}{a^2\epsilon}\,,\\
    \chi_1(\zeta_1) \defeq
    \begin{cases}
        \chi_0\paren*{\frac{\log \zeta_1}{\log R}} & \zeta_1 > 0\,,\\
        1 & \zeta_1 = 0\,,
    \end{cases}
    \quad\text{and}\quad
    \chi_2(\zeta_2) \defeq \chi_0\paren*{\frac{\zeta_2}{\tilde R}}\,.
\end{multline}
Finally, define
\begin{equation}\label{e:hatP-def-gen}
    \hat P^{(0)} (x) =
    \begin{cases}
    P(\xi(x), x-\xi(x) )  & x\in B(0, r_0)\,,\\
    0 & x \not\in B(0,r_0)\,.
    \end{cases}
\end{equation}
\end{definition}

Note that~$\chi_1 = 1$ on~$[0, R^{1/2}]$, so that~$P$ is smooth, and that~$K$ is nonnegative-definite with eigenvalues~$0$ (on~$\Pi_u\R^d$) and
\begin{equation}\label{e:K-eigenvalues-gen}
    \lambda_i- \rho \overset{\eqref{e:kappa-choice-K-def-gen}}{\in} \paren*{\frac{1}{2}\lambda_1, \lambda_{d-k}}\,, \qquad 1 \leq i \leq d-k\,.
\end{equation}

We now prove Lemma~\ref{l:hatP0-basic-prop-gen}. The argument is essentially that of~\cite[Lemmas 3.5--3.6]{Son26}, and we include it for the reader's convenience.

\begin{proof}[Proof of Lemma~\ref{l:hatP0-basic-prop-gen}]
    Throughout the proof, we write~$y = \xi(x)$ and~$z=x-\xi(x)$ for simplicity of notation. Moreover, for any given~$a > 0$, we assume~$\epsilon$ is sufficiently small to satisfy
    \begin{equation}\label{e:epsilon-small-in-supp-gen}
        \paren*{R^\frac{1}{2}+\tilde R^\frac{1}{2}}a\sqrt{\epsilon} < \frac{r_0}{2}\,.
    \end{equation}

    We first prove~\eqref{e:hatP-support-C3norm-gen}. By the definition of~$P$,
    \begin{equation}
        \supp P \subset \set*{(y,z) \in \R^{2d}\st \abs{y} \leq R^\frac{1}{2}a\sqrt{\epsilon}\,,\quad \abs{z} \leq \tilde R^\frac{1}{2}a\sqrt{\epsilon}}\,,
    \end{equation}
    and hence, combining this with the triangle inequality~$\abs{x} \leq \abs{y} + \abs{z}$ and~\eqref{e:epsilon-small-in-supp-gen} yields
    \begin{equation}
        \supp \hat P^{(0)} \subset  B\paren*{0, \paren*{R^\frac{1}{2}+\tilde R^\frac{1}{2}}a\sqrt{\epsilon}}\,.
    \end{equation}
    In particular the support of~$\hat P^{(0)}$ is a compact subset of the open ball~$B(0,r_0)$, so~$\hat P^{(0)} \in C^4(\R^d, \R) \subset C^3(\R^d,\R)$ by~\eqref{e:d-xi-regularity-gen}. This proves the first part of~\eqref{e:hatP-support-C3norm-gen} with any~$C_0 \geq R^{1/2}+\tilde R^{1/2}$.
    Moreover, if we define the rescaled function
    \begin{equation}
        \tilde P (\hat y, \hat z) \defeq \frac{1}{a^2\epsilon} P \paren*{a\sqrt{\epsilon} \hat y, a\sqrt{\epsilon}\hat z} = \frac{1}{2}\hat y^T K \hat y\, \chi_1(\abs{\hat y}^2)\, \chi_2(\abs{\hat z}^2)\,, \qquad \hat y, \hat z \in \R^d\,,
    \end{equation}
    then~$\tilde P$ does not depend on~$a$ and~$\epsilon$, and it is smooth with compact support, so its~$C^3$ norm is a constant. Writing~$\hat P^{(0)}$ in terms of~$\tilde P$ via the identity
    \begin{equation}
        \hat P^{(0)} (x) = a^2\epsilon\, \tilde P\paren*{\frac{y}{a\sqrt{\epsilon}}, \frac{z}{a\sqrt{\epsilon}}}\,, \qquad x\in B(0, r_0)\,,
    \end{equation}
    differentiating both sides~$i\leq 3$ times with respect to~$x$ by the chain rule, and using that~$y$ and~$z$ have~$C^3$ norms on~$B(0, r_0)$ independent of~$a$ and~$\epsilon$ by~\eqref{e:d-xi-regularity-gen}, we see that each term carries a factor~$a^2\epsilon\,(a\sqrt\epsilon)^{-j}$ with~$1 \leq j\leq i$ (or~$j=0$ when~$i=0$). Since~$a\sqrt\epsilon \leq 1$ for small~$\epsilon$, this gives the second part of~\eqref{e:hatP-support-C3norm-gen}, after increasing~$C_0$.

    To prove~\eqref{e:hatP-derivatives-eval-at0-gen}, we first note that~$\xi(0)=0 \in W^s(0)$ and combining this with the regularity of~$\xi$ in~\eqref{e:d-xi-regularity-gen}, we obtain that for any~$x\in B(0,r_0)$,
    \begin{equation}\label{e:yz-O(x)-gen}
        \abs{y} =  \abs{\xi(x) - \xi(0)}\leq \norm{\xi}_{C^1(B(0, r_0))} \abs{x}\,, \quad \abs{z} = \abs{x-y} \leq \paren*{1+\norm{\xi}_{C^1(B(0,r_0))}}\abs{x}\,.
    \end{equation}
    Since~$\chi_1 = 1$ on~$[0,R^{1/2}]$ and~$\chi_2 = 1$ on~$[0, \tilde R/2]$, we get that for some constant~$C=C(R, \tilde R, \xi)>0$ and any~$x \in B(0, Ca\sqrt{\epsilon})$, $\chi_1(\zeta_1)=\chi_2(\zeta_2)=1$, and consequently, on~$B(0, Ca\sqrt{\epsilon})$,
    \begin{multline}
        \hat P^{(0)} (x) = \frac{1}{2}y^T K y\,,\quad D\hat P^{(0)}(x)^T = y^T K D\xi(x) \,,\quad\text{and}\quad \\
         D^2\hat P^{(0)}(x) = D\xi(x)^T K D\xi(x) + \sum_{i=1}^d \paren*{K\xi(x)}_i D^2\xi_i(x)\,.
    \end{multline}
     Using~$y(0)=\xi(0)=0$, \eqref{e:dxi-is-projection-gen}, $\Pi_s^T = \Pi_s$, and~$K\Pi_s=\Pi_sK=K$, we get
     \begin{equation}
         \hat P^{(0)}(0) = 0\,, \quad D\hat P^{(0)}(0) = 0\,,\quad\text{and}\quad D^2\hat P^{(0)}(0) = K\,.
     \end{equation}
     This implies
     \begin{align}
         \Delta U(0) - \Delta \hat P^{(0)}(0) = \tr\paren*{H(0) - D^2\hat P^{(0)}(0)} &\overset{\eqref{e:kappa-choice-K-def-gen}}{=} \tr\paren*{H_u + \rho \Pi_s} \\
         &= -\sum_{l=1}^{k}\lambda_{u,l} + \rho (d-k) \overset{\eqref{e:kappa-choice-K-def-gen}}{<} -\frac{1}{2}\sum_{l=1}^{k}\lambda_{u,l}\,.
     \end{align}
     This completes the proof of~\eqref{e:hatP-derivatives-eval-at0-gen}.
\end{proof}

The proof of Lemma~\ref{l:DU-DhatP-difference-lb-gen} follows the strategy of~\cite[Lemma 3.7]{Son26}. Up to a negligible error, $y=\xi(x)$ and~$z=x-\xi(x)$ are the projections of~$x$ onto the stable and unstable eigenspaces of~$H(0)$, and~$\rho$ in~\eqref{e:kappa-choice-K-def-gen} is chosen so that the difference between~$\abs{DU}^2$ and~$\abs{D\hat P^{(0)}}^2$ survives in every stable and unstable direction. Unlike in~\cite{Son26}, however, the cutoff in~$y$ has to be taken on a logarithmic scale, as~$\chi_1$ in~\eqref{e:P-def-gen}. This keeps~$\zeta_1\abs{\chi_1'(\zeta_1)} \leq \norm{\chi_0'}_{L^\infty}/\log R$ small, so that cutting off~$P$ in the stable directions does not destroy the lower bound there (see~\eqref{e:stable-eigen-perturb-gen} below). We therefore give the proof in full.

\begin{proof}[Proof of Lemma~\ref{l:DU-DhatP-difference-lb-gen}]
    Throughout the proof, we write~$y=\xi(x)$ and~$z=x-\xi(x)$ as in the proof of Lemma~\ref{l:hatP0-basic-prop-gen}. We also use the~$O$ notation to denote a quantity bounded by a constant multiple of another, uniformly in~$a$ and~$\epsilon$. That is, $v = O(w)$ means~$\abs{v} \leq C \abs{w}$ for some constant~$C$ independent of~$a$ and~$\epsilon$.

    We begin by applying a Taylor expansion to~$DU$ at~$x=0$ to obtain
    \begin{equation}\label{e:DU-decompose-stable-unstable-gen}
        DU(x) = H(0)x + O(\abs{x}^2) = H_s x + H_u x + O(\abs{x}^2) \quad\text{on}\quad B(0, r_0)\,.
    \end{equation}
    Similarly, differentiating~$\hat P^{(0)}$ on~$B(0, r_0)$,
    we obtain
    \begin{equation}\label{e:DhatP-decompose-gen}
        D\hat P^{(0)}(x) =  (Dy(x))^T M_y y  + (Dz(x))^T M_z z  \quad\text{on}\quad B(0, r_0)\,,
    \end{equation}
    where
    \begin{align}
        M_y &= K \chi_1(\zeta_1)\chi_2(\zeta_2) + \frac{y^T K y}{a^2\epsilon} \chi_1'(\zeta_1)\chi_2 (\zeta_2) I_d\,,\\
        M_z &= \frac{y^T K y}{a^2\epsilon} \chi_1(\zeta_1)\chi_2'(\zeta_2) I_d\,,
    \end{align}
    and~$\zeta_1, \zeta_2$ are defined as in~\eqref{e:P-def-gen}.
    To evaluate~$D\hat P^{(0)}$ in~\eqref{e:DhatP-decompose-gen}, we use the Taylor expansions of~$Dy$ and~$Dz$ at~$x=0$,
    \begin{align}
        \label{e:Dy-Ps-gen}
        Dy(x) = D\xi(x) &\overset{\eqref{e:dxi-is-projection-gen}}{=} \Pi_s + O(\abs{x})\quad\text{on}\quad B(0, r_0)\,,\\
        \label{e:Dz-Pu-gen}
        Dz(x) = I_d-D\xi(x) &\overset{\eqref{e:dxi-is-projection-gen}}{=} \Pi_u + O(\abs{x})\quad\text{on}\quad B(0, r_0)\,.
    \end{align}
    We also note that, by~\eqref{e:K-eigenvalues-gen}, $y^TKy \leq \lambda_{d-k}\abs{y}^2 = \lambda_{d-k}a^2\epsilon\,\zeta_1$, and that~$\chi_1(\zeta_1) = 0$ for~$\zeta_1 \geq R$. Hence
    \begin{align}
        \label{e:My-compute-gen}
        \abs{M_y} &\leq \abs{K} + \lambda_{d-k}\, \zeta_1\abs{\chi_1'(\zeta_1)} \overset{\eqref{e:P-def-gen}}{\leq} \abs{K} +  \frac{\lambda_{d-k}}{\log R}\norm{\chi_0'}_{L^\infty} = O(1)\,,\\
        \label{e:Mz-compute-gen}
        \abs{M_z} &\leq \lambda_{d-k}\, \zeta_1\chi_1(\zeta_1)\norm{\chi_2'}_{L^\infty} \leq \lambda_{d-k} \norm{\chi_0'}_{L^\infty}\frac{R}{\tilde R} = O(1)\,.
    \end{align}

    Combining~\eqref{e:DhatP-decompose-gen}, \eqref{e:Dy-Ps-gen}, \eqref{e:Dz-Pu-gen}, \eqref{e:My-compute-gen} and~\eqref{e:Mz-compute-gen}, and using the Taylor expansions of~$y$ and~$z$ at~$x=0$,
    \begin{align}
        y(x) = Dy(0)x + O(\abs{x}^2) \overset{\eqref{e:dxi-is-projection-gen}}{=} \Pi_s x + O(\abs{x}^2) \quad\text{on}\quad B(0, r_0)\,,\\
        z(x) = Dz(0)x + O(\abs{x}^2) \overset{\eqref{e:dxi-is-projection-gen}}{=} \Pi_u x + O(\abs{x}^2) \quad\text{on}\quad B(0, r_0)\,,
    \end{align}
    together with~$\Pi_sM_y \Pi_s= M_y\Pi_s$ (since~$K$ commutes with~$\Pi_s$) and~$\Pi_u M_z \Pi_u= M_z \Pi_u$ (since~$M_z$ is a multiple of the identity),
    we obtain that on~$B(0, r_0)$,
    \begin{equation}\label{e:DhatP-eigen-decompose-gen}
        D\hat P^{(0)}(x) =  \Pi_s M_y y + \Pi_u M_z z + O(\abs{x}^2) = M_y \Pi_s x + M_z \Pi_u x +O(\abs{x}^2)\,.
    \end{equation}
    In particular, using the definition of~$K$ in~\eqref{e:kappa-choice-K-def-gen} yields
    \begin{equation}\label{e:DhatP-eigen-decompose2-gen}
        M_y \Pi_s x + M_z \Pi_u x = \sum_{i=1}^{d-k} \tilde \lambda_i \ip{x, t_i(0)} t_i(0) - \tilde \lambda_u \sum_{l=1}^{k}\ip{x, n_l(0)}n_l(0)\,,
    \end{equation}
    where
    \begin{align}
        \tilde \lambda_i &= \paren*{\lambda_i - \rho}\chi_1(\zeta_1)\chi_2(\zeta_2) + \frac{y^T K y}{a^2\epsilon} \chi_1'(\zeta_1)\chi_2(\zeta_2)\,, \qquad 1 \leq i \leq d-k\,,\\
        \tilde \lambda_u &= -\frac{y^T K y}{a^2\epsilon} \chi_1(\zeta_1) \chi_2'(\zeta_2)\,.
    \end{align}
    Note that a single scalar~$\tilde\lambda_u$ serves all~$k$ unstable directions, because~$M_z$ is a multiple of the identity.
    We note that~$\chi_1'\leq 0$ and~$K$ is nonnegative-definite, so that
    \begin{equation}
        \tilde \lambda_i \leq \lambda_i - \rho\,.
    \end{equation}
    We also see that
    \begin{equation}
        \frac{y^T K y}{a^2\epsilon}\abs{\chi_1'(\zeta_1)} \chi_2(\zeta_2) \leq \lambda_{d-k}\, \zeta_1\abs{\chi_1'(\zeta_1)} \leq \lambda_{d-k} \frac{\norm{\chi_0'}_{L^\infty}}{\log R} \overset{\eqref{e:R-tildeR-choice-gen}}{\leq} \frac{\lambda_1}{2}\,,
    \end{equation}
    and hence
    \begin{equation}
        \tilde \lambda_i \geq - \frac{\lambda_1}{2}\,.
    \end{equation}
    This implies that
    \begin{equation}\label{e:stable-eigen-perturb-gen}
        \abs{\tilde \lambda_i} \leq \max\set*{\frac{\lambda_1}{2}, \lambda_i - \rho}\overset{\eqref{e:kappa-choice-K-def-gen}}{=}\lambda_i - \rho\,.
    \end{equation}
    Similarly, since~$\chi_2' \leq 0$,
    \begin{equation}\label{e:unstable-eigen-perturb-gen}
        0\leq \tilde\lambda_u \leq \lambda_{d-k}\, \zeta_1\chi_1(\zeta_1) \norm{\chi_2'}_{L^\infty} \leq \lambda_{d-k} \norm{\chi_0'}_{L^\infty} \frac{R}{\tilde R} \overset{\eqref{e:R-tildeR-choice-gen}}{\leq} \frac{\lambda_{u,\min}}{2} \,.
    \end{equation}

    Finally, by~\eqref{e:DU-decompose-stable-unstable-gen},
    \begin{equation}
        H_sx + H_ux = \sum_{i=1}^{d-k} \lambda_i \ip{x, t_i(0)} t_i(0) - \sum_{l=1}^{k}\lambda_{u,l}\ip{x, n_l(0)}n_l(0)\,.
    \end{equation}
    Using this together with~\eqref{e:DU-decompose-stable-unstable-gen}, \eqref{e:DhatP-eigen-decompose-gen}, \eqref{e:DhatP-eigen-decompose2-gen}, \eqref{e:stable-eigen-perturb-gen}, and~\eqref{e:unstable-eigen-perturb-gen}, and the orthonormality of~$t_1(0),\ldots,t_{d-k}(0),n_1(0),\ldots,n_k(0)$, and decreasing~$r_0$ if necessary, we obtain that on~$B(0, r_0)$,
    \begin{align}
        \abs{DU(x)}^2 - \abs{D\hat P^{(0)}(x)}^2 &\geq \sum_{i=1}^{d-k} \paren*{\lambda_i^2 - \tilde \lambda_i^2} \abs{\ip{x, t_i(0)}}^2\\
        &\qquad\qquad + \sum_{l=1}^{k}\paren*{\lambda_{u,l}^2 - \tilde\lambda_u^2} \abs{\ip{x, n_l(0)}}^2 + O(\abs{x}^3)\\
        &\geq c(\rho, \lambda_1, \lambda_{u,\min}) \abs{x}^2 + O(\abs{x}^3)\\
        &\geq \frac{1}{2}c(\rho, \lambda_1, \lambda_{u,\min}) \abs{x}^2\,,
    \end{align}
    where
    \begin{equation}
        c(\rho, \lambda_1, \lambda_{u,\min}) \defeq \min\set*{\rho\lambda_1, \frac{3}{4}\lambda_{u,\min}^2 }\,.
    \end{equation}
    Here we used~$\lambda_i^2 - \tilde\lambda_i^2 \geq \lambda_i^2 - (\lambda_i - \rho)^2 = \rho(2\lambda_i - \rho) \geq \rho\lambda_1$ and~$\lambda_{u,l}^2 - \tilde\lambda_u^2 \geq \lambda_{u,l}^2 - \frac{1}{4}\lambda_{u,\min}^2 \geq \frac{3}{4}\lambda_{u,\min}^2$. This proves the lemma with~$c_0 = \frac{1}{2}c(\rho,\lambda_1,\lambda_{u,\min})$.
\end{proof}

\appendix
\section{Deferred proofs}\label{s:appendix}

In this appendix we prove Lemmas~\ref{l:dEF-small} and~\ref{l:partition-function-estimates}.

\begin{proof}[Proof of Lemma~\ref{l:partition-function-estimates}]
By the Laplace method (see, e.g.,~\cite[Proposition B2]{Kolokoltsov00}), for each~$i\in \set{1,2}$,
    \begin{equation}
        \lim_{\epsilon \to 0} (2\pi\epsilon)^{-\frac{d}{2}} \exp\paren*{\frac{U(x_{\min, i})}{\epsilon} } \sqrt{\det D^2 U(x_{\min, i})}    \int_{\Omega_i}\exp\paren*{-\frac{U}{\epsilon}}\, dx = 1\,.
    \end{equation}
    In particular, there exists~$\epsilon_0\in(0,1]$ such that the expression inside the limit lies between~$1/2$ and~$3/2$ for all~$\epsilon\in(0,\epsilon_0]$. For~$\epsilon\in[\epsilon_0,1]$ the same expression is a continuous and positive function of~$\epsilon$, and hence it is bounded above and below by positive constants there as well. Consequently, there exist constants~$C_{Z, 1}(d,U)$ and~$C_{Z, 2}(d,U)$ such that, for all~$\epsilon \in (\kappa,1]$,
    \begin{align}
        Z_\epsilon &=    \paren*{1 + \frac{1-\pi_\epsilon(\Omega_2)}{\pi_\epsilon(\Omega_2)}} \int_{\Omega_2} \exp\paren*{-\frac{U}{\epsilon}}\,dx  \\
        &\overset{\eqref{e:massRatioBound}}{\leq} C_{Z,2}(1+C_m^2)\epsilon^\frac{d}{2}\exp\paren*{-\frac{U(x_{\min, 2})}{\epsilon}}\,,\\
        Z_\epsilon &\geq \int_{\Omega_1}\exp\paren*{-\frac{U}{\epsilon}} \, dx \overset{U(x_{\min, 1})= 0}{\geq} C_{Z,1}\epsilon^\frac{d}{2}\,,
    \end{align}
    which completes the proof for~\eqref{e:Zepsilon-estimate}.
\end{proof}

\begin{proof}[Proof of Lemma~\ref{l:dEF-small}]
    The estimate~\eqref{e:ef-almost-indicators} is immediate from~\cite[Chapter 8, Proposition 2.2]{Kolokoltsov00}, which holds on~$\R^d$ under Assumption~\ref{a:U-growth-control} (see the discussion on p.~247(B) of that reference). The bound~\eqref{e:eigen-residual-exp-small} then follows from~\eqref{e:ef-almost-indicators} together with the fact that~$\psi_{2,\epsilon}$ is normalized in~$L^2(\pi_\epsilon)$.
    We also see that since~$P_{E_\epsilon}$ is the projection onto~$E_\epsilon$,
    \begin{equation}
        P_{E_\epsilon} \psi_{2,\epsilon} =  a_{1,\epsilon}\one_{\Omega_1} + a_{2,\epsilon}\one_{\Omega_2}\,,
    \end{equation}
    where
    \begin{equation}
        a_{i, \epsilon} = \frac{1}{\pi_\epsilon(\Omega_i)}\int_{\Omega_i} \psi_{2, \epsilon}\,  \pi_\epsilon \,dx \,, \qquad i\in\set{1,2}\,.
    \end{equation}
    Using the Cauchy--Schwarz inequality, we obtain
    \begin{equation}
        \abs{a_{i, \epsilon}} \leq \pi_\epsilon\paren*{\Omega_i}^{-\frac{1}{2}}\overset{\eqref{e:massRatioBound}}{\leq} C_m\,,
    \end{equation}
    which completes the proof for~\eqref{e:PEe-decompose}.
\end{proof}

\section*{Acknowledgments}
The author thanks Gautam Iyer for his guidance and many helpful discussions.

\bibliographystyle{halpha-abbrv}
\bibliography{gautam-refs1,gautam-refs2,preprints, references} 
\end{document}